\documentclass{article}
\usepackage{amsmath,amsthm,amssymb}
\usepackage{enumerate}
\usepackage{graphicx}
\usepackage{cite}
\usepackage{comment}
\usepackage{oands}
\usepackage{tikz}
\usepackage{changepage}
\usepackage{bbm}
\usepackage{mathtools}
\usepackage[margin=1in]{geometry}
\usepackage{hyperref}
\usepackage{authblk}
\usepackage[pagewise,mathlines]{lineno}
\usepackage{appendix}
\usepackage{caption}
\usepackage{subcaption}

\theoremstyle{plain}
\newtheorem{thm}{Theorem}[section]

\newtheorem{lem}[thm]{Lemma}
\newtheorem{prop}[thm]{Proposition}

\def\@rst #1 #2other{#1}
\newcommand\MR[1]{\relax\ifhmode\unskip\spacefactor3000 \space\fi
  \MRhref{\expandafter\@rst #1 other}{#1}}
\newcommand{\MRhref}[2]{\href{http://www.ams.org/mathscinet-getitem?mr=#1}{MR#2}}

\hypersetup{
    colorlinks=false,
    linktocpage,
    }

\theoremstyle{definition}
\newtheorem{defn}[thm]{Definition}
\newtheorem{remark}[thm]{Remark}

\numberwithin{equation}{section}

\newcommand{\dsb}{\begin{adjustwidth}{2.5em}{0pt}
\begin{footnotesize}}
\newcommand{\dse}{\end{footnotesize}
\end{adjustwidth}}

\newcommand{\ssb}{\begin{adjustwidth}{2.5em}{0pt}}
\newcommand{\sse}{\end{adjustwidth}}

\newcommand{\aryb}{\begin{eqnarray*}}
\newcommand{\arye}{\end{eqnarray*}}
\def\alb#1\ale{\begin{align*}#1\end{align*}}
\newcommand{\eqb}{\begin{equation}}
\newcommand{\eqe}{\end{equation}}
\newcommand{\eqbn}{\begin{equation*}}
\newcommand{\eqen}{\end{equation*}}

\newcommand{\BB}{\mathbbm}
\newcommand{\ol}{\overline}
\newcommand{\ul}{\underline}
\newcommand{\op}{\operatorname}
\newcommand{\la}{\langle}
\newcommand{\ra}{\rangle}

\newcommand{\frk}{\mathfrak}

\newcommand{\ep}{\epsilon}
\newcommand{\rta}{\rightarrow}

\newcommand{\wt}{\widetilde}
\newcommand{\wh}{\widehat} 
\newcommand{\mcl}{\mathcal}

\newcommand{\bdy}{\partial}

\title{Asymptotic behavior of the Eden model with positively homogeneous edge weights}
\date{  }
\author{S\'{e}bastien Bubeck\footnote{sebubeck@microsoft.com}\, and Ewain Gwynne\footnote{ewain@mit.edu}}
\affil{Microsoft Research and Massachusetts Institute of Technology}
 
\begin{document}

\maketitle

\begin{abstract}
Let $d\in\mathbb N$, $\alpha\in\mathbb R$, and let $f :\mathbb R^d\setminus \{0\} \rightarrow (0,\infty)$ be locally Lipschitz and positively homogeneous of degree $\alpha$ (e.g.\ $f$ could be the $\alpha$th power of a norm on $\mathbb R^d$). We study a generalization of the Eden model on $\mathbb Z^d$ wherein the next edge added to the cluster is chosen from the set of all edges incident to the current cluster with probability proportional to the value of $f$ at the midpoint of this edge, rather than uniformly. This model is equivalent to a variant of first passage percolation where the edge passage times are independent exponential random variables with parameters given by the value of $f$ at the midpoint of the edge.  

We prove that the $f$-weighted Eden model clusters have an a.s.\ deterministic limit shape if $\alpha< 1$, which is an explicit functional of $f$ and the limit shape of the standard Eden model, and estimate the rate of convergence to this limit shape. We also prove that if $\alpha>1$, then there is a norm $\nu$ on $\mathbb R^d$ (depending on $\alpha$) such that if we set $f(z) = \nu(z)^{ \alpha}$, then the $f$-weighted Eden model clusters are a.s.\ contained in a Euclidean cone with opening angle $<\pi$ for all time. We further show that there does \emph{not} exist a norm on $\mathbb R^d$ for which this latter statement holds for all $\alpha>1$; and that there is no choice of function $f$ for which the above statement holds with $\alpha=1$. 

Our basic approach is to compare the local behavior of the $f$-weighted first passage percolation to that of unweighted first passage percolation with iid exponential edge weights (which is equivalent to the unweighted Eden model). 

We include a list of open problems and several computer simulations.
\end{abstract}
 
\tableofcontents

\section{Introduction}
\label{sec-intro}

\subsection{Overview}
\label{sec-overview}

Let $d\in \BB N$ and equip $\BB Z^d$ with its standard cubic lattice structure. The \emph{Eden model} is a simple statistical physics model introduced in~\cite{eden}, defined as follows. Let $e_1$ be sampled uniformly from the set of edges of $\BB Z^d$ incident to 0, and set $\wt A_1 = \{e_1\}$. Inductively, if $n\in\BB N$, $n\geq 2$, and $\wt A_{n-1}$ has been defined, let $e_n$ be sampled uniformly from the set of edges of $\BB Z^d$ incident to $\wt A_{n-1}$ and set $\wt A_n := \wt A_{n-1} \cup \{e_n\}$. 

The Eden model is equivalent to first passage percolation with iid exponentially distributed edge passage times, which was first introduced in~\cite{hammersley-welsh-fpp} (this is a consequence of the ``memoryless" property of the exponential distribution). Under this representation, the Eden model has been studied extensively, but many aspects of this model are still poorly understood. For example, it is known that the clusters $\wt A_n$ have a deterministic limiting shape $\BB A$ in a rather strong sense (see~\cite{richardson-fpp,cox-durrett-fpp,kesten-fpp-speed} as well as Sections~\ref{sec-D-metric} and~\ref{sec-standard-fpp} below), but little is known about this limit shape besides that it is compact, convex, and satisfies the same symmetries as $\BB Z^d$. 
We refer the reader to the survey articles~\cite{kesten-fpp-aspects,kesten-fpp-survey,douglas-fpp-models,blair-stahn-fpp-survey,grimmett-kesten-fpp-survey,ahd-fpp-survey} and the references therein for more information on first passage percolation.

In this article, we will consider the following natural variant of the Eden model. Let $\op{wt}$ be a weight function from the edge set of $\BB Z^d$ to the positive real numbers. The \emph{weighted Eden model with edge weights $\op{wt}$} is the growing family of edge sets $\{\wt A_n\}_{n\in\BB N}$ which is defined in the same manner as the Eden clusters above, except that each edge $e_n$ is sampled from the set of edges incident to $\wt A_{n-1}$ with probability proportional to $\op{wt}(e)$ instead of uniformly. Like the standard Eden model, this model can also be expressed in terms of a variant of first passage percolation where the passage time of each edge $e$ is an independent exponential random variable with parameter $\op{wt}(e)$ (in fact, we will mostly focus our attention on this latter model, which seems to be easier to analyze), see Section~\ref{sec-fpp-setup}.

We will primarily be interested in the following special case of the above model. Fix $\alpha\in\BB R$. Let $f_0 : \bdy \BB D\rta (0,\infty)$ be a strictly positive Lipschitz function on the boundary of the Euclidean unit ball $\BB D$. Let
\eqb \label{eqn-f-decomp}
f(z) = |z|^\alpha f_0(z/|z|) ,\qquad \forall z\in \BB R^d\setminus \{0\}  
\eqe
so that $f$ is strictly positive, locally Lipschitz, and homogeneous of degree $\alpha$. We call such a function $f$ an \emph{$\alpha$-weight function}. A particular example of an $\alpha$-weight function is the $\alpha$th power of some norm $\nu$ on $\BB R^d$, which corresponds to $f_0(z) = \nu(z)^\alpha$ for $z\in \bdy \BB D$. 
The \emph{$f$-weighted Eden model} is the weighted Eden model where the weight of each edge $e$ of $\BB Z^d$ is given by
\eqb \label{eqn-f-weights}
\op{wt}(e) :=  f(m_e),
\eqe 
where $m_e$ is the midpoint of $e$. 
In the case where $d = 1$ and $f(z) = |z| $, the $f$-weighted Eden model is a slight variant of the P\'olya urn model, so the $f$-weighted Eden model can be viewed as higher-dimensional generalization of the P\'olya urn model. The $f$-weighted Eden model in the case where $f(z) = |z|^\alpha$ was first introduced as an open problem in~\cite{bubeck-polya-aggregate}. 

Weighted versions of the eden model have been studied elsewhere in the literature. Diffusion limited aggregation (DLA) on a $d$-ary tree is equivalent to a weighted variant of the Eden model on the tree with edge weights which are an exponential, rather than polynomial, function of the distance to the root vertex. This model is studied in~\cite{aldous-shields-trees,barlow-dla-on-tree}. In the computer science literature, the authors of~\cite{fko-rumor} propose a weighted version of the Eden model on a general graph, which they call ``adaptive diffusion", as a protocol for spreading a message in a network while obscuring its source.  

As we shall see, the asymptotic behavior of $f$-weighted FPP in general dimension and for general choice of $f$ depends crucially on the homogeneity degree of $f$. In particular, we will prove the following.
\begin{itemize}
\item If $\alpha \in (-\infty,  1)$, the $f$-weighted FPP clusters (for any choice of weight function $f$) have a deterministic compact limit shape which is an explicit functional of $f$ and the standard Eden model limit shape $\BB A$. We also provide a rate of convergence estimate in the spirit of~\cite{kesten-fpp-speed,alexander-fpp-speed}. 
\item If $\alpha > 1$, there exists a norm $\nu$ on $\BB R^d$ depending on $\alpha$ (which we can take to be an explicit functional of $\BB A$ and $\alpha$) such that with $f(z) = \nu(z)^{ \alpha}$, the $f$-weighted FPP clusters are a.s.\ contained in a certain Euclidean cone with opening angle $<\pi$ at all times. 
\item For any choice of the Lipschitz function $f_0 : \bdy\BB D \rta (0,\infty)$ in~\eqref{eqn-f-decomp}, there is a constant $c> 0$ (again, depending explicitly on $\BB A$ and $f_0$) such that if $\alpha \in [1, 1+c)$ then a.s.\ the $f$-weighted FPP clusters with weight function $f$ eventually hit all but finitely many edges in $\BB Z^d$. 
\end{itemize}
See Section~\ref{sec-results} below for precise statements. We also include several open problems related to the weighted Eden model, see Section~\ref{sec-open-problems}. 

The main idea of our proofs is to compare the local behavior of $f$-weighted FPP to the local behavior of standard FPP. This allows us to show that passage times in $f$-weighted FPP are well-approximated by a deterministic metric $D$, which is defined precisely in Section~\ref{sec-D-metric} and depends on $f$ and the standard FPP limit shape $\BB A$. 

\begin{remark}
In the open problem statement~\cite{bubeck-polya-aggregate}, it is conjectured that for $f(z) = |z|^\alpha$, the $f$-weighted FPP clusters a.s.\ have a deterministic limit shape if $\alpha <1$ and are a.s.\ contained in a Euclidean cone of opening angle $<\pi$ at all times if $\alpha >1$. Our results confirm this conjecture in the case $\alpha < 1$. In the case $\alpha> 1$, our results show that this conjecture is false for $\alpha$ sufficiently close to 1, but is true if we replace $|\cdot| $ with a norm which is allowed to depend on $\alpha$. It is still an open problem to determine whether it holds for large enough $\alpha$ that the $f$-weighted FPP clusters with $f(z) = |z|^\alpha$ are a.s.\ contained in a Euclidean cone of opening angle $<\pi$ for all times.
\end{remark}
 
\begin{remark} \label{remark-sim}
We include several simulations of $f$-weighted FPP clusters, which are scattered throughout Section~\ref{sec-intro}. All of these simulations are produced using Matlab and are run for $10^6$ iterations. Particles are color-coded based on the time at which they are added to the cluster. In order to reduce the file size of the images, we re-sampled a subset of the $10^6$ particles in the clusters. This re-sampling does not significantly change the images, except that some of the images include small white dots corresponding to points which are contained in the cluster, but which were removed during the re-sampling. 
\end{remark} 
 
\noindent{\bf Acknowledgments}
We thank Ronen Eldan, Shirshendu Ganguly, Christopher Hoffman, Yuval Peres, and David Wilson for helpful discussions. We thank two anonymous referees for helpful comments on an earlier version of this paper. This work was carried out while the second author was an intern with the Microsoft Research theory group in Redmond, WA.

\subsection{Basic notations}
\label{sec-basic}

Before stating our main results we record some (mostly standard) notations which we will use throughout this paper. 

\subsubsection{Intervals and asymptotics}
\label{sec-asymp-notation} 
 
\noindent
For $a < b \in \BB R$, we define the discrete intervals $[a,b]_{\BB Z} := [a, b]\cap \BB Z$ and $(a,b)_{\BB Z} := (a,b)\cap \BB Z$. 
\vspace{6pt}

\noindent
If $a$ and $b$ are two quantities, we write $a\preceq b$ (resp. $a \succeq b$) if there is a constant $C$ (independent of the parameters of interest) such that $a \leq C b$ (resp. $a \geq C b$). We write $a \asymp b$ if $a\preceq b$ and $a \succeq b$. 
\vspace{6pt}

\noindent
If $a$ and $b$ are two quantities which depend on a parameter $x$, we write $a = o_x(b)$ (resp. $a = O_x(b)$) if $a/b \rta 0$ (resp. $a/b$ remains bounded) as $x \rta 0$ (or as $x\rta\infty$, depending on context). We write $a = o_x^\infty(b)$ if $a = o_x(b^{-s})$ for each $s >0$. 
\vspace{6pt}

\noindent
Unless otherwise stated, all implicit constants in $\asymp, \preceq$, and $\succeq$ and $O_x(\cdot)$ and $o_x(\cdot)$ errors involved in the proof of a result are required to satisfy the same dependencies as described in the statement of said result.

\subsubsection{Graphs}
\label{sec-graph-notation}

\noindent
For a graph $G$, we write $\mcl V(G)$ for the set of vertices of $G$ and $\mcl E(G)$ for the set of edges of $G$.
\vspace{6pt}

\noindent
For a graph $G$ and a subset $E$ of $\mcl E(G)$ we write $\partial E$ for the set of edges of $G$ not contained in $E$ which are incident to an edge of $E$. For a subset $V$ of $\mcl V(G)$, we write $\partial V$ for the set of vertices $v\in\mcl V(G)$ which are incident to vertices of $G$ not contained in $V$. 
\vspace{6pt}

\noindent
Let $G$ be a graph and let $n\in \BB N \cup \{\infty\}$. A \emph{path} of length $n$ in $G$ is a sequence $\eta = \{\eta(i)\}_{i\in [1,n]_{\BB Z}} \subset \mcl E(G) $ such that the edges $\eta(i)$ can be oriented in such a way that the initial endpoint of $\eta(i)$ coincides with the terminal endpoint of $\eta(i-1)$ for each $i\in [2,n]_{\BB Z}$. We say that $\eta$ is \emph{simple} if $\eta$ does not visit any vertex of $G$ more than once.
We write $|\eta| = n$ for the length of $\eta$. 

\subsubsection{Metrics}
\label{sec-metric-notation}

We will have occasion to consider several different metrics on $\BB R^d$ and $\BB Z^d$. We use the following notation to distinguish these metrics.
\vspace{6pt}

\noindent
Let $D$ be a metric on $\BB R^d$. For $r > 0$ and $z\in \BB R^d$, we write $B_r^D(z)$ for the closed ball of radius $r$ centered at $z$ in the metric $D$. For a set $A\subset \BB R^d$, we write $\op{diam}^D(A)$ for the $D$-diameter of $A$. If $\nu$ is a norm on $\BB R^d$, we write $\op{dist}^\nu(z,w) = \nu(z-w)$ for the metric induced by $\nu$. We often abbreviate $B^{\op{dist}^\nu}_r(z) = B^\nu(z)$. 
\vspace{6pt}

\noindent
We write $|\cdot|$ for the Euclidean norm on $\BB R^d$ and $\BB D := B_1^{|\cdot|}(0)$ for its unit ball.

\subsection{Weighted first passage percolation model}
\label{sec-fpp-setup}

In most of this paper we will consider the following weighted variant of first passage percolation instead of the weighted Eden model described above. The two models are shown to be equivalent in Lemma~\ref{prop-fpp-equiv} below. We first define the model in the greatest possible generality, then describe the special case which is our primary interest.

\begin{defn} \label{def-fpp}
Let $G$ be a connected, countable graph in which all vertices have finite degree. Let $v_0$ be a marked vertex of $G$. Let $\op{wt} : \mcl E(G) \rta (0,\infty)$ be a deterministic function which assigns a positive weight to each $e\in\mcl E(G)$. The \emph{first passage percolation (FPP)} clusters on $G$ started from $v_0$ with weights $\op{wt}$ is the random increasing sequence of subgraphs $\{A_t\}_{t\geq 0}$ of $G$ defined as follows.
\begin{itemize}
\item For each edge $e \in \mcl E(G)$, let $X_e$ be an exponential random variable with parameter $\op{wt}(e)$. We take the $X_e$'s to be independent. 
\item For a path $\eta$ in $G$, let $T(\eta) := \sum_{e\in \eta} X_e$. For vertices $u,v \in  \mcl V(G)$, we write 
\eqbn
T(u, v) := \inf\left\{T(\eta) \,:\, \text{$\eta$ is a path in $G$ from $u$ to $v$}\right\} .
\eqen 
\item For $t\in [0,\infty)$, let $  A_t \subset G$ be the graph defined as follows. The set of vertices $\mcl V(A_t)$ is the set of $v \in\mcl V(G)$ with $T(v_0 , v) \leq t$. The set of edges $\mcl E(A_t)$ is the set of $e\in \mcl E(G)$ such that $e\in \eta$ for some path $\eta$ in $G$ with $\eta(1)$ incident to $v_0$ and $T(\eta) \leq t$.
\end{itemize}
For $t\geq 0$ we write $\mcl F_t$ for the $\sigma$-algebra generated by $\{A_s\}_{s\in [0,t]}$ and $X_e$ for $e\in \mcl E(A_t)$. We also let 
\eqbn
\tau_\infty := \inf\{t  > 0 \,:\, \# A_t = \infty\} = T(v_0 , \infty)
\eqen
be the first (possibly infinite) time at which the cluster is infinite.
\end{defn}

Note that ordinary first passage percolation with exponential passage times corresponds to the special case when $\op{wt}(e) = 1$ for each $e\in \mcl E(G)$ in Definition~\ref{def-fpp}. 

We are primarily interested in the following special case of the model of Definition~\ref{def-fpp}, which is a continuous-time parametrization of the $f$-weighted FPP model described in Section~\ref{sec-overview} (see Lemma~\ref{prop-fpp-equiv} below). 
Fix $\alpha\in\BB R$. Let $G = \BB Z^d$ for $d \in \BB N$ (with the standard cubic lattice structure) and let $v_0 = 0$. 
Let $f_0 : \bdy \BB D\rta (0,\infty)$ be a Lipschitz function and let $f(z) = |z|^\alpha f_0(z/|z|)$ be as in~\eqref{eqn-f-decomp} and $\op{wt}(e) = f(m_e)$ as in~\eqref{eqn-f-weights}.  
Let $\{A_t\}_{t\geq 0}$, $T(\cdot)$, and $\tau_\infty$ be as in Definition~\ref{def-fpp} with this choice of parameters. 
We call the above model \emph{$f$-weighted FPP}. 
We also introduce the notation
\eqb \label{eqn-f-bound}
\ol\kappa := \sup_{z\in \bdy\BB D} f_0(z) \quad \op{and}\quad \ul\kappa :=  \inf_{z\in \bdy\BB D} f_0(z) .
\eqe

We note that it is easy to see (by considering a path from 0 to $\infty$ along a coordinate axis) that for our model $\tau_\infty < \infty$ a.s.\ whenever $\alpha > 1$. It will follow from Theorem~\ref{thm-alpha<1} (resp.\ the proof of Theorem~\ref{thm-alpha-near-1}) below that a.s.\ $\tau_\infty = \infty$ whenever $\alpha < 1$ (resp.\ $\alpha =1$). 
 
\subsection{Standard FPP limiting shape and weighted metric}
\label{sec-D-metric}

Our main method for studying the model described in Section~\ref{sec-fpp-setup} is to compare it to standard FPP, i.e.\ the case where $f\equiv 1$, which is equivalent to the unweighted Eden model. In this case, it is shown in~\cite{richardson-fpp,cox-durrett-fpp} that there exists a compact convex set $\BB A    \subset \BB R^d$ which is symmetric about 0 such that the random sets $t^{-1} A_t $ converge a.s.\ as $t\rta \infty$ to $\BB A $ in the following sense. For $t > 0$, let 
\eqb \label{eqn-fatten-hull}
A_t^F := \left\{v + z \,:\, v \in \mcl V(A_t)  ,\, z \in [-1/2 , 1/2]^d \right\} 
\eqe 
be the ``fattening" of $A_t$, so that $A_t^F$ contains no isolated points and $A_t^F \cap \BB Z^d = \mcl V(A_t)$. Then for each $\ep > 0$, 
\eqb \label{eqn-cluster-conv}
\lim_{t\rta \infty} \BB P\left( (1-\ep) \BB A \subset s^{-1} A_s^F \subset (1+\ep) \BB A ,\quad \forall s \geq t \right) = 1. 
\eqe
Not much is known rigorously about the limit shape $\BB A$ besides that it is compact, convex, and has the same symmetries as $\BB Z^d$. 
It is expected that $\BB A$ is not the Euclidean unit ball, but even this is not known except in dimension $d\geq 35$~\cite{ceg-short-path}. 
See, e.g.,~\cite{fss-eden-sim,bh-eden-sim,alm-deijfen-fpp-sim} for numerical studies of Eden clusters. 
 
Let $\mu$ be the norm whose closed unit ball is $\BB A$, i.e.\
\eqb \label{eqn-convex-norm}
\mu(z) := \inf\left\{r > 0 \,:\, z \in r \BB A \right\} ,\qquad \forall z\in \BB R^d .
\eqe  
We will have occasion to compare $\BB A$ to the Euclidean unit ball. For this purpose we use the following notation.

\begin{defn} \label{def-angle-constant}
Let
\eqb \label{eqn-optimal-norm}
\ol \rho := \sup_{z \in \BB R^d\setminus \{0\} } \frac{|z|}{\mu(z)} \quad \op{and} \quad \ul \rho  := \inf_{z  \in \BB R^d\setminus \{0\} } \frac{|z|}{\mu(z)}  .
\eqe  
Also let
\eqb \label{eqn-optimal-set}
\BB X := \left\{x\in\bdy\BB A \,:\, \rho_x = \ol\rho \right\}  
\eqe  
be the set of points on $\bdy \BB A$ furthest from 0.
\end{defn}

In the remainder of this subsection, we will define a metric $D = ``f^{-1} \cdot \mu "$ on $\BB R^d\setminus \{0\}$ which will turn out to be a good approximation for passage times in our weighted FPP model.
 
\begin{defn} \label{def-pl-path}
A \emph{piecewise linear path} in $\BB R^d$ is a continuous map $\gamma : [0,T] \rta \BB R^d$ for some $T> 0$ for which there exists a subdivision $0  = t_0 < \dots < t_n = T$ of $[0,T]$ such that $\gamma|_{[t_{k-1} , t_k]}$ is affine for each $k\in [1,n]_{\BB Z}$. We say that $\gamma$ is \emph{parametrized by $\mu$-length} if the following is true. For $t\in [0,T]$, let $K_t$ be the largest $k\in [1,n]_{\BB Z}$ with $t_k \leq t$. Then
\[
t = \mu(\gamma(t) - \gamma(t_{K_t})) +  \sum_{k=1}^{K_t} \mu(\gamma(t_k) - \gamma(t_{k-1})),
\]
i.e. $t$ is the sum of the $\mu$-lengths of the linear segments of $\gamma$ traced up to time $t$. In this case we write $\op{len}^\mu(\gamma) = T$. 
\end{defn}

If $\gamma : [0,T] \rta \BB R^d $ is a piecewise linear path parametrized by $\mu$-length, we define the \emph{$D$-length} of $\gamma$ by
\eqb \label{eqn-len-bar-def}
\op{len}^D(\gamma) := \int_0^T f(\gamma(t))^{-1} \, dt ,
\eqe  
with $f$ the $\alpha$-weight function from~\eqref{eqn-f-decomp}. If $\gamma$ is not necessarily parametrized by $\mu$-length, we define the $D$-length of $\gamma$ to be the $D$-length of the path obtained by parametrizing $\gamma$ by $\mu$-length. 
We define a metric on $\BB R^d$ by
\eqb \label{eqn-weighted-metric}
D(z,w) := \inf_\gamma \op{len}^D(\gamma) \qquad \forall  z,w\in\BB R^d   
\eqe 
where the infimum is over all piecewise linear paths $\gamma$ connecting $z$ and $w$. 

As we shall see in Section~\ref{sec-fpp-estimates} below, $D(z,w)$ is a good approximation for the passage time $T(z,w)$ in the $f$-weighted FPP process $\{A_t\}_{t\geq 0}$. 
The following lemma is immediate from the $\alpha$-homogeneity of $f$ and the definition~\eqref{eqn-weighted-metric} of $D$. 

\begin{lem} \label{prop-metric-scale}
Let $z,w\in\BB R^d$ and $r > 0$. Then 
\eqb \label{eqn-metric-scale}
D(rz,rw) = r^{1-\alpha} D(z,w) .
\eqe 
\end{lem}

\subsection{Main results}
\label{sec-results}

Throughout this section, we assume that we are in the special case of Definition~\ref{def-fpp} described in Section~\ref{sec-fpp-setup}, so in particular $\alpha \in \BB R$, $f$ is an $\alpha$-weight function as in~\eqref{eqn-f-decomp}, and $\{A_t\}_{t\geq 0}$ are the $f$-weighted FPP clusters. 

Let $D$ be the metric from Section~\ref{sec-D-metric}. If $\alpha < 1$, then it is easy to see by integration that $\lim_{w\rta 0} D(w,z)$ is finite for each $z\in \BB R^d\setminus \{0\}$ and that $D$ extends to a metric on all of $\BB R^d$. In particular, the $D$-balls $B^D_r(0)$ for $r > 0$ are well-defined. Let $\BB B = \BB B_f := B_1^D(0)$. We note that Lemma~\ref{prop-metric-scale} implies that
\eqb \label{eqn-metric-ball-scale}
B_r^D(0) = r^{\frac{1}{1-\alpha}} \BB B ,\quad \forall r > 0 .
\eqe 
The set $\BB B$ is the limiting shape of the $f$-weighted FPP clusters for $\alpha <1$, in the following sense.
 
\begin{thm} \label{thm-alpha<1}
Let $\alpha  \in (-\infty, 1)$ and
\eqb \label{eqn-rate-exponent}
\chi   \in \left( 0,  \frac{1}{3(1-\alpha)} \right) .
\eqe  
For $t>0$, let $A_t^F$ be as in~\eqref{eqn-fatten-hull} (for a general choice of $f_0$). Then for $t_0 > 0$,  
\eqbn
\BB P\left( \text{$\left(1 - t^{-\chi}\right)t^{\frac{1}{1-\alpha}} \BB B \subset A_t^F \subset \left(1 + t^{-\chi }\right)t^{\frac{1}{1-\alpha}}  \BB B$ for all $t\geq t_0$} \right)  = 1-o_{t_0}^\infty(t_0) ,
\eqen
where here $o_{t_0}^\infty(t_0)$ denotes a quantity which decays faster than any negative power of $t_0$ as $t_0\rta\infty$ (recall Section~\ref{sec-asymp-notation}). 
\end{thm}

Theorem~\ref{thm-alpha<1} gives in some sense a complete qualitative characterization of the asymptotic behavior of the $f$-weighted FPP clusters when $\alpha <1$. However, we expect that the exponent $\chi$ in~\eqref{eqn-rate-exponent} is not optimal (in fact, we expect the theorem to be true at least for any $\chi \in \left(0 , \frac{2}{5(1-\alpha)} \right)$; c.f.\ Remark~\ref{remark-kesten-optimal} below). Moreover, we cannot give a more explicit description of the limit shape $\BB B$ than the one above. Indeed, we cannot even characterize the functions $f$ for which the set $\BB B$ is convex. See Figures~\ref{fig-sim1} and~\ref{fig-sim2} for simulations of $f$-weighted FPP clusters with $\alpha<1$, some of which appear to have a non-convex limit shape. 

\begin{figure}[ht!]
\centering
\begin{subfigure}{.5\textwidth}
  \centering
  \includegraphics[width=1\linewidth]{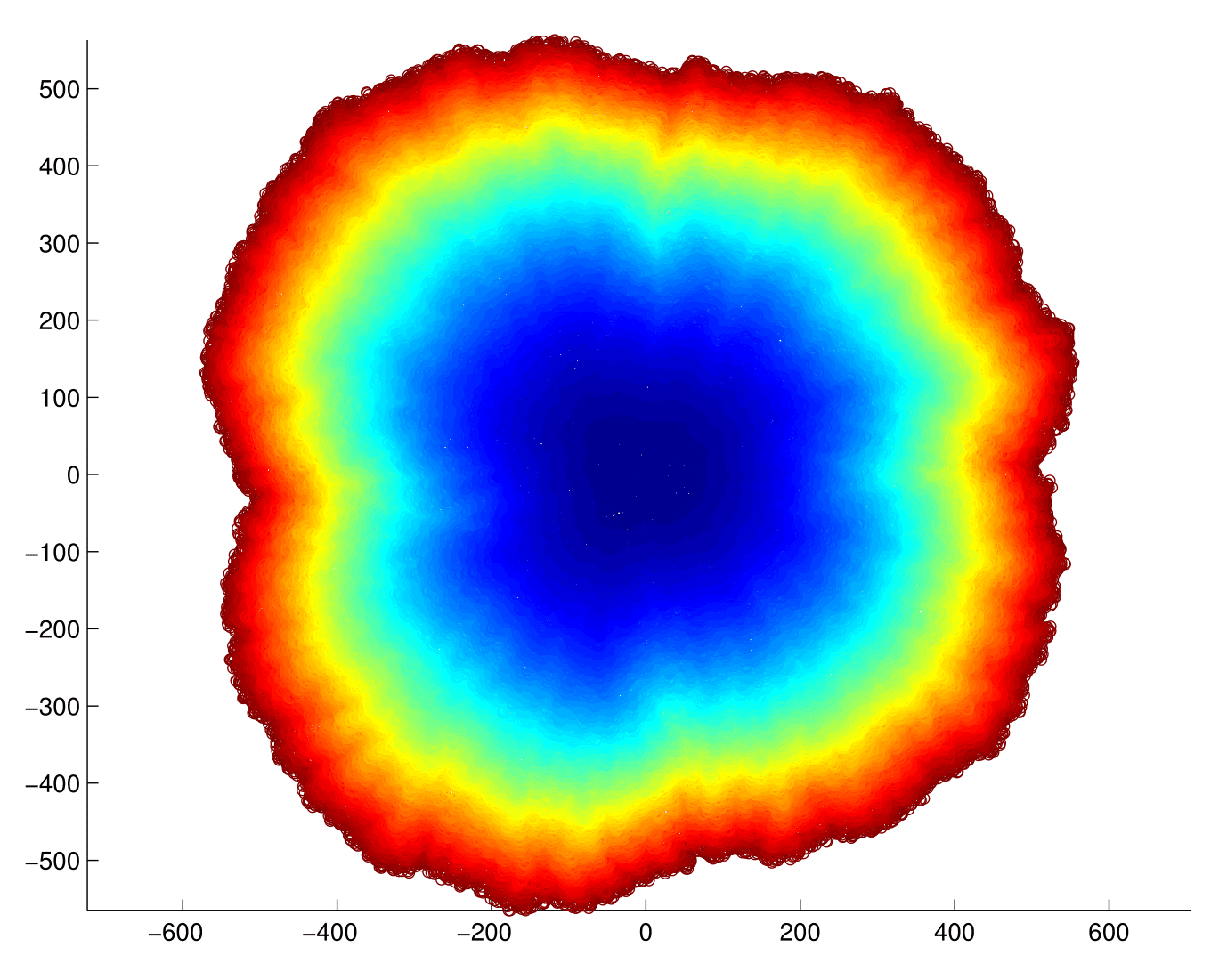}  
\end{subfigure}%
\begin{subfigure}{.5\textwidth}
  \centering
  \includegraphics[width=1\linewidth]{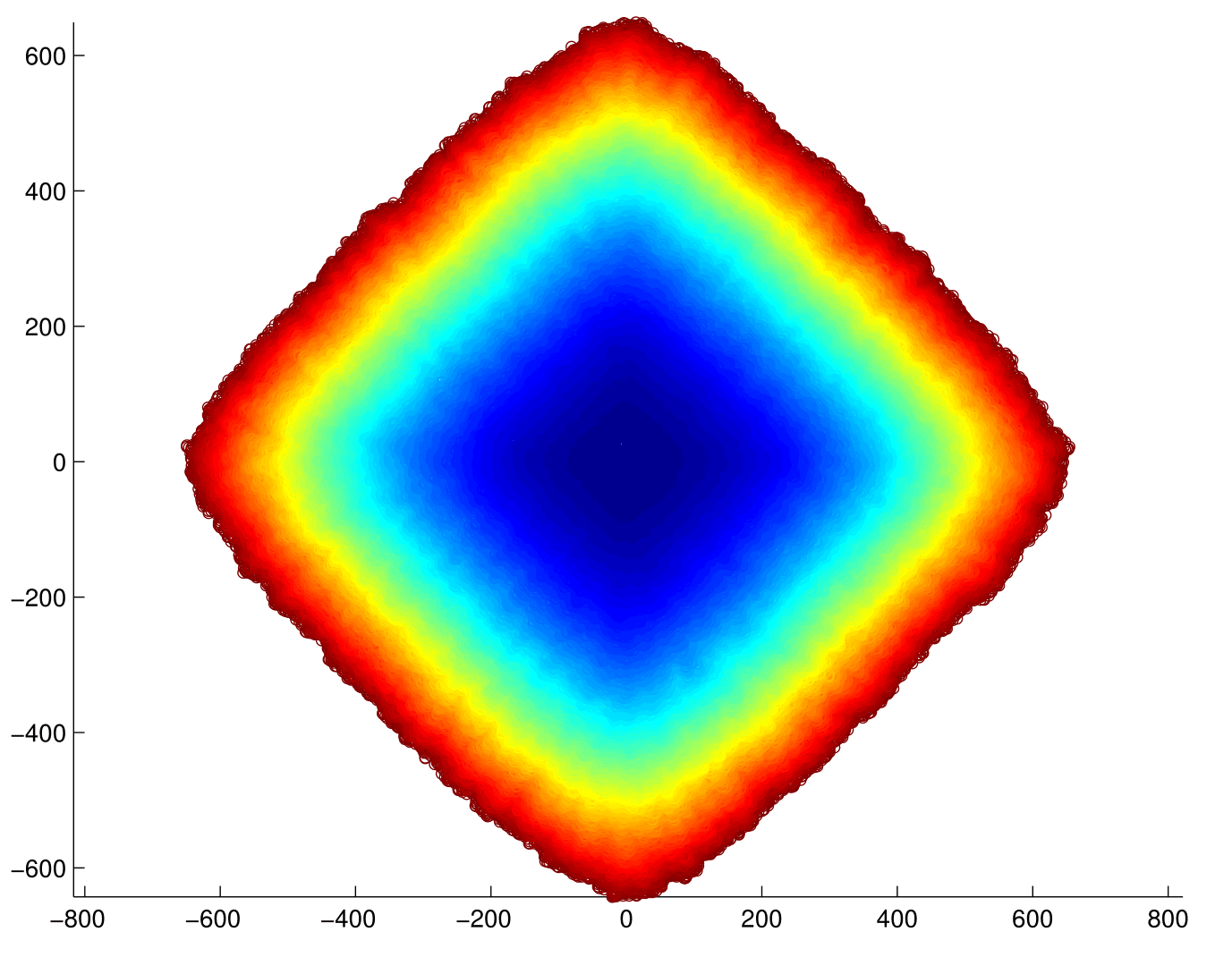}
\end{subfigure} 
\caption{\textbf{Left panel:} A simulation of an $f$-weighted FPP cluster with $f(z) = \|z\|_1^{1/2}$, where here $\|\cdot\|_1$ is the $L^1$ norm (which restricts to the graph distance on $\BB Z^2$). The clusters appear to be converging to a deterministic limit shape (which we know is a.s.\ the case by Theorem~\ref{thm-alpha<1}), but it is not clear from the simulation whether this limit shape is convex.
\textbf{Right panel:} A simulation of an $f$-weighted FPP cluster with $f(z) = \|z\|_1^{-2}$, where here $\|\cdot\|_1$ is the $L^1$ norm. The clusters appear to be converging to a deterministic limit shape which is a slight rounding of the $L^1$-unit ball.
}\label{fig-sim1} 
\end{figure}

\begin{figure}[ht!]
\centering
\begin{subfigure}{.5\textwidth}
  \centering
  \includegraphics[width=1\linewidth]{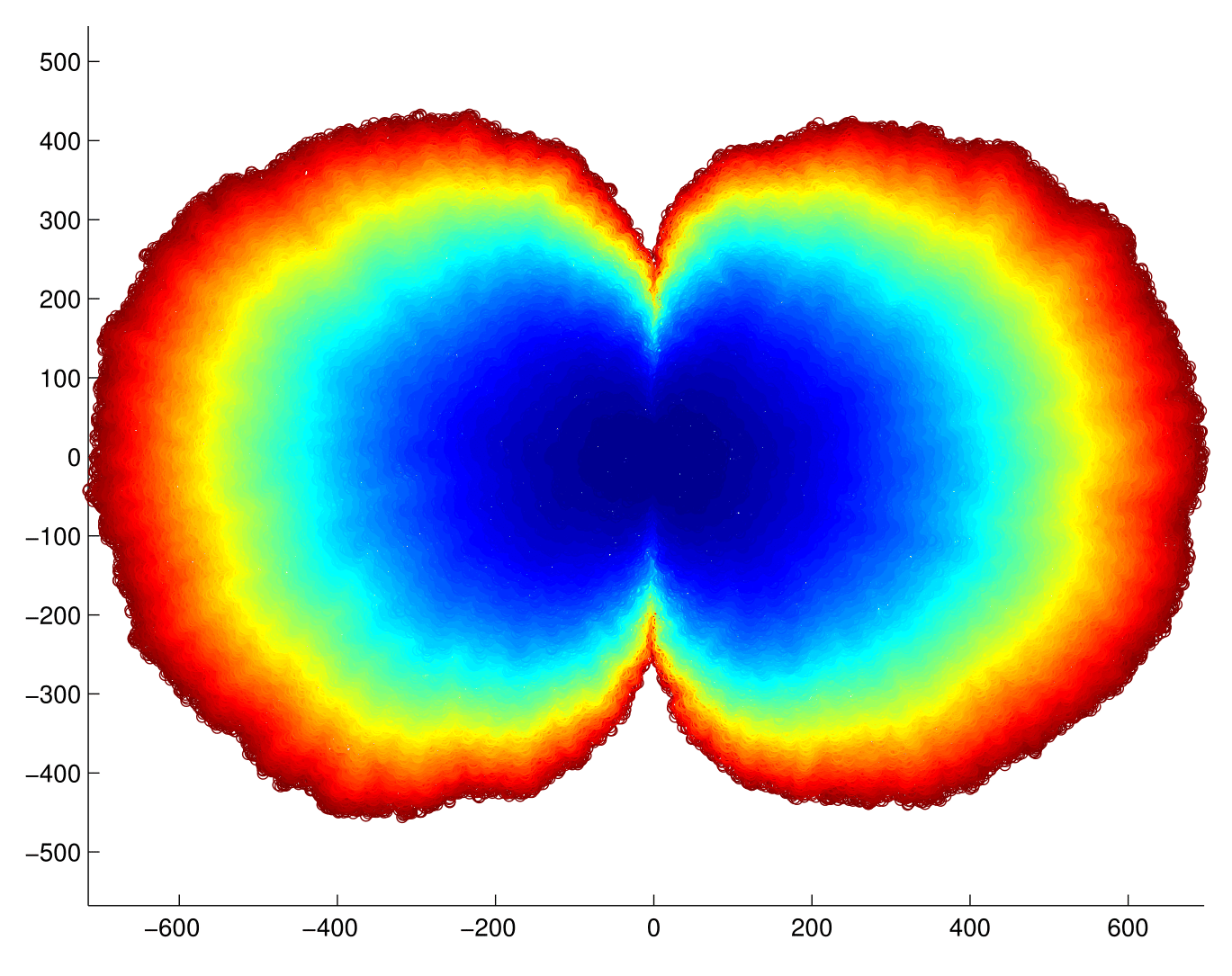} 
\end{subfigure}%
\begin{subfigure}{.5\textwidth}
  \centering
  \includegraphics[width=1\linewidth]{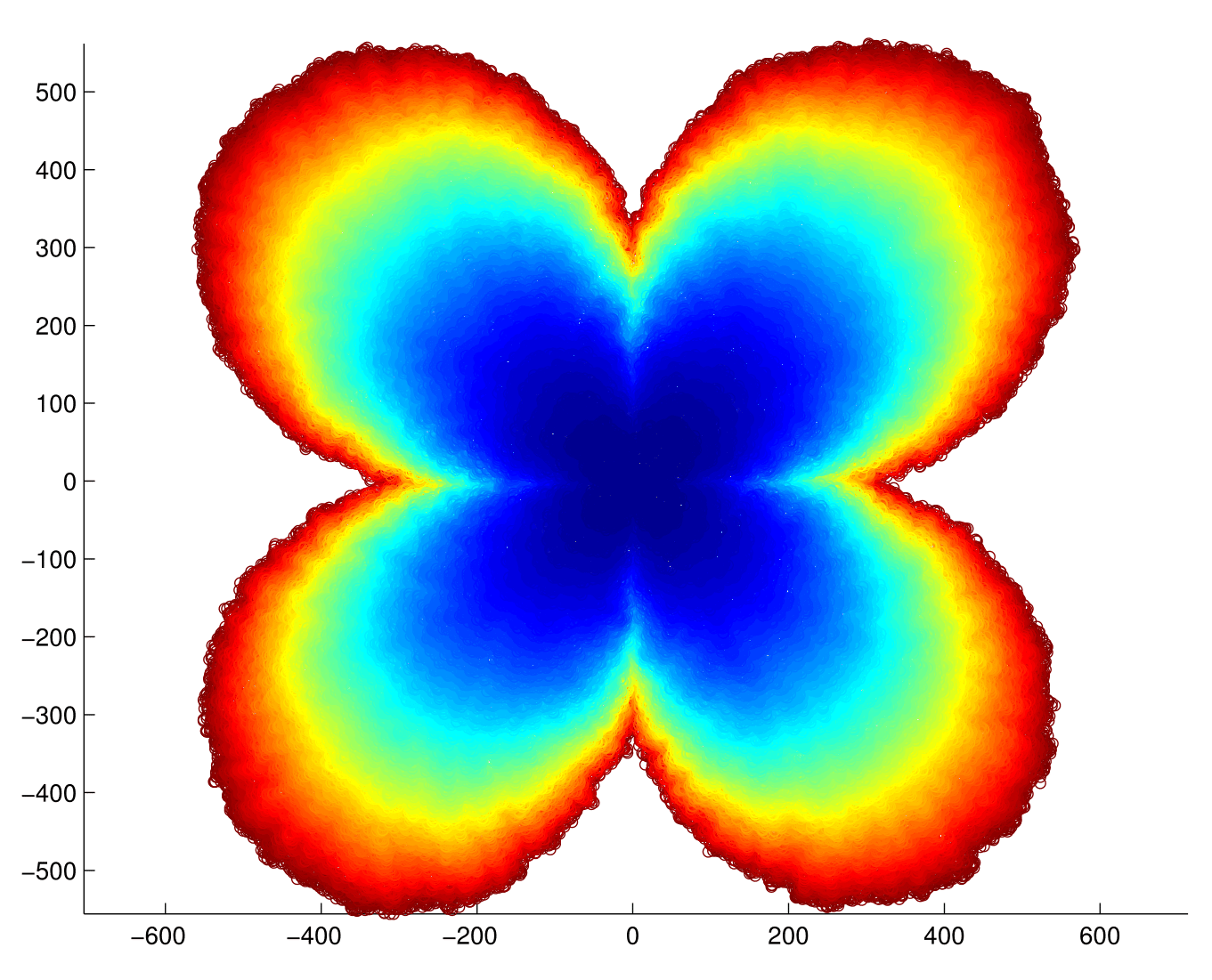} 
\end{subfigure} 
\caption{\textbf{Left panel:} A simulation of an $f$-weighted FPP cluster with $f(z) = \nu (z)^{1/2}$, where here $\nu$ is the norm whose closed unit ball is the rectangle $[-1,1]\times [-100 , 100]$.
\textbf{Right panel:} A simulation of an $f$-weighted FPP cluster with the weight function $f$ given by the third power of the ratio of the $L^1$ norm to the Euclidean norm (so $\alpha = 0$). We note that in both figures, the limit shape appears to be non-convex.
}\label{fig-sim2} 
\end{figure}

In the case $\alpha\geq 1$, matters are more complicated. The qualitative asymptotic behavior of the $f$-weighted FPP clusters depends crucially on the function $f$, rather than just the value of $\alpha$. 
In the case when $\alpha > 1$, simulations like the ones in Figure~\ref{fig-sim3} suggest that the $f$-weighted FPP clusters for many choices of $f$ tend to grow in a single direction, rather than being ball-like like in the case when $\alpha < 1$. 
We recall that $\tau_\infty = \inf\{t\geq 0 : \# \mcl E(A_t) = \infty\}$. 
Our next theorem tells us that for each $\alpha>1$, there exists a norm on $\BB R^d$ (depending on $\alpha$) such that if $f$ is the $\alpha$th power of this norm, then $A_{\tau_\infty}$ is a.s.\ contained in a cone of opening angle $<\pi$. 

\begin{thm} \label{thm-alpha>1}
For each $\alpha >1$ and each $\BB x \in \BB X$ (Definition~\ref{def-angle-constant}), there exists a norm $\nu = \nu(\alpha,\BB x)$ on $\BB R^d$ and a
$\theta \in (0,\pi)$ such that the following is true. Let
\eqbn
\mcl C= \left\{ z\in \BB R^d \,:\, \left\la \frac{z}{|z|} , \frac{\BB x}{|\BB x|} \right\ra > \cos \theta \right\}
\eqen
be the Euclidean cone based at 0 with opening angle $\theta$ centered at the ray from 0 through $\BB x$. Also let $f(z) := \nu(z)^{ \alpha}$ and let $\{A_t\}_{t\geq 0}$ the $f$-weighted FPP process. Then a.s.\ either
\eqb \label{eqn-alpha>1}
 \# \left(\mcl V(A_{\tau_\infty}) \setminus \mcl C\right) < \infty \quad \op{or} \quad  \# \left(\mcl V(A_{\tau_\infty}) \setminus (-\mcl C)\right) < \infty .
\eqe 
\end{thm}

We will actually prove a more quantitative version of Theorem~\ref{thm-alpha>1} (see Theorem~\ref{thm-cone-contain} below). This result says that the statement of Theorem~\ref{thm-alpha>1} holds for all $\alpha$-weight functions $f$ satisfying certain conditions, which are satisfied for the $\alpha$-th powers of a certain class of norms on $\BB R^d$. The unit ball of a typical norm in this class is a ``cylinder" of the form $\{s z + t\BB x : z\in Q ,\, t\in [-1,1]\}$ where $Q$ is a compact convex subset of the hyperplane through the origin perpendicular to $\BB x$ and $s$ is a large fixed parameter which tends to $\infty$ as $\alpha\rta 1^+$. See Figure~\ref{fig-big-alpha-norm} for an illustration.

\begin{figure}[ht!]
\centering
\begin{subfigure}{.5\textwidth}
  \centering
  \includegraphics[width=1\linewidth]{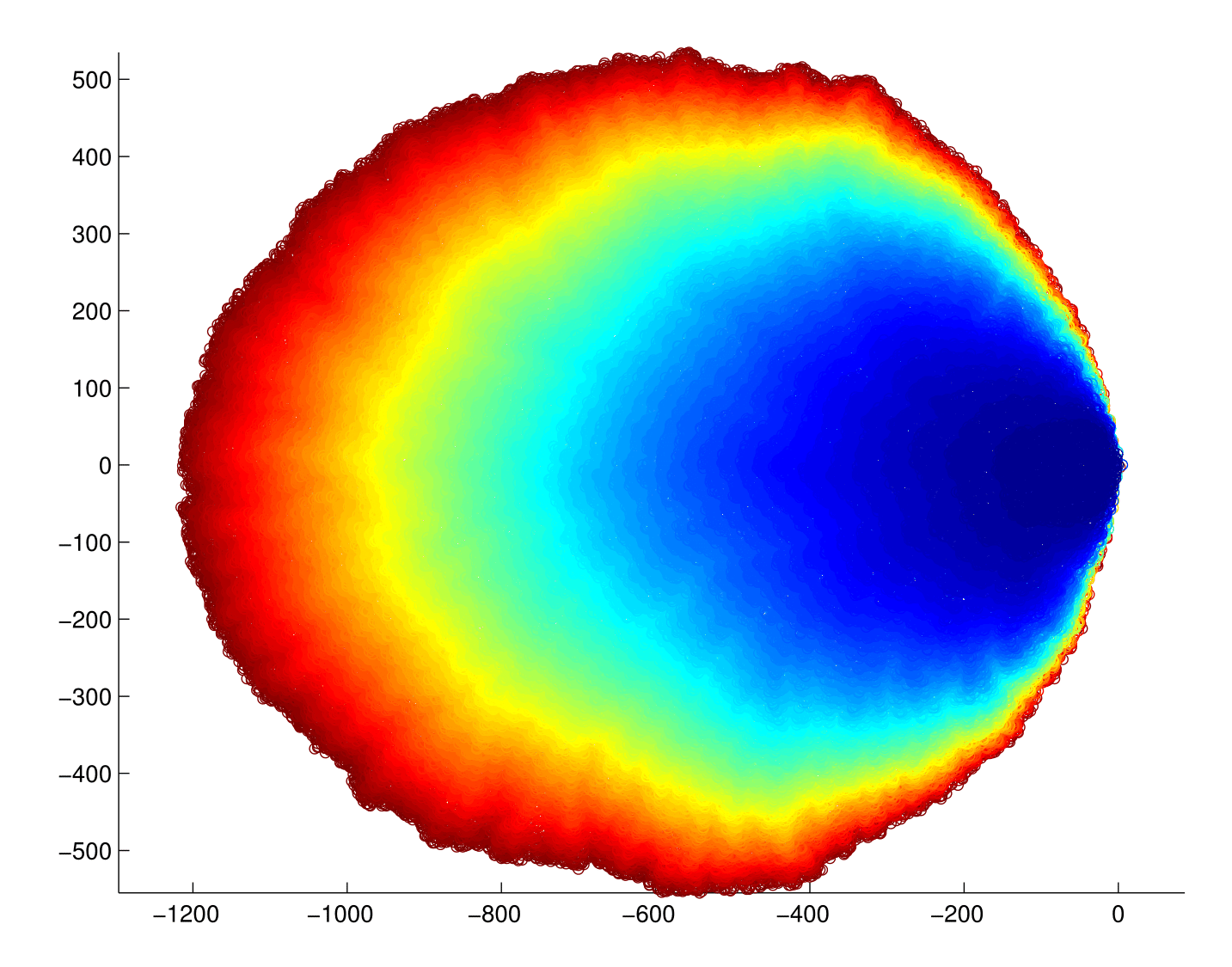} 
\end{subfigure}%
\begin{subfigure}{.5\textwidth}
  \centering
  \includegraphics[width=1\linewidth]{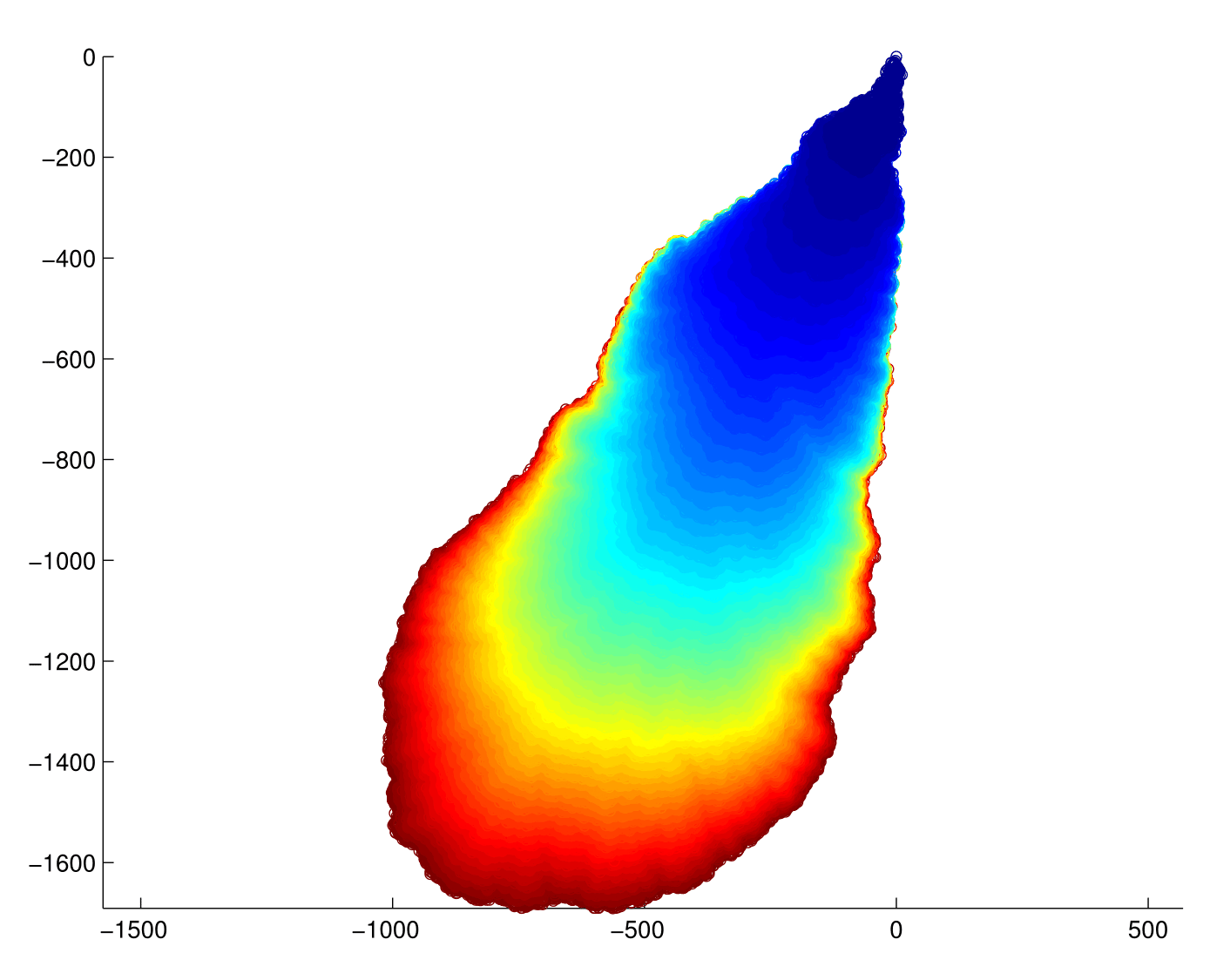} 
\end{subfigure} 
\caption{\textbf{Left panel:} A simulation of an $f$-weighted FPP cluster with $f(z) = \nu(z)^{1.1} $, where here $\nu$ is the norm whose closed unit ball is the rectangle $[-1,1]\times [-100 , 100]$. This norm $\nu$ is similar to the norm appearing in Theorem~\ref{thm-alpha>1}, although in Theorem~\ref{thm-alpha>1} the rectangle may be rotated by some (non-explicit) angle which depends on the standard FPP limit shape $\BB A$.
\textbf{Right panel:} A simulation of an $f$-weighted FPP cluster with $f(z) = |z|^5$. The figure suggests that the clusters will be contained in a Euclidean cone with opening angle $<\pi$ for all times $t \leq \tau_\infty$, but we do not prove that this is the case for this particular choice of $f$.
}\label{fig-sim3} 
\end{figure}

\begin{figure}[ht!]
 \begin{center}
\includegraphics[scale=.7]{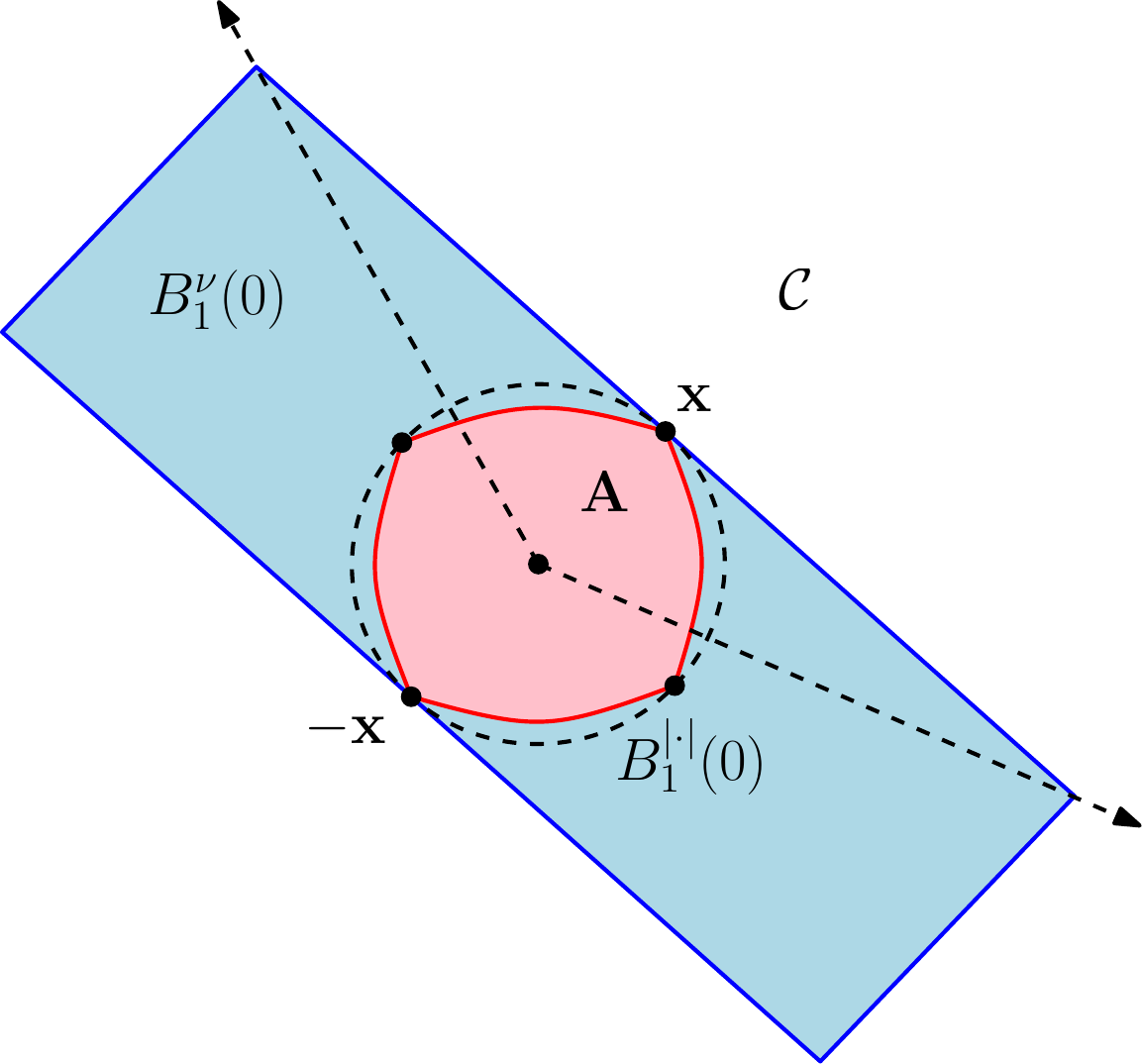} 
\caption{An illustration of a the unit ball $B_1^\nu(0)$ (light blue) of a typical norm $\nu$ satisfying the conclusion of Theorem~\ref{thm-alpha>1} when $\alpha$ slightly bigger than 1 and $d=2$. Also shown is the Eden model limit shape $\BB A$ (pink) and the smallest Euclidean ball which contains it, namely $B_{\ol\rho}^{|\cdot|}(0)$ (dashed boundary). The boundary of the cone $\mcl C$ is shown as a pair of dashed lines. As $\alpha$ approaches 1, the opening angle of this cone approaches $\pi$. However, we do not prove that the opening angle of $\mcl C$ approaches 0 as $\alpha \rta \infty$.}\label{fig-big-alpha-norm}
\end{center}
\end{figure}

It is an open problem to give for each $\alpha>1$ a reasonably (though perhaps not fully) general characterization of the choices of $f$ for which the conclusion of Theorem~\ref{thm-alpha>1} holds. We expect that a rigorous proof of such a characterization may require additional knowledge about the standard FPP limit shape $\BB A$. 
 
Theorem~\ref{thm-alpha>1} focuses on the behavior of the FPP clusters up to time $\tau_\infty$, which is a.s.\ finite for $\alpha >1$. It is natural to ask about the behavior of the clusters $A_t$ for $t  >\tau_\infty$. Straightforward tail estimates for sums of exponential random variables (see, e.g.~\cite[Theorem 5.1, item (i)]{janson-tail}) show that if $\alpha >1$, then it is a.s.\ the case that for each $\ep  > 0$, the set $\mcl V(A_{\tau_\infty +\ep})$ contains all but finitely many vertices of $\BB Z^d$. Hence there is no interesting macroscopic behavior after time $\tau_\infty$.

One may wonder to what extent the norm $\nu$ and the cone $\mcl C$ in Theorem~\ref{thm-alpha>1} can taken to be uniform in $\alpha$. It turns out that the condition on $\nu$ needed for~\eqref{eqn-alpha>1} to hold a.s.\ differs from the condition needed for this result to hold with positive probability. In particular, our more quantitative statement Theorem~\ref{thm-cone-contain} implies the following. 
\begin{itemize}
\item For any $\alpha_2 > \alpha_1 > 1$, we can choose $\nu$ and $\mcl C$ such that whenever $\alpha \in [\alpha_1,\alpha_2]$ and $f(z) = \nu(z)^\alpha$, the condition~\eqref{eqn-alpha>1} holds a.s.
\item For any $\alpha_1 > 1$, we can choose $\nu$ and $\mcl C$ such that whenever $\alpha \geq \alpha_1$, we have that~\eqref{eqn-alpha>1} holds with positive probability.
\end{itemize} 
We note that Theorem~\ref{thm-alpha-near-1} below tells us that $\nu$ cannot be chosen uniformly for all $\alpha> 1$.

Our next theorem tells us that there is no choice of the function $f_0$ of~\eqref{eqn-f-decomp} for which the conclusion of Theorem~\ref{thm-alpha>1} holds for every choice of $\alpha > 1$.
In fact, we will show that if $\alpha \geq 1$ is sufficiently close to 1 (depending on $f_0$), then $\mcl V(A_{\tau_\infty})$ a.s.\ contains all but finitely many vertices of $\BB Z^d$. To quantify how close to 1 we need $\alpha$ to be, we introduce some notation.  
For $\delta > 0$ and $z,w\in\bdy\BB D$, let $\Gamma_\delta(z,w)$ be the set of piecewise linear paths (Definition~\ref{def-pl-path}) connecting $z$ and $w$ which can be decomposed into linear segments whose endpoints are all contained in $\bdy\BB D$ and which each have Euclidean length at most $\delta$. Let
\eqb \label{eqn-ball-norm-diam}
\lambda  := \limsup_{\delta\rta 0} \sup_{z,w\in\bdy\BB D} \inf_{\gamma \in \Gamma_\delta(z,w)}  \op{len}^D(\gamma) 
\eqe 
be half the $D$-circumference of $\bdy\BB D$.
Since $f \equiv f_0$ on $\bdy \BB D$, it is easy to see that $\lambda$ depends only on $f_0$, not on $\alpha$, and that $0 < \lambda < \infty$ for any choice of $f_0$. Furthermore, if we take $f (z) = \nu(z)^\alpha$ for some norm $\nu$ on $\BB R^d$, then $\lambda$ depends on $\alpha$ but is uniformly positive for $\alpha$ in any bounded subset of $\BB R$. 
 
\begin{thm} \label{thm-alpha-near-1}
Let $f_0 : \bdy\BB D\rta (0,\infty)$ be the Lipschitz function in~\eqref{eqn-f-decomp}. 
Let $\ol\rho$ be as in~\eqref{eqn-optimal-norm}, $\ol\kappa$ as in~\eqref{eqn-f-bound}, and $\lambda$ as in~\eqref{eqn-ball-norm-diam}. Suppose
\eqbn
1\leq \alpha < 1+(\ol\rho \ol\kappa \lambda)^{-1} .
\eqen
For $r > 0$, let
\eqb \label{eqn-nu-ball-hit-time}
\sigma_r := \inf\left\{t \geq 0 \,:\, A_t\not\subset B_r^{|\cdot|}(0)\right\} .
\eqe 
There is a constant $R > 1$, depending only on $\mu$ and $f$, such that
\eqb \label{eqn-nu-annulus-swallow}
\BB P\left( \BB Z^d \cap \left(  B_{n}^{|\cdot|}(0) \setminus B_{n-1}^{|\cdot|}(0) \right) \subset \mcl V\left( A_{\sigma_{R n}}  \right) \right) = 1-o_n^\infty(n) .
\eqe  
In particular, a.s.\ $\BB Z^d\setminus \mcl V(A_{\tau_\infty})$ is a finite set.
\end{thm}

\begin{remark}
In the case when $\alpha = 1$, it will be clear from the proof of Theorem~\ref{thm-alpha-near-1} that a.s.\ $\tau_\infty = \infty$, so $\mcl V(A_{\tau_\infty}) = \BB Z^d$. 
\end{remark}

See Figure~\ref{fig-sim4} for simulations of $f$-weighted FPP clusters in the setting of Theorem~\ref{thm-alpha-near-1}.

\begin{figure}[ht!]
\centering
\begin{subfigure}{.5\textwidth}
  \centering
  \includegraphics[width=1\linewidth]{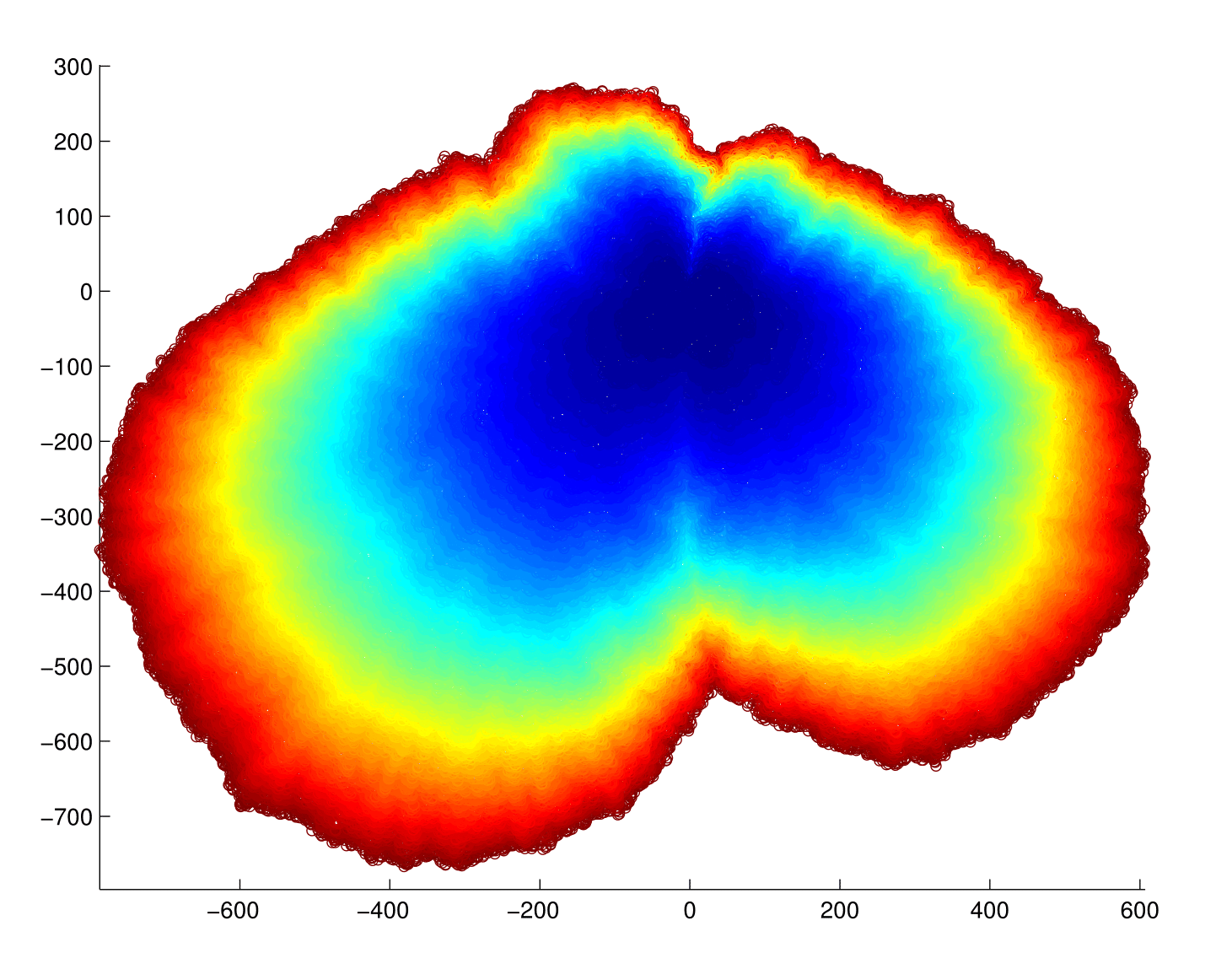} 
\end{subfigure}%
\begin{subfigure}{.5\textwidth}
  \centering
  \includegraphics[width=1\linewidth]{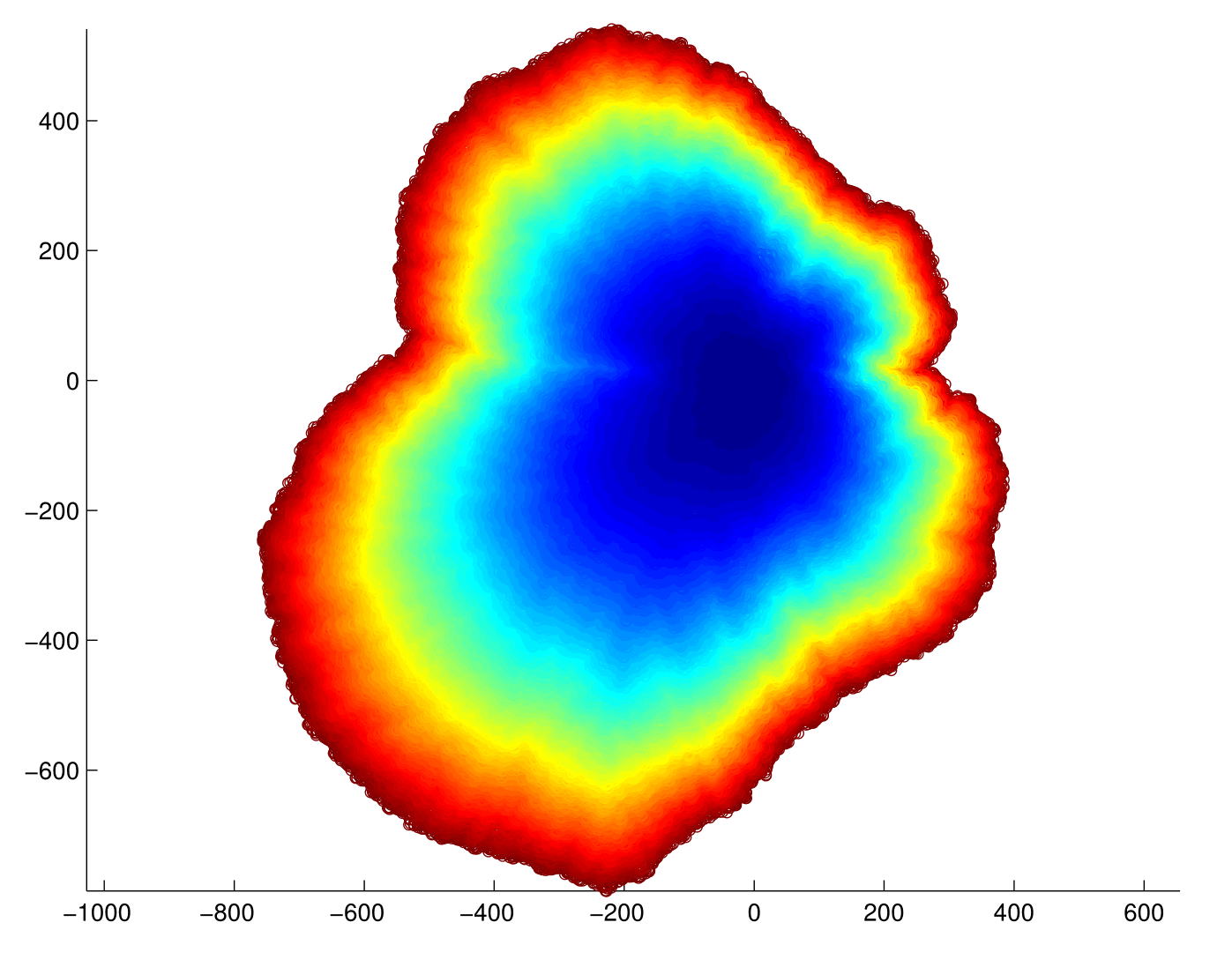} 
\end{subfigure} 
\caption{\textbf{Left panel:} A simulation of an $f$-weighted FPP cluster with $f(z) = \|z\|_1^{1.2}$, where here $\|\cdot\|_1$ is the $L^1$ norm. The figure illustrates the conclusion of Theorem~\ref{thm-alpha-near-1}, namely that the clusters will a.s.\ cover all but finitely many points of $\BB Z^2$ before reaching $\infty$. However, these clusters need not grow in a symmetric manner
\textbf{Right panel:} A simulation of an $f$-weighted FPP cluster with $f(z) = \|z\|_1 $ (so $\alpha=1$). The clusters do not appear to be converging toward a deterministic limit shape, but it is conceivable that they converge toward a random limit shape or that they converge toward a deterministic limit shape at a very slow rate.
}\label{fig-sim4} 
\end{figure}

\subsection{Outline}

The remainder of this paper is structured as follows. In Section~\ref{sec-prelim}, we prove some basic properties of the weighted FPP model of Definition~\ref{def-fpp} at a greater level of generality than what we will consider in the remainder of the paper. In Section~\ref{sec-fpp-estimates}, we prove several lemmas which allow us to approximate $f$-weighted FPP passage times via the deterministic metric $D$ of~\eqref{eqn-weighted-metric}. In Section~\ref{sec-small-alpha}, we use these estimates to prove Theorems~\ref{thm-alpha<1} and Theorem~\ref{thm-alpha-near-1}. In Section~\ref{sec-cone-contain}, we prove Theorem~\ref{thm-alpha>1}. In Section~\ref{sec-open-problems}, we list some open problems related to the model studied in this paper.

\section{General results for weighted FPP}
\label{sec-prelim}

Throughout this section we assume we are in the setting of Definition~\ref{def-fpp} for a general choice of graph $G$, starting vertex $v_0$, and weights $\op{wt}$. We recall in particular the FPP clusters $\{A_t\}_{t\geq 0}$ and the FPP filtration $\{\mcl F_t\}_{t\geq 0}$. 

In this section we will point out some basic properties of the model of Definition~\ref{def-fpp}. In later sections we will only need the case where $G = \BB Z^d$, $v_0 = 0$, and $\op{wt}$ is as in~\eqref{eqn-f-weights}, but it is no more difficult to treat the general case. 
In Section~\ref{sec-fpp}, we state the strong Markov property of our model (which follows from the fact that the passage times have an exponential distribution) and deduce some basic consequences. In Section~\ref{sec-comparison}, we will prove a lemma which allows us to compare weighted FPP to standard FPP (equivalently, the unweighted Eden model). In Section~\ref{sec-one-end}, we will prove a weak form of one-endedness for weighted FPP clusters in the case where the graph $G$ is infinite and the passage time to $\infty$, $\tau_\infty$, is a.s.\ finite.

\subsection{Markov property and applications}
\label{sec-fpp}

The following lemma gives a Markov property for weighted FPP clusters, and is the reason why we consider exponential passage times.

\begin{lem}[Strong Markov property] \label{prop-fpp-law}
Let $\tau$ be a stopping time for the FPP filtration $\{\mcl F_t\}_{t\geq 0}$. The conditional law of the passage times of the explored edges, $\{X_e \,:\, e\in \mcl E(G\setminus A_\tau)\}$ given $\mcl F_\tau$ is described as follows.
\begin{itemize}
\item For $e\in \mcl E(G)\setminus \left( \mcl E(A_\tau) \cup \partial \mcl E( A_\tau) \right)$, the conditional law of $X_e$ is the same as its marginal law. 
\item For $e\in \partial \mcl E(A_\tau)$, the conditional law of $e$ is that of an exponential random variable of parameter $\op{wt}(e)$ plus $\tau - \wt T^\tau(e)$, where $\wt T^\tau(e)$ is the minimum of $T(\eta)$ over all paths $\eta$ in $A_\tau$ joining $v_0$ to an endpoint of $e$. 
\item The random variables $\{X_e \,:\, e\in \mcl E(G) \setminus \mcl E(A_\tau)\}$ are conditionally independent given $\mcl F_\tau$.
\end{itemize}
\end{lem}
\begin{proof}
The case where $\tau$ is deterministic follows from the memoryless property of exponential random variables. From this, we immediately obtain the case where $\tau$ takes on only countably many possible values. The case of a general stopping time $\tau$ is proven by approximating $\tau$ by a sequence of stopping times which take on only countably many possible values.
\end{proof}

Lemma~\ref{prop-fpp-law} motivates the following definition. 

\begin{defn} \label{def-modified-exp-rv}
For $t\in \BB R$ and an edge $e \in \mcl E(G)\setminus \mcl E(A_t)$, let
\eqb \label{eqn-modified-exp-rv}
\wh X_e^t :=
\begin{cases}
X_e  & e\notin \partial \mcl E(A_t) \\
X_e  -  t + \wt T^{t}(e)   &e \in \partial \mcl E(A_t)
\end{cases}
\eqe 
where $\wt T^{t}(e)$ is as in Lemma~\ref{prop-fpp-law}. For a path $ \eta$ in $ G \setminus A_\tau $, let
\eqb \label{eqn-modified-exp-sum}
\wh T^t( \eta) := \sum_{i=1}^{| \eta|} \wh X_{ \eta(i)}^t .
\eqe 
\end{defn}

By Lemma~\ref{prop-fpp-law}, if $\tau$ is a stopping time for the filtration $\{\mcl F_t\}_{t\geq 0}$, then the conditional law given $\mcl F_{\tau}$ of $\{\wh X_{e}^\tau\}_{e \in \mcl E(G)}$ is that of a collection of independent exponential random variables where each $\wh X_{e}^t$ has parameter $\op{wt}(e)$. Furthermore, if $\eta$ is a path in $\mcl E(G) \setminus A_t$ with only one edge lying in $\partial \mcl E(A_t)$, then 
\eqb \label{eqn-modified-time-compare}
\wh T^t(\eta) = T(\eta) -t  + \wt T^t(\eta(1)) .
\eqe 

Lemma~\ref{prop-fpp-law} easily implies the following, which gives the equivalence of the model of Definition~\ref{def-fpp} and the weighted Eden model described in Section~\ref{sec-overview}.

\begin{lem} \label{prop-fpp-equiv}
Assume we are in the setting of Definition~\ref{def-fpp} with $\op{wt}(e) > 0$ for each $e\in \mcl E(G)$. Let $t_0 = 0$ and for $n\in\BB N$, let $t_n$ be the smallest $t\geq 0$ for which $\# A_t \geq n+1$. Let $\wt A_n := \mcl E( A_{t_n})$. Then the law of the sequence of random sets $\{ \wt A_n \}_{n\in\BB N}$ is described as follows. Let $\wt A_0 := \{\emptyset\}$. Let $ e_1$ be chosen uniformly from the set of edges of $G$ incident to $v_0$ and let $\wt A_1 := \{e_1\}$. 
Inductively, if $n\in\BB N$ and $\wt A_{n-1}$ has been defined, let $e_n$ be sampled from the uniform measure on the set of edges adjacent to $\wt A_{n-1}$ weighted by $\op{wt}$. Let $\wt A_n := \wt A_{n-1} \cup \{e_n\}$.  
\end{lem}

We next record another application of the random variables of Definition~\ref{def-modified-exp-rv}, namely a monotonicity statement for realizations of the cluster $A_\tau$ when $\tau$ is a stopping time for $\{\mcl F_t\}_{t\geq 0}$. 

\begin{lem} \label{prop-hitting-time-law}
Let $\tau$ be a stopping time for $\{\mcl F_t\}_{t\geq 0}$. Let $V$ be a subset of $\mcl V(G)  $ chosen in a manner which is measurable with respect to $A_\tau$. Let $\sigma$ be the smallest $t \geq 0 $ for which $\mcl V(A_t) \cap V \not=\emptyset$. 
\begin{enumerate}
\item $\sigma-\tau$ is conditionally independent from $\mcl F_\tau$ given $A_\tau$. \label{item-hitting-time-ind}
\item Let $\frk A$ and $\frk A'$ be two possible realizations of $A_\tau$ such that $\frk A \subset \frk A'$ and the realizations of $V$ corresponding to $\frk A$ and $\frk A'$ are the same. Then the conditional law of $\sigma-\tau$ given $\{A_\tau = \frk A \}$ stochastically dominates the conditional law of $\sigma-\tau$ given $\{A_\tau = \frk A'\}$. \label{item-hitting-time-dom}
\end{enumerate}
\end{lem}
\begin{proof}
First we prove assertion~\ref{item-hitting-time-ind}. 
Let $H$ be the set of simple paths $\eta$ for which the following is true.
\begin{enumerate}
\item $\eta$ connects $v_0$ to a vertex in $V$. \label{item-path-set-connect}
\item $\eta$ contains exactly one edge in $\bdy \mcl E(A_\tau)$. \label{item-path-set-bdy}
\item Let $i_\eta$ be the time $i$ for which $\eta(i_\eta) \in \bdy \mcl E(A_\tau)$. There is no path $\wt\eta$ in $A_\tau$ whose last edge shares an endpoint with $\eta(i_\eta)$ and which satisfies $T(\wt\eta) < T(\eta|_{[1,i_\eta-1]})$. \label{item-path-set-min}
\end{enumerate} 
For $\eta\in H$, we write 
\[
\wh\eta := \eta|_{[i_\eta , |\eta|]_{\BB Z}} .
\]
Then a.s.\ $\sigma = \min_{\eta \in H} T(\eta) $. 

Define the random variables $\wh X_e^{\tau }$ for $e\in \mcl E(G) \setminus \mcl E(A_{\tau} )$ and the passage times $\wt T^\tau(\cdot)$ and $\wh T^{\tau }(\cdot)$ as in Definition~\ref{def-modified-exp-rv}. Note that condition~\ref{item-path-set-min} in the definition of $H$ implies that $\wt T^\tau(\eta(i_\eta)) = T(\eta\setminus \wh\eta)$. Hence for $\eta \in H$, 
\alb
T(\eta) = T(\wh\eta) + T(\eta \setminus \wh\eta)   =    T(\wh\eta) + \wt T^\tau(\eta(i_\eta)) .
\ale
By~\eqref{eqn-modified-time-compare}, we obtain $T(\eta) -\tau  = \wh T^\tau(\wh\eta)$. Therefore,
\eqb \label{eqn-hitting-time-diff}
\sigma-\tau = \min_{\eta \in H} \wh T^\tau(\wh \eta) 
\eqe 
is a deterministic functional of the set $V$ and the random variables $\wh X_e^\tau$ for $e\in \mcl E(G)\setminus \mcl E(A_\tau)$. By Lemma~\ref{prop-fpp-law}, the conditional law of this latter collection of random variables given $\mcl F_\tau$ depends only on $A_\tau$, so this collection of random variables is conditionally independent from $\mcl F_\tau$ given $A_\tau$. We thus obtain assertion~\ref{item-hitting-time-ind}.

Now suppose we are in the setting of assertion~\ref{item-hitting-time-dom}. Let $\wh H$ be the set of simple paths $\wh\eta$ whose first edge belongs to $\bdy \mcl E( \frk A)$, none of whose other edges belong to $\bdy \mcl E(\frk A)$, and whose last edge is incident to a vertex in $V$. In the notation introduced at the beginning of the proof, $\wh H$ is the set of paths $\wh \eta$ for $\eta \in H$ on the event $\{A_\tau = \frk A\}$. Define $\wh H'$ similarly but with $\frk A'$ in place of $\frk A$. For $\wh\eta \in \wh H$, let $i'$ be the largest $i\in [1,|\wh\eta|]_{\BB Z}$ with $\wh\eta(i) \in \partial \frk A'$ and let $\wh\eta' := \wh\eta|_{[i' , |\wh\eta|]_{\BB Z}}$. Then $\wh\eta'\in \wh H'$. 

Let $\{\wh X_e \,:\, e\in \mcl E(G) \}$ be a collection of independent exponential random variables, each with parameter $\op{wt}(e)$. For a path $\eta$ in $G$, let $\wh T(\eta) := \sum_{e\in \eta} \wh X_e$. By~\eqref{eqn-hitting-time-diff}, the conditional law of $\sigma-\tau$ given $\{A_\tau = \frk A\}$ (resp. $\{A_\tau = \frk A\}$) is the same as the law of 
\eqbn
\min_{\wh\eta \in \wh H} \wh T(\wh \eta) \quad \left(\text{resp. } \min_{\wh\eta' \in \wh H'} \wh T(\wh \eta') \right) .
\eqen
Since $\wh\eta \mapsto \wh\eta'$ is a surjective map from $\wh H$ to $\wh H'$, we obtain the desired stochastic domination. 
\end{proof}

\subsection{Comparison to standard exponential FPP}
\label{sec-comparison}

In this subsection, we will record some observations which allow us to compare the model of Section~\ref{sec-comparison} to standard FPP on $G$ (i.e.\ with all of the edge weights $\op{wt}(e)$ equal to 1). For this purpose we first define a collection of iid exponential random variables which are related to the weighted FPP passage times $X_e$. 

\begin{defn} \label{def-normalized-exp-rv}
For $t  \geq 0$ and an edge $e \in \mcl E(G\setminus A_t)$, let $\ol X^t_e := \op{wt}(e) \wh X_e^t$, with $\wh X_e^t$ as in Definition~\ref{def-modified-exp-rv}. Also let $\{\ol X_e^t\}_{e\in \mcl E(A_t)}$ be a collection of random variables whose conditional law given $\mcl F_t$ is that of a family of iid exponential random variables with parameter 1, independent from the random variables $\wh X_e^t$ for $e\in \mcl E(G\setminus A_t)$. For a path $ \eta$ in $ G $, let
\eqb \label{eqn-normalized-exp-sum}
\ol T^t( \eta) := \sum_{i=1}^{| \eta|} \ol X_{ \eta(i)}^t .
\eqe 
For $v\in \mcl V(G)$, also let $\{\ol A_{v,s}^t \}_{s \geq 0}$ be the FPP clusters started from $v$ corresponding to the collection of random variables $\{\ol X_e^t\}_{e \in \mcl E(G)}$, i.e.\ $e \in \mcl E( \ol A_{v,s}^t)$ if and only if there is a path $\eta$ in $\mcl E(G)$ joining $v$ to $e$ with $\ol T^t(\eta) \leq s$ and $\mcl V(\ol A_{v,s}^t)$ is the set of endpoints of edges in $\mcl E( \ol A_{v,s}^t)$.
\end{defn} 

We also define an \emph{FPP geodesic} from $v \in \mcl V(G)$ to $e \in \mcl E(G) $ to be a path $\eta$ in $G$ such that $\eta(1)$ is incident to $v $, $\eta(|\eta|) = e$, and $T(\eta )$ is minimal among all such paths. If we do not specify the point $v$, we assume $v = v_0$ is the root vertex of $G$.  It is easy to see that there a.s.\ exists at most one FPP geodesic from $v$ to $e$. 

\begin{lem} \label{prop-normalized-compare}
Let $\tau$ be a stopping time for the FPP filtration $\{\mcl F_t\}_{t\geq 0}$ and define the random variables $\ol X^\tau_e$ for $e\in \mcl E(G)$ and the clusters $\ol A_{v,t}^\tau$ for $t\geq 0$ as in Definition~\ref{def-normalized-exp-rv}. Then the conditional law of $\{\ol X_e^\tau\}_{e \in \mcl E(G)}$ given $\mcl F_\tau$ is that of a collection of iid exponential random variables, each of which has parameter $1$. If $s > 0$ and $e\in \mcl E(G)\setminus \mcl E(A_\tau)$ then the following holds.
\begin{enumerate}
\item Suppose $e \in \mcl E(A_{\tau+s})$. Let $\eta_e : [1,|\eta_e|]_{\BB Z} \rta \mcl E(G)$ be the FPP geodesic from $v_0$ to $e$. Let $v $ be the last vertex in $\partial \mcl V(A_\tau)$ crossed by $\eta_e$ and let $\ol m := \max_{e'\in \eta_e \setminus \mcl E(A_\tau)} \op{wt}(e')$. Then $e\in \mcl E( \ol A_{v,\ol m  s}^\tau)$.  \label{item-weight-to-normal}
\item Suppose there exists a simple path $\eta$ in $G$ started from $v_0$ such that 
$e\in\eta$,
$\eta \cap \partial \mcl E(A_\tau) \not=\emptyset$, and 
with $\ol T^\tau(\eta)$ as in~\eqref{eqn-normalized-exp-sum},  
\eqbn \label{eqn-normal-to-weight-sum} 
\left( \min_{e' \in \eta} \op{wt}(e') \right)^{-1} \ol T^\tau(\eta) \leq s .
\eqen 
Then $e \in \mcl E(A_{\tau+s})$.   \label{item-normal-to-weight} 
\end{enumerate}
\end{lem}
\begin{proof}
From the strong Markov property (Lemma~\ref{prop-fpp-law}), Definition~\ref{def-normalized-exp-rv}, and the scaling property of exponential random variables, it is clear that the conditional law of $\{\ol X_e^\tau\}_{e \in \mcl E(G)}$ given $\mcl F_\tau$ is as claimed. 

Now suppose the hypotheses of assertion~\ref{item-weight-to-normal} are satisfied. Let $i_e$ be the (a.s.\ unique) integer $i\in [1,|\eta_e|]_{\BB Z}$ for which $\eta_e(i) \in \bdy \mcl E(A_\tau)$. Then $\eta_e\setminus \mcl E(A_\tau) = \eta_e([i_e , |\eta_e|]_{\BB Z})$ is a simple path in $\mcl E(G)$ and only its first edge belongs to $\partial \mcl E(A_\tau)$. Since $e\in \mcl E(A_{\tau+s})$,~\eqref{eqn-modified-time-compare} implies that (with $\wt T^\tau(\cdot)$ as in Lemma~\ref{prop-fpp-law})
\alb
\tau+ s 
\geq T(\eta_e) 
 =  T(\eta_e \setminus \mcl E(A_\tau)) + T \left(\eta_e \cap \mcl E(A_\tau) \right)
= T(\eta_e \setminus \mcl E(A_\tau) ) + \wt T^\tau (\eta_e(i_e) )
=  \wh T^\tau(\eta_e \setminus \mcl E(A_\tau)) + \tau .
\ale 
Hence 
\eqbn
s \geq \wh T^\tau(\eta_e \setminus \mcl E(A_\tau) )
\geq \ol m^{-1} \ol T^\tau(\eta_e\setminus \mcl E(A_\tau)) .
\eqen
Therefore $e\in \mcl E(A_{v,\ol m  s}^\tau)$. 

Finally, we consider the setting of assertion~\ref{item-normal-to-weight}. Let $i_1$ be the largest $i \in [1,|\eta|]_{\BB Z}$ such that $\eta(i) \in \partial \mcl E(A_\tau)$ and let $\eta_1 := \eta|_{[i_1 , |\eta|]_{\BB Z}}$. Then $\eta_1$ is a simple path whose first edge belongs to $\partial \mcl E(A_\tau)$, none of whose other edges belong to $\partial \mcl E(A_\tau)$, and one of whose edges is $e$. We have
\eqbn
\wh T^\tau(\eta_1) \leq \wh T^\tau(\eta) 
\leq \left( \min_{e' \in \eta} \op{wt}(e') \right)^{-1} \ol T^\tau(\eta) \leq s .
\eqen
By~\eqref{eqn-modified-time-compare}, $T(\eta_1) \leq s + \tau - \wt T^\tau(\eta(i_1))$. 
By definition of $\wt T^\tau$ (see Lemma~\ref{prop-fpp-law}), there is a path $\eta_0$ in $ A_\tau $ joining 0 to an endpoint of $\eta_1(i_1)$ which satisfies $T(\eta_0)  =\wt T^\tau(\eta(i_1))$. If we let $\eta_*$ be the concatenation of $\eta_0$ and $\eta_1$, then $T(\eta_*)   \leq  \tau + s$, so $e\in \mcl E(A_{\tau+s})$. 
\end{proof}

\subsection{Weak one-endedness}
\label{sec-one-end}

The purpose of this subsection is to prove the following weak form of ``one-endedness" for the clusters $\{\mcl F_t\}_{t\geq 0}$, which will be used in the proof of Theorem~\ref{thm-alpha>1} to rule out the possibility that all but finitely points of $A_{\tau_\infty}$ are contained in the disjoint union of two cones of opening angle $<\pi$, rather than a single such cone.

\begin{prop} \label{prop-one-end-weak}
Suppose we are in the setting of Definition~\ref{def-fpp} with $G$ infinite and the weights $\op{wt}$ are chosen in such a way that $\tau_\infty<\infty$ a.s. Let $\tau$ be a stopping time for $\{\mcl F_t\}_{t\geq 0}$ with $\tau<\tau_\infty$ a.s. Let $\Gamma_1$ and $\Gamma_2$ be infinite subgraphs of $G$ which lie at graph distance at least 3 from one another, each of which shares a vertex with $A_\tau$, chosen in some $\mcl F_\tau$-measurable manner (i.e.\ $\Gamma_1$ and $\Gamma_2$ are $\mcl F_\tau$-measurable random variables). Then 
\eqb \label{eqn-one-end-weak}
\BB P\left( \#\mcl E \left( A_{\tau_\infty} \cap \Gamma_1 \right) =  \# \mcl E \left( A_{\tau_\infty}\cap \Gamma_2  \right) = \infty ,\,  \#\left(\mcl E(A_{\tau_\infty}) \setminus  \mcl E(\Gamma_1 \cap \Gamma_2)  \right) <\infty \,|\, \mcl F_\tau \right) = 0. 
\eqe 
\end{prop}

We note that Proposition~\ref{prop-one-end-weak} is \emph{not} sufficient to conclude that $A_{\tau_\infty}$ is a.s.\ one ended (i.e., that for large enough $n$ the set of edges of $A_{\tau_\infty}$ which are not contained in the graph distance ball of radius $n$ centered at the starting vertex is connected). For this to be the case we would need~\eqref{eqn-one-end-weak} to hold simultaneously a.s.\ for every choice of $\Gamma_1$ and $\Gamma_2$. However, Proposition~\ref{prop-one-end-weak} is sufficient for the proof of Theorem~\ref{thm-alpha>1}. 

Proposition~\ref{prop-one-end-weak} is proven via a purely probabilistic argument. See Figure~\ref{fig-weak-one-end} for an outline of the proof. We first need the following elementary lemma.

\begin{figure}[ht!]
 \begin{center}
\includegraphics[scale=.8]{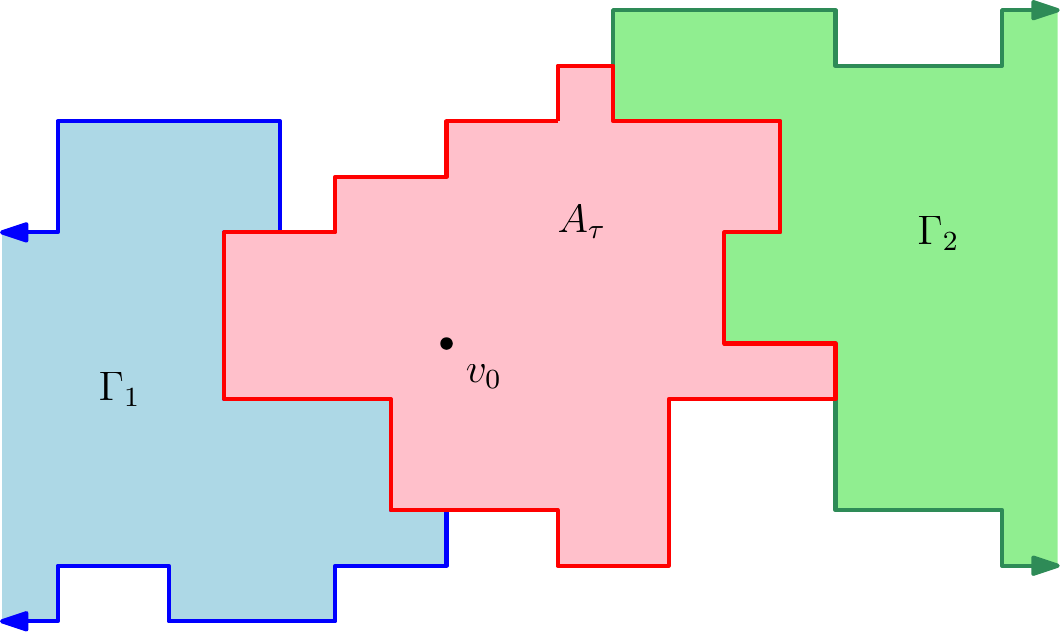} 
\caption{An illustration of a the setup for Proposition~\ref{prop-one-end-weak} with $G = \BB Z^d$. To prove the proposition, we first prove a version of~\eqref{eqn-one-end-weak} where we require that $\# \left(\mcl E(A_{\tau_\infty}) \setminus  \mcl E(A_\tau \cap (\Gamma_1\cup \Gamma_2))  \right)$ is empty, rather than finite. Let $\sigma_1$ (resp. $\sigma_2$) be the first time that $A_t$ covers infinitely many edges of $\Gamma_1$ (resp. $\Gamma_2$) and let $E_1$ (resp. $E_2$) be the event that there is no path in $A_{\sigma_1}$ (resp. $A_{\sigma_2}$) which crosses the boundary of $\Gamma_1$ (resp. $\Gamma_2$). Lemma~\ref{prop-fpp-law} implies that $\sigma_1 \BB 1_{E_1}$ and $\sigma_2 \BB 1_{E_2}$ are conditionally independent given $\mcl F_\tau$. Hence if $\sigma_1 \BB 1_{E_1} = \sigma_2 \BB 1_{E_2} = \tau_\infty$ with positive conditional probability given $\mcl F_\tau$, then the conditional law of $\tau_\infty$ given $\mcl F_\tau$ must have an atom with positive probability. This contradicts Lemma~\ref{prop-infty-non-atomic}, and we obtain the desired weaker version of~\eqref{eqn-one-end-weak}. The full version is obtained by applying the weaker version to countably many stopping times between $\tau$ and $\tau_\infty$ which increase to $\tau_\infty$. 
}\label{fig-weak-one-end}
\end{center}
\end{figure}

\begin{lem} \label{prop-equal-atom}
Let $X_1 , X_2$, and $Y$ be random variables taking values in a common state space $\mcl X$. Suppose that $X_1$ and $X_2$ are independent and that
\eqb \label{eqn-equal-atom}
\BB P\left(X_1 = X_2 = Y \right) > 0. 
\eqe 
Then there is a deterministic $x \in \mcl X$ such that
\eqb \label{eqn-equal-atom-conclusion}
\BB P\left(X_1 = X_2 = Y = x \right) > 0.
\eqe 
\end{lem}
\begin{proof}
For $i\in \{1,2\}$, let $\mcl A_i$ be the set of atoms of the law of $X_i$, i.e.\ the set of $x\in \mcl X$ such that $\BB P(X_i = x) > 0$. 
We first claim that it is a.s.\ the case that on the event $\{X_1 = X_2\}$, the common value of $X_1$ and $X_2$ belongs to $\mcl A_1\cap \mcl A_2$. To see this, we observe that by independence,
\eqbn
\BB P\left(X_1 = X_2 \,|\, X_2\right) \BB 1_{(X_2 \in \mcl X\setminus \mcl A_1)} = 0, 
\eqen
so
\eqbn
\BB P\left(X_1 = X_2  \in \mcl X\setminus \mcl A_1 \right) = 0.
\eqen
By symmetry, also $\BB P\left(X_1 = X_2  \in \mcl X\setminus \mcl A_2 \right) = 0$. Hence~\eqref{eqn-equal-atom} implies that
\eqbn
\BB P\left( X_1 = X_2 = Y \in \mcl A_1 \cap \mcl A_2 \right) > 0.
\eqen
The set $\mcl A_1\cap \mcl A_2$ is countable, so there must exist $x\in \mcl A_1\cap \mcl A_2$ for which~\eqref{eqn-equal-atom-conclusion} holds.
\end{proof}

\begin{lem} \label{prop-infty-non-atomic}
Suppose we are in the setting of Definition~\ref{def-fpp} with $G$ infinite and the weights $\op{wt}$ are such that a.s.\ $\tau_\infty < \infty$. a.s. 
Let $\tau$ be a stopping time for $\{\mcl F_t\}_{t\geq 0}$ with $\tau < \tau_\infty$ a.s. Almost surely, the conditional law given $\mcl F_\tau$ of the random variable $\tau_\infty-\tau$ is non-atomic, i.e.\ 
\eqbn
\BB P(\tau_\infty-\tau = t \,|\, \mcl F_\tau) = 0 ,\quad \forall t \geq 0 .
\eqen
\end{lem} 

Roughly speaking, the idea of the proof is to write $\tau_\infty -\tau = \tau_\infty - \tau' + \tau' - \tau$, where $\tau'$ is the smallest time $t > \tau'$ at which another edge is added to the cluster. The conditional law of $\tau'$ given $\mcl F_\tau$ is non-atomic since it is the minimum of finitely many exponential random variables, and $\tau_\infty - \tau'$ is ``almost" conditionally independent from $\tau' - \tau$ given $\mcl F_\tau$ due to the strong Markov property (Lemma~\ref{prop-fpp-law}). However, $\tau'-\tau$ and $\tau_\infty-\tau'$ are not quite conditionally independent since the law of $\tau_\infty - \tau'$ depends on the particular realization of $A_{\tau'}$, which in turn might depend on $\tau'-\tau$, so slightly more work is needed.

\begin{proof}[Proof of Lemma~\ref{prop-infty-non-atomic}]
Suppose by way of contradiction that the statement of the lemma is false, i.e.\ the conditional law of $\tau_\infty-\tau$ given $\mcl F_\tau$ has an atom with positive probability. Let $S$ be an $\mcl F_\tau$-measurable random variable chosen in such a way that  
\eqb \label{eqn-choose-pt-mass}
\BB P\left(\tau_\infty-\tau = S   \right) > 0 .
\eqe 
For example, $S$ could be the location of the largest atom of the conditional law of $\tau_\infty-\tau$ given $\mcl F_\tau$ if it exists (with ties broken in some arbitrary $\mcl F_\tau$-measurable manner) or $S \equiv 0$ if no such atom exists.
 
Let $\tau'$ be the smallest $t  > \tau$ for which $A_t\not=A_\tau$. 
If $\tau_\infty - \tau = S$, then $\tau_\infty - \tau' = S - (\tau'-\tau)$. The random variable 
\eqb \label{eqn-2nd-mass-def}
S' := S - (\tau'-\tau)
\eqe 
is $\mcl F_{\tau'}$-measurable and by~\eqref{eqn-choose-pt-mass}, 
$\BB P\left(\tau_\infty-\tau' = S'  \right)  > 0$.
There are only countably many possible realizations of $A_{\tau'}$, so we can find a positive-probability realization $\frk A'$ of $A_{\tau'}$ such that
\eqbn
\BB P\left(\tau_\infty-\tau' = S' \,|\, A_{\tau'} = \frk A' \right) > 0 .
\eqen
Since $S'$ is $\mcl F_{\tau'}$-measurable, Lemma~\ref{prop-hitting-time-law} implies that the random variables $\tau_\infty-\tau'$ and $S'$ are conditionally independent given $\{A_{\tau'} = \frk A'\}$. By Lemma~\ref{prop-equal-atom} (applied with $X_1 = Y = \tau_\infty-\tau'$ and $X_2 = S'$) there exists a deterministic $t' > 0$ such that 
\eqbn
\BB P\left(\tau_\infty-\tau'  =  S' = t' \,|\, A_{\tau'} = \frk A' \right) > 0 .
\eqen
In particular $\BB P\left(S' = t'  \right) > 0$, so (recall~\eqref{eqn-2nd-mass-def}) 
\eqbn
\BB P\left( \BB P\left(\tau' - \tau = S - t'  \,|\, \mcl F_\tau \right) > 0 \right) > 0 .
\eqen
The random variable $S - t'$ is $\mcl F_\tau$-measurable, so with positive probability the conditional law of $\tau'-\tau$ given $\mcl F_\tau$ has an atom at $S-t'$. By Lemma~\ref{prop-fpp-law}, the conditional law of $\tau'-\tau$ given $\mcl F_\tau$ is that of the minimum of finitely many independent exponential random variables, so a.s.\ this conditional law is non-atomic. This contradiction completes the proof.
\end{proof}

The following is the main input in the proof of Proposition~\ref{prop-one-end-weak}.

\begin{lem} \label{prop-one-end-contained}
Suppose we are in the setting of Proposition~\ref{prop-one-end-weak}. For $i\in \{1,2\}$, let $\sigma_i$ be the smallest $t > \tau$ for which $ \#\left(\mcl V(A_t\cap \Gamma_i) \right) = \infty$. Also let $E_i$ be the event that there is no path in $ A_{\sigma_i}  \setminus A_\tau $ which contains an edge in $ \Gamma_i $ and an edge in $\bdy \mcl E(\Gamma_i)$. Then 
\eqb \label{eqn-equal-hit-times}
\BB P\left(\sigma_1 = \sigma_2 = \tau_\infty  ,\, E_1\cap E_2 \,|\, \mcl F_\tau\right)  = 0.
\eqe 
\end{lem}  
\begin{proof}
We first argue that the random variables $\sigma_1\BB 1_{E_1}$ and $\sigma_2 \BB 1_{E_2}$ are conditionally independent given $\mcl F_\tau$. To see this, define the random variables $\wh X_e^\tau$ for $e\in \mcl E(G)\setminus \mcl E(A_\tau)$ as in Definition~\ref{def-modified-exp-rv}, so that the conditional law of the $\wh X_e^\tau$'s given $\mcl F_\tau$ is that of a collection of iid exponential random variables with parameters $\op{wt}(e)$. Since $\Gamma_1$ and $\Gamma_2$ lie at graph distance at least 3 from one another, the sets $\mcl E(\Gamma_1) \cup \bdy \mcl E(\Gamma_1)$ and $\mcl E(\Gamma_2) \cup \bdy\mcl E(\Gamma_2)$ are disjoint. Therefore, the collections of random variables
\eqb \label{eqn-disjoint-set-rv}
\left\{\wh X_e^\tau \,:\, e \in \left( \mcl E(\Gamma_i) \cup \bdy \mcl E(\Gamma_i) \right) \setminus \mcl E(A_\tau)  \right\}  
\eqe
for $i\in \{1,2\}$ are conditionally independent given $\mcl F_\tau$.

For $i\in\{1,2\}$, let $\wh\tau_i$ be the smallest $t>\tau$ for which the following is true. For infinitely many $v\in \mcl V(\Gamma_i)$, there exists an infinite path $\eta $ in $G$ from $\bdy \mcl V(A_\tau)$ to $v$ which is contained in $ \Gamma_i $ and satisfies
\eqb \label{eqn-hat-path-sum}
\sum_{e\in\eta } \wh X_e^\tau \leq t - \tau .
\eqe 
Also let $\wh\sigma_i$ be the smallest $t > \tau$ for which there exists a finite path $\eta $ in $\mcl E(G)\setminus \mcl E(A_\tau)$ which contains an edge of $ \Gamma_i $ and an edge in $\bdy\mcl E(\Gamma_i)$ and satisifes~\eqref{eqn-hat-path-sum}. 
Then $\wh \tau_i$ and $\wh \sigma_i$ are measurable functions of $\mcl F_\tau$ and the collection of random variables~\eqref{eqn-disjoint-set-rv}. Furthermore, the event $E_i$ occurs if and only if $\wh\tau_i < \wh\sigma_i$, in which case $\sigma_i = \wh\tau_i$. Hence $\sigma_i \BB 1_{E_i} $ is a measurable function of $\mcl F_\tau$ and the collection~\eqref{eqn-disjoint-set-rv}. Therefore $\sigma_1\BB 1_{E_1}$ and $\sigma_2 \BB 1_{E_2}$ are conditionally independent given $\mcl F_\tau$. 

Now suppose by way of contradiction that~\eqref{eqn-equal-hit-times} is false. Then  
\eqbn
\BB P\left(\sigma_1 \BB 1_{E_1} = \sigma_2 \BB 1_{E_2}  = \tau_\infty \,|\, \mcl F_\tau\right) > 0 .
\eqen
Since $\sigma_1\BB 1_{E_1}$ and $\sigma_2 \BB 1_{E_2}$ are conditionally independent given $\mcl F_\tau$, Lemma~\ref{prop-equal-atom} implies that we can find a $\mcl F_\tau$-measurable random variable $S$ such that with positive probability, 
\eqbn
\BB P\left(\sigma_1 = \sigma_2 = \tau_\infty = S \,|\, \mcl F_\tau\right) > 0.
\eqen
In particular, the conditional law of $\tau_\infty$ given $\mcl F_\tau$ has an atom with positive probability, which contradicts Lemma~\ref{prop-infty-non-atomic}.  
\end{proof}

\begin{proof}[Proof of Proposition~\ref{prop-one-end-weak}]
For $n\in\BB N$, let $\tau^n$ be the smallest $t \geq \tau$ for which $\mcl E(A_t) \setminus \mcl E(A_\tau)$ contains $n$ edges. Let $\sigma_1^n$, $\sigma_2^n$, $E_1^n$, and $E_2^n$ be as in Lemma~\ref{prop-one-end-contained} with $\tau^n$ in place of $\tau$. Then Lemma~\ref{prop-one-end-contained} implies that a.s.\ the event 
\eqbn
\left\{\sigma_1^n = \sigma_2^n = \tau_\infty   \right\} \cap E_1^n\cap E_2^n
\eqen
does not occur for any $n\in\BB N$. On the other hand, every $e\in\mcl E(A_{\tau_\infty})$ is contained in some $\mcl E(A_{\tau^n})$, so if the event in~\eqref{eqn-one-end-weak} occurs then there is a finite $n_0\in\BB N$ such that 
\eqb \label{eqn-one-end-exhaust}
\mcl E(A_{\tau_\infty} \setminus A_{\tau^{n_0}}) \subset \mcl E(\Gamma_1 \cup \Gamma_2) 
\eqe
and $\sigma_1^{n_0} = \sigma_2^{n_0} = \infty$. Since $ \Gamma_1 $ and $ \Gamma_2 $ lie at graph distance at least 3 from one another, the condition~\eqref{eqn-one-end-exhaust} implies that there is no path in $\mcl E(A_{\tau_\infty}) \setminus \mcl E(A_{\tau^{n_0}})$ which contains an edge of $\bdy \mcl E(\Gamma_1) \cup \bdy \mcl E(\Gamma_2)$. Therefore $E_1^{n_0} \cap E_2^{n_0}$ occurs. Hence the event in~\eqref{eqn-one-end-weak} must have probability zero.
\end{proof}

\section{Estimating passage times via a deterministic metric}
\label{sec-fpp-estimates}

In the remainder of this paper we will consider the $f$-weighted FPP process $\{A_t\}_{t\geq 0}$ on $\BB Z^d$ started from 0, as described in Section~\ref{sec-fpp-setup}, the associated filtration $\{\mcl F_t\}_{t\geq 0}$ from Definition~\ref{def-fpp}, as well as the metric $D$ from Section~\ref{sec-D-metric}. 

Throughout this section we allow a general choice of $\alpha\in \BB R$ and $\alpha$-weight function $f$. 
In this section, we will prove that the metric $D$ is a good approximation for passage times in our FPP model. We start in Section~\ref{sec-standard-fpp} by reviewing some known estimates for standard FPP. In Section~\ref{sec-D-metric-properties}, we prove some basic deterministic estimates for $D$. 
We then prove upper and lower bounds for $f$-weighted FPP passage times in terms of $D$ in Section~\ref{sec-growth-estimates}. These latter bounds will be the key inputs in the proofs of Theorems~\ref{thm-alpha<1},~\ref{thm-alpha>1}, and~\ref{thm-alpha-near-1} in the subsequent sections.

\subsection{Rate of convergence estimates for standard FPP}
\label{sec-standard-fpp} 

Recall the standard FPP limit shape $\BB A$ from Section~\ref{sec-D-metric} and the fattened standard FPP clusters $A_t^F$ for $t \geq 0$ from~\eqref{eqn-fatten-hull}.   
Estimates for the rate of convergence in~\eqref{eqn-cluster-conv} are obtained in~\cite{kesten-fpp-speed} and sharpened in~\cite{alexander-fpp-speed}. In particular,~\cite[Theorem 2]{kesten-fpp-speed} tells us that for each $\zeta \in (0,1/2)$ and each $t > 0$,  
\begin{equation} \label{eqn-hull-upper}
\BB P\left(   A_t^F \not\subset \left(1 + t^{-1/2+\zeta} \right) t \BB A    \right) = o_t^\infty(t)
\end{equation} 
at a rate depending only $\zeta$ and $d$ (here we recall the notation $o_t^\infty(t)$ from Section~\ref{sec-asymp-notation}). Furthermore, the proof of~\cite[Theorem 3.1]{alexander-fpp-speed} shows that for each $\zeta \in (0,1/2)$ and $t>0$  
\eqb \label{eqn-hull-lower}
\BB P\left( \left(1 - t^{-1/2+\zeta} \right) t\BB A  \not\subset A_t^F \right) = o_t^\infty(t), 
\eqe 
at a rate depending only $\zeta$ and $d$. 

\begin{remark} \label{remark-kesten-optimal}
It is expected that the error exponent $1/2$ in~\eqref{eqn-hull-upper} and~\eqref{eqn-hull-lower} is not optimal. Heuristic arguments and numerical simulations suggest that these estimates should hold with $2/3$ in place of $1/2$ in the case $d=2$; see~\cite{kpz-fluctuation} as well as the discussion immediately following~\cite[Theorem B]{kesten-fpp-speed} and the references therein.
\footnote{We remark that in any dimension, the error exponent in~\eqref{eqn-hull-upper} and~\eqref{eqn-hull-lower} is closely related to the so-called \emph{wandering exponent}, which measures the amount by which FPP geodesics deviate from straight lines. See~\cite{chatterjee-fpp-exponent} for a formula relating these exponents as well as~\cite{auf-dam-fpp-exponent} for a simplified proof of this formula.}
If we had such improved error estimates, then we would also obtain better error estimates in Lemmas~\ref{prop-hull-increment-time} and~\ref{prop-hull-increment-expand}, which would lead to better error estimates in Theorem~\ref{thm-alpha<1}. 
\end{remark}
 
We will sometimes have occasion to apply~\eqref{eqn-hull-upper} and~\eqref{eqn-hull-lower} with the scaling of $\BB A $, rather than the time $t$, specified. 
For this reason, we record the following estimates, which are immediate from~\eqref{eqn-hull-upper} and~\eqref{eqn-hull-lower}. For each $r > 0$, 
\eqb \label{eqn-hull-upper'}
\BB P\left(   A_{r - r^{1/2+\zeta}}^F  \not\subset r \BB A    \right) = o_r^\infty(r) ,  \quad \forall \zeta  > 0 
\eqe 
and
\eqb \label{eqn-hull-lower'}
\BB P\left( r \BB A  \not\subset A_{r  + r^{1/2 + \zeta}}^F \right) = o_r^\infty(r) ,\quad \forall \zeta > 0  
\eqe  
at a rate depending only on $\zeta$ and $d$.

\subsection{Estimates for the weighted metric}
\label{sec-D-metric-properties}

In this subsection we prove some basic estimates for the metric $D$ of~\eqref{eqn-weighted-metric} which will be used to compare $D$-distances to $f$-weighted FPP distances. We first have an upper bound for $D$-distances in terms of Euclidean distances.

\begin{lem} \label{prop-f-cont}
Let $f$ be as in~\eqref{eqn-f-decomp}. There is a constant $a > 0$, depending only on $f$, such that for each $z,w\in\BB R^d \setminus \{0\}$,  
\eqbn
|f(z) - f(w)| \leq a\left(|z|^{\alpha-1} \vee |w|^{\alpha-1} \right) |z-w| .
\eqen
\end{lem}
\begin{proof}
Let $w':= (|z|/|w|) w$. 
By Lipschitz continuity of $f_0$ and $\alpha$-homogeneity of $f$,   
\eqbn
|f(z) - f(w')| = |z|^\alpha  |f_0(z/|z|) - f_0(w'/|z|)| \preceq |z|^{\alpha-1} |z-w'| \preceq |z|^{\alpha-1} |z-w| .
\eqen 
Furthermore, by the mean value theorem  
\eqbn
|f(w) - f(w')| \preceq ||w|^\alpha -|z|^\alpha| \preceq \left(|z|^{\alpha-1} \vee |w|^{\alpha-1}\right) |z-w| . 
\eqen
Combining these inequalities proves the lemma.
\end{proof}

Our next lemma shows that $D$ is comparable to the metric induced by $\mu$ (and hence to that induced by any norm on $\BB R^d$) when we restrict attention to sets at positive distance from 0 and $\infty$.

\begin{lem} \label{prop-metric-compare} 
Let $z,w\in \BB R^d$. Then
\eqb \label{eqn-metric-compare}
\phi(z,w)  \leq   D(z,w)  \leq     \mu(z-w)   \int_0^1 f(t w + (1-t) z)^{-1}   \, dt        
\eqe 
where, with $\ol\kappa$ as in~\eqref{eqn-f-bound} and $\ol\rho$ as in~\eqref{eqn-optimal-norm},
\eqbn
\phi(z,w) = \begin{dcases}
\frac{ |z|^{1-\alpha}  - \left(|z| + \ol\rho \mu(w-z) \right)^{1-\alpha} }{\ol\rho \ol\kappa(\alpha-1)}  ,\quad &\alpha \in [0,\infty)\setminus \{1\} \\
\ol\rho^{-1} \ol\kappa^{-1} \log \left(\frac{  |z| + \ol\rho \mu(w-z)  }{ |z| }\right) ,\quad &\alpha =1 \\
\frac{ |z|^{1-\alpha}  - \left(|z| - \ol\rho \mu(w-z) \right)^{1-\alpha} }{\ol\rho \ol\kappa(\alpha-1)}  ,\quad &\alpha < 0 .
\end{dcases}
\eqen
\end{lem}
\begin{proof}
To obtain the upper bound in~\eqref{eqn-metric-compare}, let 
\eqbn
\gamma(t) := \frac{t w + \left(\mu(w-z) - t\right) z}{\mu(w-z)} ,\quad \forall t\in [0, \mu(w-z)] .
\eqen
Then $\gamma$ is parametrized by $\mu$-length and by a change of variables,
\alb
\op{len}^D(\gamma) 
= \int_0^{\mu(w-z)} f(\gamma(t))^{-1} \, dt 
= \mu(w-z) \int_0^1 f(tw + (1-t) z)^{-1} \, dt .  
\ale

To obtain the lower bound, let $\gamma : [0,T] \rta \BB R^d$ be a piecewise linear path from $z$ to $w$ parametrized by $\mu$-length. Then $T \geq \mu(w-z)$ and for each $t\in [0,T]$, 
\eqbn
|z| -\ol\rho t \leq |\gamma(t)| \leq |z| + \ol\rho t  .
\eqen
Hence for each $t\in [0,T]$, 
\eqbn
f(\gamma(t))^{-1} \geq
\begin{dcases}
\ol\kappa^{-1} (|z| + \ol\rho t)^{-\alpha} ,\quad &\alpha \geq 0 \\
\ol\kappa^{-1} (|z| - \ol\rho t)^{-\alpha} ,\quad &\alpha < 0.
\end{dcases}
\eqen  
If $\alpha \in [0,\infty)\setminus \{1\}$, we thus have
\eqbn
\op{len}^D(\gamma) 
\geq \ol\kappa^{-1} \int_0^T  (|z| + \ol\rho t)^{-\alpha} \, dt  
\geq \frac{ |z|^{1-\alpha}  - \left(|z| + \ol\rho \mu(w-z) \right)^{1-\alpha} }{\ol\rho  \ol\kappa (\alpha-1)}    .
\eqen
This concludes the proof in the case where $\alpha \in [0,\infty)\setminus \{1\}$. Similar arguments apply in the case where $\alpha < 0$ or $\alpha=1$. 
\end{proof}

\subsection{Growth estimates for weighted FPP}
\label{sec-growth-estimates}

In this subsection, we will use the bounds for standard FPP described in Section~\ref{sec-standard-fpp} to prove results which show that the metric $D$ of Section~\ref{sec-D-metric} is a good approximation for passage times in weighted FPP. The intuition behind the estimates of this subsection is as follows. If $v\in\BB Z^d$, then near $v$ the $f$-weighted FPP metric $T$ locally looks like the ordinary ($f\equiv 1$) FPP metric, re-scaled by $f(v) $. This, in turn, is well-approximated by the metric induced by the norm $f(v)\cdot \mu$ due to the estimates of Section~\ref{sec-standard-fpp}. Hence $f$-weighted FPP distances are comparable to $\mu$-distances, weighted by $f$, i.e., $D$-distances.
 
We first state our main upper bound for $f$-weighted FPP passage times. 
Roughly speaking, the estimate says that if $\tau$ is a stopping time for the $f$-weighted FPP filtration $\{\mcl F_t\}_{t\geq0}$ and $v \in \BB Z^d\setminus A_\tau $, then it is very unlikely that the amount of time after $\tau$ before $v$ is absorbed by the FPP clusters is too much larger than $D(v,A_\tau)$. 
The reason for the conditions involving $B_{m^\xi}^{|\cdot|}(0) \setminus B_{m^{1/\xi}}^{|\cdot|}(0)$ in the statement of the lemma is to keep us away from 0 and $\infty$, so that we can apply the estimates for $D$ in Section~\ref{sec-D-metric-properties} and so that we only need to consider polynomially many points (which is important when we apply a union bound to events with probability $o_m^\infty(m)$).

\begin{lem}[Upper bound for passage times] \label{prop-hull-time-metric}  
Fix $\xi > 1$. 
Let $\tau$ be a stopping time for $\{\mcl F_t\}_{t\geq 0}$ and let $m\in \BB N$. For $z , w \in B^{|\cdot|}_{m^\xi}(0) \setminus B^{|\cdot|}_{m^{1/\xi}}(0)$, let
\eqb \label{eqn-hull-time-restricted-D}
\wt D_m(z,w) := \inf\left\{\op{len}^D(\gamma) \,:\, \text{$\gamma$ is a piecewise linear path from $z$ to $w$ in $B^{|\cdot|}_{m^\xi }(0) \setminus B^{|\cdot|}_{ m^{1/\xi}}(0)$}\right\}  .
\eqe   
Also let 
\[
V_\tau^m  := \left( \BB Z^d  \cap \left( B^{|\cdot|}_{m^\xi}(0) \setminus B^{|\cdot|}_{m^{1/\xi}}(0) \right) \right) \setminus \mcl V(A_\tau) .
\]
Then for each $\beta \in (0,1/(3\xi))$, 
\eqbn
\BB P\left(T(0,v) - \tau \leq \left(1 + m^{ -\beta } \right) \wt D_m(v,A_\tau) + \frac{ m^{1/\xi -\beta   }}{|v|^\alpha}  ,\, \forall v\in V_\tau^m \,|\, \mcl F_\tau\right) = 1-o_m^\infty(m)   
\eqen
at a deterministic rate depending only on $\beta $ and $\xi$. 
\end{lem}

We note that the metric $\wt D_m$ of~\eqref{eqn-hull-time-restricted-D} is defined in the same manner as $D$ but with a smaller set of allowed paths. Hence $\wt D_m \geq D$. However, if $z,w \in \BB R^d$ and there is a $D$-geodesic from $z$ to $w$ which does not enter $B^{|\cdot|}_{m^\xi }(0) \setminus B^{|\cdot|}_{ m^{1/\xi}}(0)$, the $\wt D_m(z,w) = D(z,w)$.

The following is our main lower bound for $f$-weighted FPP passage times, which says that (roughly speaking) the amount of time after a stopping time $\tau$ before a vertex $v\in\BB Z^d\setminus A_\tau$ is absorbed is very unlikely to be much larger then $D(v,A_\tau)$. 

\begin{lem}[Lower bound for passage times] \label{prop-hull-expand-metric} 
Fix $\xi >1$.
Let $\tau$ be a stopping time for $\{\mcl F_t\}_{t\geq 0}$ with $\tau < \infty$ a.s.\ and let $m\in \BB N$. Let $\wh V_\tau^m $ be the set of $v \in \BB Z^d \setminus \mcl V(A_\tau)$ such that the FPP geodesic $\eta_v$ connecting 0 to $v$ satisfies $\eta_v \setminus \mcl E( A_\tau) \subset B_{m^\xi}^{|\cdot|}(0) \setminus B_{m^{1/\xi}}^{|\cdot|}(0)$.
For $v \in \wh V_\tau^m$, let $u_v$ be the last vertex of $\mcl V(A_\tau)$ hit by the geodesic $\eta_v$. 
Let $D$ be the metric~\eqref{eqn-weighted-metric}. For each $\beta \in (0,1/(3\xi))$, 
\eqbn
\BB P\left(T(0,v) - \tau \geq \left(1-m^{-\beta } \right) D(u_v , v) - \frac{m^{1/\xi - \beta}}{|v|^\alpha} ,\,  \forall  v \in \wh V_\tau^m     \,|\, \mcl F_\tau\right) = 1 - o_m^\infty(m)
\eqen
at a deterministic rate depending only on $\beta $ and $\xi$. 
\end{lem}

We will first prove our upper bound for passage times, Lemma~\ref{prop-hull-time-metric}. The following lemma tells us how long it takes for the $f$-weighted FPP clusters to absorb a $\mu$-ball centered at a vertex in $\bdy \mcl V(A_\tau)$. It will be used to prove upper bounds for $f$-weighted FPP passage times in terms of the metric $D$ by, roughly speaking, considering a piecewise linear path covered by many small balls.

\begin{lem} \label{prop-hull-increment-time}
Let $ x \in \partial \BB D$ and $m\in\BB N$. Let $\tau$ be a stopping time for $\{\mcl F_t\}_{t\geq 0}$. Let $v_*$ be a vertex in $\partial\mcl V(A_\tau)$, chosen in some $\mcl F_\tau$-measurable manner. Let $\sigma$ be the smallest $t > 0$ for which each element of $\BB Z^d\cap B^\mu_m(v_*)$ belongs to $\mcl V(A_t)$. Also let $\theta > 1/2$. There is a constant $a > 0$, depending only on $f$, $\mu$, and $\theta$, such that the following is true. 
On the event $\{f(v_*) \geq 2 a |v_*|^{\alpha-1} m \}$, we have
\eqb \label{eqn-hull-increment-time}
\BB P\left(\sigma  - \tau \leq \frac{  m  + m^\theta   }{ f(v_*) - a |v_*|^{\alpha-1} m    } \,|\, \mcl F_\tau \right) = 1 - o_m^\infty(m)
\eqe 
at a deterministic rate (here we recall the notation $o_m^\infty(m)$ from Section~\ref{sec-asymp-notation}).  
\end{lem}
\begin{proof} 
Define the normalize edge passage times $\ol X^\tau_e$ for $e\in \mcl E(\BB Z^d)$ and the corresponding clusters $\ol A_{v,s}^\tau$ for $s\geq 0$ and $v\in\BB Z^d$ as in Definition~\ref{def-normalized-exp-rv}.  

Let $t_*$ be the smallest $s  > 0$ for which $   \BB Z^d\cap B^\mu_m(v_*)  \subset \mcl V( \ol A_{v_*,s}^\tau )$. 
By definition of FPP, for each $v\in  \BB Z^d\cap B^\mu_m(v_*)  $, there is a simple path $\eta_v$ in $\ol A_{v_*, t_* }^\tau$ with $\ol T^\tau(\eta_v) \leq t_* $ which connects $v_*$ to $v$.
By assertion~\ref{item-weight-to-normal} of Lemma~\ref{prop-normalized-compare},  
if we let $M_v :=  \min_{e\in\eta_v} f(m_e)$, where $m_e$ is the midpoint of the edge $e \in \mcl E(\BB Z^d)$,
then 
\eqb \label{eqn-ol-A-contain}
v \in  \mcl V\left(   A_{\tau +  M_v^{-1} t_*     }  \right).
\eqe 
Hence it suffices to prove an upper bound for $t_*$ and a lower bound for $\min_{v\in \left(   \BB Z^d\cap B^\mu_m(v_*) \right) \setminus \mcl V(A_\tau)}  M_v $.
 
To this end, let $v\in  \left(   \BB Z^d\cap B^\mu_m(v_*) \right) \setminus A_\tau$ be chosen in a $\mcl F_\tau$-measurable manner. By Lemma~\ref{prop-normalized-compare}, the conditional law given $\mcl F_\tau$ of $\{\ol A_{v_*,s}^\tau\}_{s \geq 0}$ is that of a standard FPP process starting from $v_*$ (i.e.\ with weight 1 at each edge). If $s>0$ and $t_* > s$, then  
$B_m^\mu(v_*) \cap \BB Z^d \not\subset \mcl V( \ol A_{v_*,s}^\tau )$. 
By the rate of convergence bound~\eqref{eqn-hull-lower'},
\eqb \label{eqn-hull-time-time}
\BB P\left(t_* >  m   +  m^\theta     \,|\, \mcl F_\tau  \right) 
= o_m^\infty(m) . 
\eqe 
By~\eqref{eqn-hull-upper}, 
\eqb \label{eqn-hull-time-diam}
\BB P\left( \sup \left\{  |u - v_*| \,:\, u \in \mcl V\left( \ol A_{v_*,    m   + m^\theta }^\tau \right)    \right\} > \ol \rho  m +   2  m^\theta    \,|\, \mcl F_\tau \right)  = o_m^\infty(m) .
\eqe 
By Lemma~\ref{prop-f-cont}, we can find a constant $a > 0 $ as in the statement of the lemma such that whenever $|u - v_*| \leq \ol \rho  m +   2  m^\theta$ and $f(v_*) \geq 2a |v_*|^{\alpha-1} m$,  
\eqb \label{eqn-hull-time-f}
f(u) \geq f(v_*) - a |v_*|^{\alpha-1} m .
\eqe 

By combining~\eqref{eqn-hull-time-time} and~\eqref{eqn-hull-time-diam}, we obtain that if $f(v_*) \geq 2 a |v_*|^{\alpha-1} m$, then with conditional probability $1-o_m^\infty(m)$ given $\mcl F_\tau$, we have $t_* \leq  m   +  m^\theta$ and $M_v \geq f(v_*) - a |v_*|^{\alpha-1} m $. 
By combining this with~\eqref{eqn-ol-A-contain} and a union bound over all $v\in\BB Z^d\cap B^\mu_m(v_*)$ we conclude.
\end{proof}

\begin{proof}[Proof of Lemma~\ref{prop-hull-time-metric}]
Let $v_* \in V_\tau^m$ be chosen in an $\mcl F_\tau$-measurable manner.
By definition of $\wt D_m$, there is a piecewise linear path $\gamma$ contained in $  B^{|\cdot|}_{m^\xi}(0) \setminus B^{|\cdot|}_{m^{1/\xi}}(0)$ which connects some element of $\mcl V(A_\tau)$ to $v$ and satisfies
\eqbn
\op{len}^D(\gamma) \leq \left( 1 +  m^{-100\xi}    \right) \wt D_m(v , A_\tau) .
\eqen  
Choose some such path $\gamma$ in a $\mcl F_\tau$-measurable manner. We set $T = \op{len}^\mu(\gamma)$ and take $\gamma$ to be parametrized by $\mu$-length (Definition~\ref{def-pl-path}). It follows from Lemma~\ref{prop-metric-compare} that $T$ is at most a constant times some power of $m$ (the constant and the exponent depend only on $f$ and $\xi$). 
We will show using Lemma~\ref{prop-hull-increment-time} that (roughly speaking) the amount of time it takes the FPP clusters to traverse $\gamma$ is not too much longer than the $D$-length of $\gamma$. 

We first construct a modified version of $\gamma$, which we call $\wt\gamma$, whose $D$-length is not too much larger than that of $\gamma$ and for which FPP passage times between points of $\wt\gamma$ are easier to estimate. 
To this end, fix $\zeta \in (0,1/\xi)$, to be chosen later. Let $s_0$ be the last time $s  \in [0,1]$ for which $\gamma(s) \in A_\tau^F$ (as defined in~\eqref{eqn-fatten-hull}). Let $v_0$ be the element of $\mcl V(A_\tau)$ closest to $\gamma(s_0)$ in the Euclidean norm. Inductively, if $k\in\BB N$ and $s_{k-1}$ and $v_{k-1}$ have been defined, let $s_k$ be the first time $s$ after $s_{k-1}$ for which $\gamma(s) \in \partial B^\mu_{m^\zeta}(\gamma(s_{k-1}))$, or $s_k = 1$ if no such $s$ exists. Also let $v_k$ be the element of $\BB Z^d$ closest to $\gamma(s_k)$ in the Euclidean norm. 
Let $k_*$ be the smallest $k\in\BB N$ for which $s_k = 1$ (and hence $v_k = v_{k_*}$).
Let $\wt\gamma $ be the piecewise linear path which is the concatenation of the line segments $  [v_{k-1} , v_k] $ for $k\in [1,k_*]_{\BB Z}$. 

We will now estimate $\op{len}^D(\wt\gamma)$. For each $k\in [1,k_*]_{\BB Z}$,  
\eqb \label{eqn-len-compare-upper}
\mu\left(v_k - v_{k-1} \right) \leq \mu\left(\gamma(s_k) - \gamma(s_{k-1}) \right) + C \leq s_k -s_{k-1} + C , 
\eqe 
where here $C>0$ is a deterministic constant depending only on $\mu$. Note that in the second inequality we have used that straight lines are geodesics for the metric induced by the norm $\mu$ and that $\gamma$ is parametrized by $\mu$-length. 
The sets $\gamma([s_{k-1} , s_k])$ and $ [v_{k-1} , v_k]$ are each contained in the Euclidean ball $B^{|\cdot|}_{\ol\rho m^\zeta  +C}(v_{k-1} )$, where $\ol\rho$ is the constant from Definition~\ref{def-angle-constant}. By Lemma~\ref{prop-f-cont}, for each $y \in \gamma([s_{k-1} , s_k]) \cup   [v_{k-1} , v_k]$, we have $|f(v_{k-1}  ) - f(  y)| \preceq |v_{k-1}|^{\alpha-1} m^\zeta$
with the implicit constant depending only on $f$.
Note that here we have used that $|y| \geq m^{1/\xi} - \ol\rho m^\zeta$ and that $\zeta < 1/\xi$. 
Therefore,
\eqb \label{eqn-sup-compare-upper}
\frac{ \sup_{y\in  [v_{k-1} , v_k]}  f(y)^{-1}  }{ \inf_{y\in \gamma([s_{k-1} , s_k])}  f(y)^{-1} } 
\leq  \frac{ f(v_{k-1} )  +  |v_{k-1}|^{\alpha-1} m^\zeta      }{    f(v_{k-1})  - |v_{k-1}|^{\alpha-1} m^\zeta   }  
\leq 1  + O_m( m^{\zeta - 1/\xi}) ,
\eqe   
at deterministic rate depending only on $\mu$ and $f$. 
By~\eqref{eqn-len-compare-upper} and~\eqref{eqn-sup-compare-upper} we find that
\begin{align} \label{eqn-len-bar-compare}
\op{len}^D(\wt\gamma) 
&\leq    \left( 1  + O_m( m^{\zeta - 1/\xi}) + O_m(m^{-\zeta}) \right)   \op{len}^D( \gamma) +  O_m\left(m^{ \zeta}\right) |v_*|^{-\alpha} \notag\\
&\leq \left( 1  + O_m( m^{\zeta - 1/\xi}) + O_m(m^{-\zeta}) \right) \wt D_m(v_* , A_\tau)  +  O_m\left(m^{ \zeta}\right) |v_*|^{-\alpha}
\end{align} 
where here the last term $O_m\left(m^{ \zeta}\right) |v_*|^{-\alpha}$ comes from the final segment $[v_{k_*-1} , v_{k_*}]$.

It remains to estimate the amount of time it takes for the clusters $A_t$ to traverse the marked vertices on the path $\wt\gamma$. Let $t_0 = \tau$ and for $k\in [1,k_*]_{\BB Z}$, let $t_k$ be the smallest $t > 0$ for which $v_k \in \mcl V(A_t)$. 
By Lemma~\ref{prop-hull-increment-time}, for any $\theta \in (1/2,1)$, there is a constant $a > 0$ such that it holds except on an event of conditional probability $o_m^\infty(m)$ given $\mcl F_\tau$ that 
\eqb \label{eqn-upper-inc-bound}
t_k  -t_{k-1} \leq  \frac{ m^\zeta  + m^{\zeta\theta} }{ f(v_{k-1}) -     a |v_{k-1}|^{\alpha-1} m^\zeta     }  ,\quad \forall k\in [1,k_* ]_{\BB Z}  .
\eqe  
We have $ \mu(v_k - v_{k-1}) = m^\zeta + O_m(1)$ for $k\in [1,k_*-1]_{\BB Z}$ and 
\eqbn
\sup_{y\in  [v_{k-1} , v_k]}  f(y)^{-1} \geq \left( f(v_{k-1}) +  |v_{k-1}|^{\alpha-1} m^\zeta \right)^{-1}        .
\eqen
Hence~\eqref{eqn-upper-inc-bound} implies that  
\eqb \label{eqn-inc-time-ratio}
\frac{t_k - t_{k-1}}{  \mu(v_k - v_{k-1})     \sup_{y\in [v_{k-1} , v_k]}  f(y)^{-1}   }  \leq  1 + O_m\left(m^{-\zeta(1-\theta)}\right) + O_m\left( m^{\zeta - 1/\xi}\right) .
\eqe 
If we choose $\zeta = \xi^{-1}(2-\theta)^{-1}$, then $-\zeta(1-\theta) = \zeta -1/\xi$ and $-\zeta < \zeta - 1/\xi$. If we are given $\beta \in (0,1/(3\xi))$ and we choose $\theta$ sufficiently close to $1/2$, then we can arrange that the sum of the error terms on the right side of~\eqref{eqn-inc-time-ratio} is $O_m(m^{-\beta })$, the sum of the error terms being multiplied by $\wt D_m(v_* , A_\tau)$ on the right side of~\eqref{eqn-len-bar-compare} is at most $O_m(m^{-\beta})$, and $\zeta < 1/\xi -\beta$. 
By~\eqref{eqn-upper-inc-bound}, we also have
\eqbn
t_{k_*} -t_{k_*-1} = O_m\left(m^{\zeta}\right)|v_*|^{-\alpha} = O_m(m^{1/\xi-\beta}) |v_*|^{-\alpha}     .
\eqen
By summing over all $k\in [1,k_*]_{\BB Z}$, we find that except on an event of conditional probability $o_m^\infty(m)$ given $\mcl F_\tau$,  
\eqbn
T(0,v_*) - \tau \leq   \left(1 + O_m(m^{-\beta}) \right)    \op{len}^D(\wt\gamma) + O_m\left(m^{1/\xi-\beta}\right) |v_*|^{-\alpha} .
\eqen
We conclude by combining this with~\eqref{eqn-len-bar-compare}, applying the union bound, and slightly increasing $\beta$. 
\end{proof}

We next prove our lower bound for FPP passage times.
For the proof of Lemma~\ref{prop-hull-expand-metric}, we need the following lemma to help us translate the estimates of Section~\ref{sec-standard-fpp} to a lower bound for weighted FPP passage times in terms of $D$. Roughly speaking, the lemma tells us that if $\tau$ is a stopping time for the $f$-weighted FPP clusters, then it is very unlikely that $A_{\tau+s}$ contains an edge whose $\mu$-distance to $A_\tau$ is too large.

\begin{lem} \label{prop-hull-increment-expand}
Let $\tau$ be a stopping time for $\{\mcl F_t\}_{t\geq 0}$. Also let $v_* \in \partial \mcl V(A_\tau)$ be chosen in a $\mcl F_\tau$-measurable manner. Fix $\theta \in ( 1/2 , 1)$ and for $s > 0$ and $R> 0$, let $F^{s,R}_\tau(v_*) $ be the event that there is an edge $e_* \in \mcl E(\BB Z^d) \setminus \mcl E(A_\tau)$ such that the following is true.
\begin{enumerate}
\item $e_* \in \mcl E(A_{\tau+s})$. \label{item-expand-contain}
\item Let $\eta_{e_*}$ be the FPP geodesic from 0 to $e_*$. Then $f(z) \leq R$ for each $e\in \eta_{e_*}$ and $z\in e$. \label{item-expand-ball}
\item $v_*$ is the last vertex in $\partial \mcl V(A_\tau)$ crossed by $\eta_{e_*}$. 
\item $e_* \not\subset B_{Rs + (Rs)^\theta}^\mu( v_* )$.   \label{item-expand-diam}
\end{enumerate}
Then for each $p > 0$, 
\eqbn
\BB P\left( F_\tau^{s,R }(v_*)  \,|\, \mcl F_\tau \right)  \preceq (R s)^{-p} ,
\eqen
with the implicit constant depending only on $p$, $\mu$, and $f$. 
\end{lem}

In the statement of Lemma~\ref{prop-hull-increment-expand}, one should think of $R$ as being large (if $\alpha < 0$), small (if $\alpha > 0$), or of constant order (if $\alpha=0$) and $s$ as being much larger than $1/R$, so that $Rs >> 1$. 

\begin{proof}[Proof of Lemma~\ref{prop-hull-increment-expand}]
Define the random variables $\ol X^\tau_e$ for $e\in \mcl E(\BB Z^d)$ and the clusters $\ol A_{v,s}^\tau$ for $s\geq 0$ and $v\in\BB Z^d$ as in Definition~\ref{def-normalized-exp-rv}. 

Suppose the event $F_\tau^{s,R }(v_*) $ occurs. Let $e_* \in \mcl E(\BB Z^d) \setminus \mcl E(A_\tau)$ and $\eta_{e_*}$ be as in the definition of $F_{\tau}^{s,R}(v_*)$. By conditions~\ref{item-expand-contain} and~\ref{item-expand-ball} in the definition of $F_\tau^{s,R }(v_*)$ together with assertion~\ref{item-weight-to-normal} of Lemma~\ref{prop-normalized-compare}, $e_* \in \mcl E(\ol A_{v_*  , R  s}^\tau)$. 
 By condition~\ref{item-expand-diam} in the definition of $F_\tau^{s,R } $, we therefore have
\eqbn
F_\tau^{s,R }(v_*)  \subset \left\{   \ol A_{v_*  , R s}^\tau \not\subset v_* + \left(R s + (R s)^\theta \right) \BB A \right\} .  
\eqen
Since the conditional law of $ \ol A_{v_*  , R s}^\tau$ given $\mcl F_\tau$ is that of a time-$R  s$ standard FPP cluster based at $v_*$ (Lemma~\ref{prop-normalized-compare}) we deduce the statement of the lemma from Kesten's upper bound~\eqref{eqn-hull-upper}.
\end{proof}

\begin{proof}[Proof of Lemma~\ref{prop-hull-expand-metric}]
The basic idea of the proof is to use an FPP geodesic to construct a piecewise linear path from $u_v$ to $v$ whose $D$-length is bounded above. 

We first define an event on which we have lower bounds for certain FPP passage times, building on the event of Lemma~\ref{prop-hull-increment-expand}. 
Fix $\wh\theta \in (0,1/2)$ and $\theta \in (1/2 , 1-\wh\theta)$. 
For $t >0$, $s > 0$, $R > 0$, and $v\in \mcl V(A_t)$, define the event $F^{s,R}_t(v)$ as in Lemma~\ref{prop-hull-increment-expand} with the above choice of $\theta$. Also fix $\zeta \in (0,1/\xi)$ 
and a constant $a > 0$ (to be chosen later, in a manner depending only on $\mu$ and $f$). For $v\in \BB Z^d$, let
\alb
E^m(v) :=   F_{T(0,v) \vee \tau }^{s , R}(v )^c \quad \op{for} \quad s = \frac{ \left(1- m^{-\wh\theta \zeta  } \right) m^\zeta}{  f(v) + a |v|^{\alpha-1} m^\zeta   } \quad \op{and} \quad R = f(v)+ a |v|^{\alpha-1} m^\zeta       .
\ale
Lemma~\ref{prop-hull-increment-expand} and the union bound imply that the event
\eqbn
E^m_*  := \left\{E^m (v) ,\, \forall v\in   \left( \partial \mcl V(A_\tau) \cup \left(  \BB Z^d \setminus \mcl V(A_\tau) \right) \right) \cap B_{2m^\xi}(0) \right\}
\eqen
has conditional probability $1-o_m^\infty(m)$ given $\mcl F_\tau$. 

Suppose now that $E^m_* $ occurs. Let $v_* \in \wh V_\tau^m $ be chosen in some $\mcl F_\tau$-measurable manner.
Let $\eta_{v_*}$ be the FPP geodesic from 0 to $v_*$, so that by definition of $\wh V_\tau^m $, we have $\eta_{v_*} \subset B_{m^\xi}^{|\cdot|}(0)$. 
Let $v_0 = u_{v_*}$ be as in the statement of the lemma.
Inductively, for $k\in \BB N$ let $v_k$ be the first vertex in $\mcl V(\BB Z^d)$ hit by $\eta_{v_*}$ after it hits $v_{k-1}$ which does not lie in $B_{m^\zeta}^\mu(v_{k-1})$; or $v_k = v_*$ if no such vertex exists. Let $k_*$ be the smallest $k\in\BB N$ for which $v_k =  v_*$. For $k\in[1,k_*]_{\BB Z}$, let $\eta_k$ be the segment of $\eta$ between $v_{k-1}$ and $v_k$. By definition, each $\eta_k$ is contained in $A_{T(0,v_k) } \cap B_{\ol \rho m^\zeta + C}^{|\cdot|}(v_{k-1})  $ for an appropriate $C>0$ depending only on $\mu$. By Lemma~\ref{prop-f-cont}, we can find a constant $a > 0$, depending only on $\mu$ and $f$, such that
\eqbn
 f(z) \leq  f(v_{k-1}) + a |v_{k-1}|^{\alpha-1} m^\zeta,\quad \forall e\in\eta_k \: \op{and} \: z\in e .
\eqen
We henceforth take this choice of $a$ in the definition of the events $E^m(v)$ above.

By definition of $E^m(v_{k-1}) $ (c.f.\ the definition of the event from Lemma~\ref{prop-hull-increment-expand}), we have for large enough $m$
\eqbn
T( 0 , v_k) - T(0 , v_{k-1}) \vee \tau
\geq  \frac{\left( 1 - m^{-\wh\theta\zeta}    \right) m^\zeta}{  f(v_{k-1}) + a |v_{k-1}|^{\alpha-1} m^\zeta    }  
,\quad \forall k \in [1,k_*-1]_{\BB Z} .
\eqen
Hence for large enough $m$,
\eqb \label{eqn-dist-split-lower}
T(0 , v_*) - \tau \geq  \sum_{k=1}^{k_*-1} \frac{\left( 1 - m^{-\wh\theta\zeta}    \right) m^\zeta}{  f(v_{k-1}) + a |v_{k-1}|^{\alpha-1} m^\zeta    }  .
\eqe  

Let $\gamma$ be the concatenation of the line segments $[v_{k-1} , v_k]$ for $k\in [1,k_*]_{\BB Z}$. Then $\gamma$ is a piecewise linear path from $u_{v_*}$ to $v_*$. Furthermore, each point of each segment $[v_{k-1} , v_k]$ lies within Euclidean distance $\ol\rho m^\zeta +C$ of $v_{k-1}$ (with $\ol\rho$ the constant from Definition~\ref{def-angle-constant}) so by Lemma~\ref{prop-f-cont},
\eqbn
 \sup_{y \in [v_{k-1} , v_k]} f(y)^{-1} \leq \left( f(v_{k-1}) -a' |v_{k-1}|^{\alpha-1} m^\zeta    \right)^{-1}   
\eqen
for appropriate deterministic $a' > 0$ depending only on $\mu$ and $f$. 
Hence the definition~\eqref{eqn-weighted-metric} of $D$ implies that
\eqb \label{eqn-sum-to-dist-lower}
\sum_{k=1}^{k_*} \frac{ \mu(v_k - v_{k-1}) }{   f(v_{k-1}) -a' |v_{k-1}|^{\alpha-1} m^\zeta     } \geq D(u_{v_*} , v_*) .
\eqe   
We have $\mu(v_k - v_{k-1}) \leq  m^\zeta + O_m(1)$ for $k\in [1,k_*]_{\BB Z}$ and 
\eqbn
\frac{ f(v_{k-1}) -a' |v_{k-1}|^{\alpha-1} m^\zeta     }{   f(v_{k-1})  + a |v_{k-1}|^{\alpha-1} m^\zeta    }  = 1 - O_m(m^{  \zeta - 1/\xi})  ,
\eqen
at a deterministic rate which does not depend on the particular choices of $v_*$, $k$, or the realization of our random variables. Note that here we use that each $v_k$ belongs to $B_{m^\xi}^{|\cdot|}(0) \setminus B_{m^{1/\xi}}^{|\cdot|}(0)$.  
Hence  
\eqbn
\frac{m^\zeta}{   f(v_{k-1})  + a |v_{k-1}|^{\alpha-1} m^\zeta } 
\geq \frac{\left(1 - O_m(m^{\zeta-1/\xi} ) \right) \mu(v_k - v_{k-1})}{  f(v_{k-1}) -a' |v_{k-1}|^{\alpha-1} m^\zeta        }  
,\quad \forall k\in [1,k_*-1]_{\BB Z} 
\eqen  
and
\eqbn
 \frac{   \mu(v_* - v_{k_*-1})}{ f(v_{k_*-1}) - a' |v_{k_*-1}|^{\alpha-1} m^\zeta }  = O_m(m^\zeta) |v_*|^{-\alpha} .
\eqen
By combining this with~\eqref{eqn-dist-split-lower} and~\eqref{eqn-sum-to-dist-lower}, we obtain that if $1/\xi - \zeta \geq \wh\theta \zeta$, then
\eqb \label{eqn-hull-expand-metric-end}
T(0 , v_*) - \tau \geq \left(1 - O_m(m^{-\wh\theta\zeta} ) \right) D(u_{v_*} , v_*)  -  O_m(m^\zeta) |v_*|^{-\alpha}  .
\eqe 
Now set $\zeta = \xi^{-1} (1+\wh\theta)^{-1}$, so that $1/\xi - \zeta = \wh\theta\zeta$. If we are given $\beta \in (0,1/(3\xi))$ and we choose $\wh\theta \in (0,1/2)$ sufficiently close to $1/2$, then we have $\wh\theta \zeta >  \beta $ and $\zeta < 1/\xi -  \beta $. 
Since our choice of $v_* \in \wh V_\tau^m$ was arbitrary, the desired estimate now follows from~\eqref{eqn-hull-expand-metric-end}.
\end{proof}

\begin{remark} \label{remark-truncated-expand}
The proof of Lemma~\ref{prop-hull-expand-metric} actually yields a slightly stronger but somewhat more complicated version of the statement of the lemma which we will need in Section~\ref{sec-cone-contain}. Suppose we are in the setting of Lemma~\ref{prop-hull-expand-metric}. Also let $\mcl U \subset \BB R^d$ be a deterministic open set and let $\mcl U_m$ be the set of $z \in \BB R^d$ which lie at Euclidean distance $<m^{1/\xi}$ from $\mcl U $. Define an internal version of the metric $D$ by
\eqbn
\wt D_{\mcl U_m}(z,w) :=  \inf\left\{\op{len}^D(\gamma) \,:\, \text{$\gamma$ is a piecewise linear path from $z$ to $w$ contained in $\mcl U_m$}\right\} ,\quad \forall z,w\in\mcl U_m .
\eqen
Let $\wh V_\tau^m(\mcl U )$ be the set of vertices $v\in \wh V_\tau^m$ such that the corresponding FPP geodesic $\eta_v$ satisfies $\eta_v\setminus \mcl E(A_\tau) \subset \mcl U$. Then 
\eqbn
\BB P\left(T(0,v) - \tau \geq \left(1-m^{- \beta } \right) \wt D_{\mcl U_m}(u_v , v) - \frac{m^{1/\xi - \beta}}{|v|^\alpha} ,\,  \forall  v \in \wh V_\tau^m(\mcl U)    \,|\, \mcl F_\tau\right) = 1 - o_m^\infty(m) .
\eqen
Indeed, this follows from the proof of Lemma~\ref{prop-hull-expand-metric} upon noting that, with $\gamma$ the piecewise linear path defined in the proof, for large enough $m$ (how large is deterministic and depends only on $f$, $\xi$, and $\beta$) we have $\gamma \subset \mcl U_m$ for each $v_* \in \wh V_\tau^m(\mcl U)$. Therefore, the estimate~\eqref{eqn-sum-to-dist-lower} holds with $\wt D_{\mcl U_m}$ in place of $D$. 
\end{remark}

\section{Proof of limit shape and covering results}
\label{sec-small-alpha}

\subsection{Proof of Theorem~\ref{thm-alpha<1}}
\label{sec-alpha<1}
 
In this subsection we will use the estimates of Section~\ref{sec-growth-estimates} to prove Theorem~\ref{thm-alpha<1}. For the proof, we use the setup of Theorem~\ref{thm-alpha<1}, so in particular we always assume $\alpha<1$ and we let $\{A_t\}_{t\geq 0}$ be the $f$-weighted FPP clusters and $\{\mcl F_t\}_{t\geq 0}$ be the associated filtration, as in Definition~\ref{def-fpp}.
 
 We also introduce the following additional notation. 
For $r \geq  0$, let
\eqb \label{eqn-alpha<1-stopping}
\tau_r := \inf\left\{t \geq 0 \,:\, A_t \not\subset B_r^D(0)\right\}  . 
\eqe 
For $0 \leq r' \leq r $, let
\eqb \label{eqn-alpha<1-induct-event}
G_{r,r'} := \left\{ B_{r'}^D(0) \cap \BB Z^d \subset \mcl V(A_{\tau_r}) \right\}    
\eqe 
be the event that the $f$-weighted FPP clusters fill in $B_{r'}^D(0)$ before time $\tau_r$. Then $\tau_r$ is a $\{\mcl F_t\}_{t\geq 0}$-stopping time and $G_{r,r'} \in \mcl F_{\tau_r}$.

The basic outline of the proof of Theorem~\ref{thm-alpha<1} is as follows. In Lemma~\ref{prop-alpha<1-time}, we will use the estimates of Section~\ref{sec-growth-estimates} to prove that if $\wt n \geq n$ with $\wt n \asymp n$, then with high conditional probability given $\mcl F_{\tau_n}$, it holds that $\tau_{\wt n} - \tau_n$ is not too far from $\wt n - n$. 
In Lemma~\ref{prop-alpha<1-induct}, we will use Lemma~\ref{prop-hull-expand-metric} to show that if $m\leq n \leq \wt n$ with $m \asymp n \asymp \wt n$, then on $G_{n,m}$ it holds with high conditional probability given $\mcl F_{\tau_n}$ that the event $G_{\wt n , \wt m}$ occurs for $\wt m $ slightly smaller than $m + \wt n - n $. Both of these two lemmas are proven using the estimates of Section~\ref{sec-growth-estimates}. Together with a straightforward induction argument, these lemmas imply that if $a>0$ is fixed and $G_{n , an}$ occurs for large enough $n$, then the event of Theorem~\ref{thm-alpha<1} occurs with high probability. To complete the proof of Theorem~\ref{thm-alpha<1}, we still need to show that for an appropriate choice of constant $a > 0$, we have $\BB P(G_{n , an}) =1-o_n^\infty(n)$. This is accomplished in Lemma~\ref{prop-alpha<1-small-ball}. 
 
We first record the following convenient fact, which is an immediate consequence of~\eqref{eqn-metric-ball-scale}.

\begin{lem} \label{prop-metric-ball-compare}
There is a constant $c > 1$, depending only on $\mu$ and $f$ such that
\eqb \label{eqn-metric-ball-compare}
 c^{-1} r^{\frac{1}{1-\alpha}} \BB D \subset r \BB B \subset  cr^{\frac{1}{1-\alpha}} \BB D ,\quad \forall r > 0 .
\eqe 
\end{lem}

We now use Lemma~\ref{prop-metric-ball-compare} and the estimates of Section~\ref{sec-growth-estimates} to prove some basic estimates for the $D$-ball exit times $\tau_n$. 
 
\begin{lem} \label{prop-alpha<1-time}
Fix $R \geq 2$, $\delta>0$, and $\chi $ as in~\eqref{eqn-rate-exponent}.
For $n\in\BB N$ and $\wt n\in [n  + R^{-1} n , R n ]_{\BB Z}$, we have (in the notation~\eqref{eqn-alpha<1-stopping})
\eqb \label{eqn-alpha<1-time-lower}
\BB P\left(     \tau_{\wt n } -\tau_n \geq \wt n-  n - \delta n^{1- \chi}      \,|\, \mcl F_{\tau_n}  \right) = 1 - o_n^\infty(n) .
\eqe 
Furthermore, for each $m \in [R^{-1} n , n]$, 
\eqb \label{eqn-alpha<1-time-upper}
\BB P\left(     \tau_{\wt n } -\tau_n \leq \wt n- m  + \delta n^{1- \chi}      \,|\, \mcl F_{\tau_n}  \right) \BB 1_{G_{n,m}} = \left( 1 - o_n^\infty(n) \right) \BB 1_{G_{n,m}}  .
\eqe 
The $o_n^\infty(n)$ errors above are deterministic and independent of the particular choices of $\wt n$ and $m$, but may depend on $\alpha$, $f$, and $R$. 
\end{lem}
\begin{proof}
First we consider the lower bound~\eqref{eqn-alpha<1-time-lower}. We first reduce to the case when the realization of $A_{\tau_n}$ is as large as possible. Let $U_n$ be the set of $u \in\BB Z^d \setminus B_n^D(0)$ such that $u$ is incident to a vertex in $  B_n^D(0) \cap \BB Z^d$. For $u\in U_n$, let $\frk A_u$ be the subgraph of $\BB Z^d$ whose vertex set is $ \left(\BB Z^d \cap B_n^D(0)\right)\cup \{u\} $ and whose edge set is the set of all edges in $\mcl E(\BB Z^d)$ which join vertices in its vertex set.
Almost surely, the set $A_{\tau_n}$ contains exactly one element of $U_n$ and no other elements of $\BB Z^d\setminus B_n^D(0)$. By Lemma~\ref{prop-hitting-time-law}, the conditional law of $\tau_{\wt n } -\tau_n$ given $\mcl F_{\tau_n}$ a.s.\ stochastically dominates the conditional law of $\tau_{\wt n} - \tau_n$ given $\{A_{\tau_n} = \frk A_u\}$ for some $u\in U_n$. Hence to prove~\eqref{eqn-alpha<1-time-lower} it suffices to show that
\eqb \label{eqn-alpha<1-time-lower'}
\BB P\left(     \tau_{\wt n } -\tau_n  \geq \wt n-  n - \delta n^{1- \chi}  \,|\, A_{\tau_n}  = \frk A_u \right) = 1- o_n^\infty(n) ,
\eqe 
uniformly over all choices of $u\in U_n$. 

To this end, let $u_{\wt n}$ be the (a.s.\ unique) element of $\mcl V(A_{\tau_{\wt n}})\setminus B_{\wt n}^D(0)$. By Lemma~\ref{prop-metric-ball-compare}, 
\[
D\left(u_{\wt n} ,\frk A_u \right)\geq \wt n - n + O_n\left(n^{-\frac{\alpha}{1-\alpha}}\right) ,\quad\forall u\in U_n.
\]
Furthermore, by Lemma~\ref{prop-metric-ball-compare}, on the event $\{A_{\tau_n} =\frk A_u\}$ for $u\in U_n$, the FPP geodesic $ \eta_{u_{\wt n}}$ from 0 to $u_{\wt n}$ satisfies
\eqbn
\eta_{u_{\wt n}} \setminus \mcl E(A_{\tau_n}) \subset B_{2\wt n}^D(0) \setminus B_n^D(0) \subset B_{C  n^{\frac{1}{1-\alpha}}}^{|\cdot|} (0)  \setminus B_{C^{-1}  n^{\frac{1}{1-\alpha}}}^{|\cdot|} (0) 
\eqen
for an appropriate constant $C>0$, depending only on $R$, $\mu$, and $f$. 
Therefore, Lemma~\ref{prop-hull-expand-metric} (applied with $n^{\frac{1}{1-\alpha}}$ in place of $m$ and $\xi$ slightly larger than $1$) implies that for each $\beta \in (0 , 1/3)$, the following is true. For each $u\in U_n$, it holds except on an event of conditional probability $1-o_n^\infty(n)$ given $\{ A_{\tau_n}  = \frk A_u \}$ (at a rate independent from $u$) that
\eqb \label{eqn-alpha<1-time-lower''}
\tau_{\wt n} - \tau_n = T(0 , u_{\wt n}) - \tau_n \geq \left(  1 - n^{-\frac{\beta }{1-\alpha}}  \right) \left(  \wt n -  n  - O_n\left(n^{-\frac{\alpha}{1-\alpha}} \right)\right) - O_n\left( n^{\frac{1-\beta - \alpha}{1-\alpha}} \right) \geq \wt n-   n - O_n\left( n^{  \frac{1-\beta  -\alpha}{1-\alpha}}   \right) ,
\eqe
provided $\beta$ is chosen sufficiently close to $1/3$. If we choose $\beta$ sufficiently close to $1/3$, then for large enough $n$ the error term on the right side of~\eqref{eqn-alpha<1-time-lower''} is smaller than $\delta n^{1-\chi}$. This proves~\eqref{eqn-alpha<1-time-lower'}. 

Now we turn our attention to the upper bound~\eqref{eqn-alpha<1-time-upper}. To this end, suppose $G_{n,m}$ occurs. We can choose $v\in \BB Z^d\cap B_m^D(0)$ and $v' \in \BB Z^d\setminus B_{\wt n}^D(0)$ (in some $\mcl F_{\tau_n}$-measurable manner) in such a way that
\eqb \label{eqn-inner-outer-ball-dist}
D\left(v , v' \right) \leq \wt n - m + O_n\left(n^{-\frac{\alpha}{1-\alpha}}\right) .
\eqe 
Since $G_{n,m}$ occurs, $v \in \mcl V(A_{\tau_n})$. 
We remark that for an arbitrary choice of $v \in \bdy B_m^D(0)$, there need not exist $v'\in \BB Z^d \setminus \bdy B_{\wt n}^D(0)$ for which~\eqref{eqn-inner-outer-ball-dist} holds; this is why we need to assume that $G_{n,m}$ occurs in~\eqref{eqn-alpha<1-time-lower}.

By definition, $T(0,v') \leq \tau_{\wt n}$. For each $\ep > 0$, we can find a piecewise linear path $\gamma$ connecting $v$ to $v'$ with
\eqbn
\op{len}^D(\gamma) \leq   D(v , v') + \ep \leq  \wt n - m + O_n\left(n^{-\frac{\alpha}{1-\alpha}}\right) +\ep  .
\eqen
We observe that for small enough $\ep$ and large enough $n$, this path $\gamma$ cannot enter $B_{m/2}^D(0)$ or exit $B_{2\wt n}^D(0)$. Indeed, if this were the case then we would have
\eqbn
\op{len}^D(\gamma ) \geq \wt n - m + \frac{m}{2} \wedge \wt n .
\eqen 
It therefore follows from Lemma~\ref{prop-hull-time-metric} that for each $\beta \in \left(0 ,1/3\right)$, it holds except on an event of conditional probability $1-o_n^\infty(n)$ given $\mcl F_{\tau_n}$ that 
\eqb \label{eqn-alpha<1-expand}
   T(0, v' ) - \tau_n \leq \left(1 +  n^{-\frac{\beta}{1-\alpha} } \right) \left(\wt n - m + O_n(n^{-\frac{\alpha}{1-\alpha}} )\right) + O_n\left( n^{   \frac{1-\beta - \alpha}{1-\alpha}} \right) = \wt n  - m + O_n\left(n^{ \frac{1-\beta-\alpha}{1-\alpha}}\right) .
\eqe  
By choosing $\beta$ sufficiently close to $1/3$, we conclude. 
\end{proof}

Our next lemma (plus an induction argument) will eventually tells us that if for some $n\in\BB N$, $\mcl V_{A_{\tau_n}}$ contains $B_m^D(0)\cap \BB Z^d$ for $m$ at least a constant times $n$, then with high probability the same is in fact true for all sufficiently large $n\in\BB N$. 

\begin{lem} \label{prop-alpha<1-induct} 
Fix $R \geq 2$, $\delta>0$, and $\chi  $ as in~\eqref{eqn-rate-exponent}.
For $n\in\BB N$, $m\in [ R^{-1}  n ,  n]_{\BB Z}$, and $\wt n\in [n  + R^{-1} n , R n ]_{\BB Z}$, set $\wt m :=    m + \wt n - n  - \delta n^{1-\chi} $. Then in the notation~\eqref{eqn-alpha<1-induct-event},  
\eqbn
\BB P\left(  G_{\wt n , \wt m}   \,|\, \mcl F_{\tau_n}  \right) \BB 1_{G_{n,m}} \geq \left( 1 - o_n^\infty(n) \right) \BB 1_{G_{n,m}} ,
\eqen
at a deterministic rate independent of the particular choices of $m$ and $\wt n$ but which may depend on $\alpha$, $f$, and $R$. 
\end{lem}
\begin{proof}
Assume $G_{n,m}$ occurs and consider a vertex $v\in (  B^D_{\wt m}(0) \cap \BB Z^d) \setminus \mcl V(A_{\tau_n})$, chosen in some $\mcl F_{\tau_n}$-measurable manner. Let $A_{\tau_n}^F$ be the fattening of $A_{\tau_n}$, as in~\eqref{eqn-fatten-hull}. Since $G_{n,m}$ occurs, $B_m^D(0) \subset A_{\tau_n}^F$ so
\eqb 
D\left(v , A_{\tau_n}^F \right) \leq \wt m - m  .
\eqe 
By definition of $D$, for each $\ep > 0$ there is a piecewise linear path $\wt\gamma$ connecting some point in $A_{\tau_n}^F$ to $v$ with $\op{len}^D(\wt \gamma) \leq   \wt m - m  +\ep $. By possibly replacing $\wt\gamma$ with its restriction to some interval of times, we can arrange that only the first point of $\wt \gamma$ belongs to $A_{\tau_n}^F$, so that (by definition of $G_{n,m}$) $\wt\gamma$ is disjoint from $B_{ m}^D(0)$. Furthermore, for small enough $\ep$, $\wt\gamma$ cannot exit $B_{2\wt m  }^D(0)$ for otherwise its $D$-length would be larger than $\wt m$. 
Let $u$ be an element of $\mcl V(A_{\tau_n})$ lying at minimal $D$-distance from the initial point of $\wt\gamma$, with ties broken in a $\mcl F_{\tau_n}$-measurable manner. By adding a line segment at the beginning of $\wt\gamma$, we obtain for each $\ep > 0$ a piecewise linear path $\gamma$ which connects some $u \in \mcl V(A_{\tau_n})$ to $v$, is contained in 
$B_{2\wt m  }^D(0)\setminus B_{m - O_n(n^{-\frac{\alpha}{1-\alpha}})}^D(0)$
 provided $G_{n,m}$ occurs, and satisfies 
\eqbn
\op{len}^D( \gamma) \leq \wt m - m   + O_n\left(n^{-\frac{\alpha}{1-\alpha}}\right)  + \ep   .
\eqen 
By Lemmas~\ref{prop-hull-time-metric} and~\ref{prop-metric-ball-compare}, 
for each $\beta \in \left(0 ,1/3\right)$, the following is true. Whenever $G_{n,m}$ occurs, it holds except on an event of conditional probability $o_n^\infty(n)$ given $\mcl F_{\tau_n}$ that
\eqb \label{eqn-alpha<1-expand'}
   T(0 , v ) - \tau_n \leq \left(1 +  n^{-\frac{\beta}{1-\alpha} } \right) \left(\wt m - m + O_n(n^{-\frac{\alpha}{1-\alpha}} )\right) + O_n\left( n^{   \frac{1-\beta - \alpha}{1-\alpha}} \right) = \wt m - m + O_n\left(n^{ \frac{1-\beta-\alpha}{1-\alpha}}\right) ,
\eqe 
for every possible choice of $v\in (  B^D_{\wt m}(0) \cap \BB Z^d) \setminus \mcl V(A_{\tau_n})$.
If we choose $\beta$ sufficiently close to $1/3$ then for large enough $n$ (how large is deterministic and independent of the particular choice of $\wt n$), the right side of~\eqref{eqn-alpha<1-expand'} is smaller than $\wt n - n - (\delta/2)n^{1-\chi}$. The statement of the lemma now follows from Lemma~\ref{prop-alpha<1-time}. 
\end{proof}

In order to deduce Theorem~\ref{thm-alpha<1} from Lemma~\ref{prop-alpha<1-induct}, we need to start with a large $n\in\BB N$, an integer $m\leq n$ with $m\asymp n$, and a realization of $\mcl F_{\tau_n}$ for which $G_{n,m}$ occurs and $\tau_n \asymp n$. Our next lemma will provide such a realization.

\begin{lem} \label{prop-alpha<1-small-ball}
There is a constant $a \in (0,1)$  (independent from $n$) such that for each $\delta > 0$ and each $\chi$ as in~\eqref{eqn-rate-exponent}, 
\eqbn
\BB P\left(  G_{n , an} \cap \left\{n - \delta n^{1- \chi} \leq  \tau_n    \leq a^{-1} n \right\}  \right) = 1-o_n^\infty(n) ,\quad \forall n\in\BB N .
\eqen
\end{lem}
\begin{proof}
Fix $R\geq 2$ and $\chi' \in (\chi , 1/3)$. 
Given $n\in\BB N$, let $n_0 = \lfloor n^{(1-\chi')/2}\rfloor$. We can select integers $m\in\BB N$ and $n_0 < n_1 < \dots < n_m = n$ with $n_k \in [(1+R^{-1}) n , R n]_{\BB Z}$ for each $k\in [1,m]_{\BB Z}$ and $m \preceq \log n$. 
By~\eqref{eqn-alpha<1-time-lower} of Lemma~\ref{prop-alpha<1-time} and the union bound, it holds except on an event of probability $1-o_{n}^\infty(n)$ that
\eqbn
 \tau_{n_k} - \tau_{n_{k-1}} \geq n_k - n_{k-1} -  n_{k-1}^{1-\chi'}, \: \forall k \in [1,m]_{\BB Z} .
\eqen
In this case, $\tau_n \geq n - n_0 - O_n( \log n) n^{1-\chi'}$,
which is at least $n - n^{1-\chi}$ for large enough $n$. Hence
\eqb \label{eqn-large-hit-time}
\BB P\left(\tau_n \geq n - n^{1-\chi}\right) = 1- o_n^\infty(n). 
\eqe

It remains to find an $a \in (0,1)$ as in the statement of the lemma such that with high probability $G_{n,an}$ occurs and $\tau_n    \leq a^{-1} n$. 
By Lemma~\ref{prop-metric-ball-compare}, we can find a $C>1$ depending only on $f$ such that for each $n\in\BB N$ and each $v\in B_n^D(0) \cap \BB Z^d $, we have that $v$ lies at graph distance at most $C n^{\frac{1}{1-\alpha}}$ and at least $C^{-1} n^{\frac{1}{1-\alpha}}$ from 0. Now fix such a $v$ and let $\eta$ be a simple path in $\BB Z^d$ from 0 to $v$ with $|\eta|$ minimal, so $|\eta| \asymp n^{\frac{1}{1-\alpha}}$.  
Since the function $f_0$ is bounded above and below by positive constants, there is a constant $c > 0$ depending only on $f$ such that the law of the random variables $X_{\eta(i)}$ for $i\in [1,|\eta|]_{\BB Z}$ is that of a collection of independent exponential random variables, each with parameter at least $c i^\alpha$. 
Therefore, the law of the passage time $ T(\eta)$ is stochastically dominated by the random variable 
\eqbn
\ol Y :=  \wt C \sum_{i=1}^{\lceil C n^{\frac{1}{1-\alpha}} \rceil} Y_j  
\eqen
where the $Y_j$'s are independent exponential random variables each with parameter $ i^\alpha$. We have 
$\BB E(\ol Y) \asymp \int_0^{ C n^{\frac{1}{1-\alpha}}}  t^{-\alpha} dt  \asymp n$
with the implicit constant depending only on $C, \wt C$, and $\alpha$. By elementary tail bounds for sums of exponential random variables (see~\cite[Theorem 5.1, item (i)]{janson-tail}) we can find a constant $a \in (0,1)$, depending only on $f$, such that
\eqbn
\BB P\left(\ol Y > a^{-1} n\right) \leq 
\begin{cases}
\exp\left(-n \right) ,\quad \alpha \in [0,1) \\
\exp\left( -n^{1 + \frac{\alpha}{1-\alpha}}\right) ,\quad \alpha < 0 . 
\end{cases}
\eqen
It follows that except on an event of probability $1-o_n^\infty(n)$, we have $T(\eta) \leq a^{-1} n$, so by a union bound except on an event of probability $1-o_n^\infty(n)$,
$ B_n^D(0) \cap\BB Z^d \subset  A_{a^{-1} n}  $.
In particular, except on an event of probability $o_n^\infty(n)$ we have $\tau_n \leq a^{-1} n$ and by~\eqref{eqn-large-hit-time}, $B_{a' n}^D(0) \cap \BB Z^d \subset A_{\tau_n}$ for $a'$ slightly smaller than $a$. This proves the statement of the lemma with $a'$ in place of $a$. 
\end{proof}

\begin{proof}[Proof of Theorem~\ref{thm-alpha<1}]
Let $a$ be the constant from Lemma~\ref{prop-alpha<1-small-ball}. Also fix $R > 3\vee a^{-1}$, let $\chi $ be as in~\eqref{eqn-rate-exponent}, and let $\delta >0$ to be chosen later, depending only on $R$ and $\chi$. 

For $n,m \in \BB N$ with $m\leq n$, let $\tau_n$ and $G_{n,m}$ be as in~\eqref{eqn-alpha<1-stopping} and~\eqref{eqn-alpha<1-induct-event}. For $n , n_0 \in \BB N$ with $n_0 \leq n$ and $k\in\BB N \cup \{0\}$, let
\eqb
\wh G_{n_0 , n}^k := G_{n,n-n^{1-\chi} - (1-a) n_0}  \cap \left\{ n-n^{1-\chi} \leq  \tau_n  \leq n + \delta^{-1} n^{1-\chi} + a^{-1} n_0 + k(1-a) n_0 \right\} .
\eqe
We first claim that for an appropriate choice of $\delta$, it holds for each $n_0 \in \BB N$, $n\geq n_0$, and $k\in\BB N$ that
\eqb \label{eqn-alpha<1-all-scale}
\BB P\left(\bigcap_{\wt n = \lceil n+ R^{-1} n \rceil}^{\lfloor R n \rfloor} \wh G_{n_0 , \wt n}^k \,|\, \mcl F_{\tau_n} \right) \BB 1_{\wh G_{n_0,n}^{k-1}} =\left(1-o_n^\infty(n) \right) \BB 1_{\wh G_{n_0,n}^{k-1}} ,
\eqe 
at a deterministic rate independent from $n_0$.
To see this, we first apply Lemmas~\ref{prop-alpha<1-time},~\ref{prop-alpha<1-induct}, and the union bound to find that if $n\geq n_0$ and $\wh G_{n_0,n}^{k-1}$ occurs, then except on an event of conditional probability $o_n^\infty(n)$ given $\mcl F_{\tau_n}$, it holds for each $\wt n\in [n  +R^{-1} n , R n]_{\BB Z}$ that 
\begin{align} \label{eqn-end-induct-event}
&\qquad\qquad\text{$G_{\wt n ,   \wt n - (1+\delta) n^{1-\chi} - (1-a) n_0  }$ occurs and} \notag \\
&\text{$\wt n  - (1+\delta)  n^{1-\chi}    \leq  \tau_{\wt n }   \leq \wt n + (\delta^{-1} + 2) n^{1-\chi} + a^{-1} n_0 + k(1-a) n_0   $}   .
\end{align} 
If $\delta > 0$ is chosen sufficiently small depending only on $R$ and $\chi$, then 
\alb
\wt n^{1-\chi} &\geq \left(1+R^{-1}\right)^{1-\chi} n^{1-\chi} \geq (1+\delta) n^{1-\chi}     \quad \op{and}   \\
\delta^{-1} \wt n^{1-\chi} &\geq   \delta^{-1} \left(1+R^{-1} \right)^{1-\chi} n^{1-\chi} \geq (\delta^{-1} + 2) n^{1-\chi}        .
\ale
Therefore, $\wh G_{n_0,\wt n}^k$ is contained in the event~\eqref{eqn-end-induct-event}. 
This proves~\eqref{eqn-alpha<1-all-scale}. 

By~\eqref{eqn-alpha<1-all-scale} and induction, we infer that for each $n_0 \in \BB N$,  
\eqb \label{eqn-alpha<1-large-n}
\BB P\left(\bigcap_{n = \lceil (1+R^{-1}) n_0 \rceil}^\infty \wh G_{n_0 , n}^{k_n} \,|\, \mcl F_{\tau_{n_0}}\right) \BB 1_{\wh G_{n_0,n_0}^0} = (1-o_{n_0}^\infty(n_0)) \BB 1_{\wh G_{n_0,n_0}^0} ,
\eqe 
where
\eqbn
k_n = \left\lfloor \frac{\log n}{\log (1+R^{-1})}  \right\rfloor .
\eqen
By Lemma~\ref{prop-alpha<1-small-ball}, 
\eqb  \label{eqn-alpha<1-some-n}  
\BB P\left( \wh G_{n_0 , n_0}^0 \right)  =  1-o_{n_0}^\infty(n_0)  .
\eqe  
By combining~\eqref{eqn-alpha<1-large-n} and~\eqref{eqn-alpha<1-some-n} (the latter applied with $\lfloor (1+R^{-1})^{-1} n_0\rfloor$ in place of $n_0$), we obtain
\eqb \label{eqn-alpha<1-all-n}
\BB P\left(\bigcap_{n = \lceil  n_0 \rceil}^\infty \wh G_{n_0 , n}^{k_n} \right) = 1-o_{n_0}^\infty(n_0) .
\eqe 
 
Now suppose that $t_0 > 0$. Set $n_0 = \lfloor  2^{-1} t_0^{(1-\chi)/2} \rfloor$ so that by~\eqref{eqn-alpha<1-large-n}, it holds except on an event of probability $ o_{t_0}^\infty(t_0)$ that the event $\wh G_{n_0 , n}^{k_n+1} $ occurs for each $n\geq n_0$. Let $t\geq t_0$ and let $n \in \BB N$ be chosen so that $t \in [\tau_n , \tau_{n+1}]$. For large enough values of $t_0$, we have $(\log n + 1) n_0 \leq n^{1- \chi}$. By definition of $\wh G_{n_0,n}^{k_n}$, 
\eqbn
B^D_{n- O_n(n^{1-\chi})}(0)\cap\BB Z^d \subset \mcl V(A_t) \subset B^D_{n+O_n(1) }(0)\cap\BB Z^d
\eqen
and
\eqbn
n- O_n(n^{1-\chi} ) \leq t \leq  n  + O_n(n^{1-\chi}) ,
\eqen
with the $O_n(\cdot)$ deterministic and depending only on $n$, $\delta$, and $\chi$. 

Therefore, for an appropriate constant $C' >0$, depending only on $\delta$ and $\chi$,  
\eqbn
B^D_{t - C' t^{1-\chi}}(0)\cap\BB Z^d \subset \mcl V(A_t) \subset B^D_{t + C' t^{1- \chi} }(0)\cap\BB Z^d .
\eqen
By~\eqref{eqn-metric-ball-scale}, for a possibly larger constant $C'$,
\eqbn
\left(1 - C' t^{-\chi}\right)t^{\frac{1}{1-\alpha}} \BB B \subset A_t^F \subset \left(1 + C' t^{-\chi}\right)t^{\frac{1}{1-\alpha}} \BB B .
\eqen
We conclude by slightly increasing $\chi$ and recalling~\eqref{eqn-alpha<1-all-n}.  
\end{proof}

\subsection{Proof of Theorem~\ref{thm-alpha-near-1}}
\label{sec-alpha-near-1}

Throughout this section, we assume that we are in the setting of Theorem~\ref{thm-alpha-near-1}. In particular, we let $\sigma_r$ for $r>0$ be the exit time from the Euclidean ball of radius $r$ centered at 0, as in~\eqref{eqn-nu-ball-hit-time}. 

We note that Lemma~\ref{prop-metric-compare} implies that if $\alpha \geq 1$, then for any $q > r > 1$, then
\eqb \label{eqn-ball-dist-lower}
D(q\bdy \BB D , r\bdy \BB D) \geq 
\begin{dcases}
\frac{  r^{1-\alpha} -  q^{1-\alpha} }{\ol\rho \ol\kappa ( \alpha-1)} ,\quad &\alpha > 1 \\
\ol\rho^{-1} \ol\kappa^{-1} \log \left(\frac{ q}{ r}\right) ,\quad &\alpha =1  .
\end{dcases}
\eqe 
Furthermore, by the definition~\eqref{eqn-ball-norm-diam} of $\lambda$ together with Lemma~\ref{prop-metric-scale}, 
\eqb \label{eqn-ball-D-diam}
\sup_{z,w\in \in \bdy \BB D} D(z,w) \leq \lambda r^{1-\alpha} ,\qquad \forall r > 0 .
\eqe 
Roughly speaking, the proof of Theorem~\ref{thm-alpha-near-1} proceeds as follows. Lemma~\ref{prop-hull-expand-metric} and~\eqref{eqn-ball-dist-lower} imply that if $R$ is sufficiently large, then for $n\in\BB N$ it is typically the case that $\sigma_{R n} - \sigma_n $ is not too much smaller than $(\alpha-1)^{-1} \ol\rho^{-1} \ol\kappa^{-1} n^{1-\alpha}$. If $(\alpha-1)^{-1} \ol\rho^{-1} \ol\kappa^{-1} < \lambda$, then Lemma~\ref{prop-hull-time-metric} and~\eqref{eqn-ball-D-diam} imply that with high probability, the clusters $A_t$ absorb every vertex of $\BB Z^d \cap \left(  B_n^{|\cdot|}(0) \setminus B_{n-1}^{|\cdot|}(0)\right)$ between times $\sigma_n$ and $\sigma_{R n}$. Sending $n\rta\infty$ concludes the proof. We now proceed with the details. 

\begin{lem} \label{prop-alpha-near-1-time}
Fix $R \geq 2$ and $\beta \in \left(0 , 1/3\right)$. Suppose $n\in\BB N$ and $\wt n\in \left[(1+R^{-1})n , R n\right]_{\BB Z}$. If $\alpha > 1$, then 
\eqb \label{eqn-alpha-near-1-time}
\BB P\left( \sigma_{\wt n} - \sigma_n  \geq \frac{n^{1-\alpha} - \wt n^{1-\alpha} - n^{1-\alpha-\beta} }{\ol\rho \ol\kappa (\alpha-1)} \,|\, \mcl F_{\sigma_n} \right) = 1-o_n^\infty(n) ,
\eqe 
at a deterministic rate independent of the particular choice of $\wt n$. If $\alpha  = 1$, we instead have
\eqb \label{eqn-alpha=1-time}
\BB P\left( \sigma_{\wt n} - \sigma_n  \geq ( \ol\rho \ol\kappa)^{-1} \log\left(\frac{\wt n}{n} \right)  + n^{-\beta} \,|\, \mcl F_{\sigma_n} \right) = 1-o_n^\infty(n) ,
\eqe 
at a deterministic rate independent of the particular choice of $\wt n$.
\end{lem}
\begin{proof}
We treat the case where $\alpha > 1$; the case where $\alpha=1$ is treated similarly. 
Let $U_n$ be the set of $u \in\BB Z^d \setminus B_n^{|\cdot|}(0)$ such that $u$ is incident to a vertex in $ B_n^{|\cdot|}(0) \cap \BB Z^d$. For $u\in U_n$, let $\frk A_u$ be the subgraph of $\BB Z^d$ whose vertex set is $ \left(\BB Z^d \cap B_n^{|\cdot|}(0)\right)\cup \{u\} $ and whose edge set is the set of all edges in $\mcl E(\BB Z^d)$ which join vertices in its vertex set.
Almost surely, the set $A_{\sigma_n}$ contains exactly one element of $U_n$ and no other elements of $\BB Z^d\setminus B_n^{|\cdot|}(0)$. By Lemma~\ref{prop-hitting-time-law}, the conditional law of $\sigma_{\wt n } -\sigma_n$ given $\mcl F_{\sigma_n}$ a.s.\ stochastically dominates the conditional law of $\sigma_{\wt n} - \sigma_n$ given $\{A_{\sigma_n} = \frk A_u\}$ for some $u\in U_n$. Hence to prove~\eqref{eqn-alpha-near-1-time}, it suffices to show that
\eqb \label{eqn-alpha-near-1-time'}
\BB P\left(   \sigma_{\wt n } -\sigma_n   \geq \frac{n^{1-\alpha} - \wt n^{1-\alpha} - n^{(1-\beta)(1-\alpha)} }{\ol\rho  \ol\kappa (\alpha-1)}      \,|\, A_{\sigma_n}  = \frk A_u \right) = 1- o_n^\infty(n) ,
\eqe 
uniformly over all choices of $u\in U_n$.  

To prove~\eqref{eqn-alpha-near-1-time'}, let $v_{\wt n}$ be the (a.s.\ unique) element of $\mcl V(A_{\sigma_{\wt n}})\setminus B_{\wt n}^{|\cdot|}(0)$. By~\eqref{eqn-ball-dist-lower},
\eqbn
D\left(v_{\wt n} , \frk A_u\right) \geq \frac{n^{1-\alpha} - \wt n^{1-\alpha} }{\ol\rho \ol\kappa (\alpha-1)} - O_n\left(n^{-\alpha}\right) ,\quad \forall u\in U_n .
\eqen
Furthermore, on the event $\{A_{\sigma_n} =\frk A_u\}$ for $u\in U_n$, the FPP geodesic $ \eta_{v_{\wt n}}$ from 0 to $v_{\wt n}$ satisfies
\eqbn
\eta_{v_{\wt n}} \setminus \mcl E(A_{\sigma_n}) \subset B_{(R+1)n}^D(0) \setminus B_n^{|\cdot|}(0) .
\eqen 
Therefore, Lemma~\ref{prop-hull-expand-metric} (applied with $n$ in place of $m$ and $\xi$ slightly larger than $1$) implies that for each $\beta \in (0 , 1/3)$, the following is true. For each $u\in U_n$, it holds except on an event of conditional probability $1-o_n^\infty(n)$ given $\{ A_{\sigma_n}  = \frk A_u \}$ (at a rate independent from $u$) that
\begin{align} 
\sigma_{\wt n} - \sigma_n 
= T(0 , v_{\wt n}) - \sigma_n 
&\geq \left(  1 - n^{-\beta}  \right) \left(  \frac{ \wh n^{1-\alpha} - n^{1-\alpha}}{\ol\rho \ol\kappa (\alpha-1)}  - O_n\left(n^{-\alpha} \right)\right) - O_n\left( n^{ 1-\beta-\alpha}  \right) \notag \\ 
&\geq\frac{ \wh n^{1-\alpha} - n^{1-\alpha}- O_n\left( n^{  1-\alpha-\beta}   \right)    }{\ol\rho \ol\kappa(\alpha-1)}.
\end{align}
We obtain~\eqref{eqn-alpha-near-1-time'} by slightly increasing $\beta$, which completes the proof of~\eqref{eqn-alpha-near-1-time}.  
\end{proof}

The following lemma tells us that vertices of $\BB Z^d$ sufficiently close to $A_{\sigma_n}$ are likely to be absorbed by the FPP clusters before time $\sigma_{\wt n}$, for $\wt n \geq n$ with $\wt n \asymp n$.

\begin{lem} \label{prop-alpha-near-1-swallow} 
Fix $R \geq 2$ and $\beta \in \left(0, 1/3\right)$. Suppose $n\in\BB N$ and $\wt n\in \left[(1+R^{-1})n , R n\right]_{\BB Z}$. For $\alpha > 1$, let
\eqb \label{eqn-alpha-near-1-set} 
V_{n,\wt n} := \left\{v\in \BB Z^d\cap \left( B_n^{|\cdot|}(0) \setminus B_{n/2}^{|\cdot|}(0) \right) \,:\,  D(v , A_{\sigma_n}) \leq \frac{n^{1-\alpha} - \wt n^{1-\alpha} - n^{1-\alpha-\beta} }{\ol\rho \ol\kappa (1-\alpha)} \right\} .
\eqe
For $\alpha=1$, instead let
\eqb \label{eqn-alpha=1-set} 
V_{n,\wt n} := \left\{v\in \BB Z^d\cap \left( B_n^{|\cdot|}(0) \setminus B_{R^{-1} n}^{|\cdot|}(0) \right) \,:\,  D(v , A_{\sigma_n}) \leq \ol\rho^{-1} \ol\kappa^{-1} \log \left(\frac{\wt n}{n}\right) - n^{-\beta} \right\} .
\eqe 
Then  
\eqb \label{eqn-alpha-near-1-swallow} 
\BB P\left(V_{n,\wt n} \subset \mcl V(A_{\sigma_{\wt n}}) \,|\, \mcl F_{\sigma_n}\right) = 1-o_n^\infty(n)  ,
\eqe 
at a deterministic rate independent of the particular choice of $\wt n$. 
\end{lem}
\begin{proof}
We treat the case where $\alpha > 1$; the case where $\alpha=1$ is treated similarly. 
Let $v\in V_{n,\wt n}$ be chosen in some $\mcl F_{\sigma_n}$-measurable manner. For each $\ep > 0$, we can find a piecewise linear path $\gamma$ connecting $v$ to some vertex in $\mcl V(A_{\sigma_n})$ with
\eqb \label{eqn-alpha-near-1-geodesic}
\op{len}^D(\gamma) \leq   \frac{n^{1-\alpha} - \wt n^{1-\alpha} - n^{1-\alpha-\beta} + O_n(n^{-\alpha}) }{\ol\rho \ol\kappa (1-\alpha)} + \ep.
\eqe 
Observe that for small enough $\ep$ and large enough $n$, $\gamma$ cannot exit $B_{2\wt n}^{|\cdot|}(0)$. Indeed, if it did, then by~\eqref{eqn-ball-dist-lower} we would have
\eqbn
\op{len}^D(\gamma ) \geq \frac{n^{1-\alpha} - (2\wt n)^{1-\alpha}}{\ol\rho \ol\kappa (1-\alpha)}, 
\eqen
which is larger than the right side of~\eqref{eqn-alpha-near-1-geodesic} for large enough $n$ and small enough $\ep$.
On the other hand, if $\delta \in (0,R^{-1})$ and $\gamma$ enters $B_{\delta n}^{|\cdot|}(0)$, then  
\eqbn
\op{len}^D(\gamma ) \geq \frac{(\delta n)^{1-\alpha} - (R^{-1} n)^{1-\alpha}}{\ol\rho \ol\kappa (1-\alpha)}, 
\eqen
which is larger than the right side of~\eqref{eqn-alpha-near-1-geodesic} for large enough $n$ and small enough $\ep$ provided $\delta$ is chosen sufficiently small, depending only $R$. 

It therefore follows from Lemma~\ref{prop-hull-time-metric} that for each $\beta' \in \left(\beta , 1/3\right)$, it holds except on an event of conditional probability $1-o_n^\infty(n)$ given $\mcl F_{\sigma_n}$ that 
\begin{align} \label{eqn-alpha-near-1-expand}
T(0, v  ) - \sigma_n 
&\leq \left(1 +  n^{-\beta'} \right)  \frac{n^{1-\alpha} - \wt n^{1-\alpha} - n^{1-\alpha-\beta} + O_n(n^{-\alpha}) }{\ol\rho \ol\kappa (1-\alpha)} + n^{1-\beta'-\alpha} \notag \\
&\leq \frac{n^{1-\alpha} - \wt n^{1-\alpha} - n^{1-\alpha-\beta}   }{\ol\rho  \ol\kappa(1-\alpha)} + O_n(n^{1-\alpha-\beta'}) .
\end{align} 
For large enough $n$ (how large is deterministic and independent from $v$), the right side of this last inequality is smaller than our lower bound for $\sigma_{\wt n} - \sigma_n$ from Lemma~\ref{prop-alpha-near-1-time}. Hence $v\in \mcl V(A_{\sigma_{\wt n}})$ except on an event of conditional probability $1-o_n^\infty(n)$ given $\mcl F_{\sigma_n}$. We conclude by means of the union bound.  
\end{proof}

\begin{proof}[Proof of Theorem~\ref{thm-alpha-near-1}]
The statement of the theorem is immediate from Theorem~\ref{thm-alpha<1} in the case where $\alpha<1$, so we can assume without loss of generality that $\alpha \in [1,1 + (\ol\rho \ol\kappa \lambda)^{-1}] $. 
Fix $R \geq 2$ and for $n , \wt n \in \BB N$, let $V_{n,\wt n}$ be as in Lemma~\ref{prop-alpha-near-1-swallow}. Also let $v_n$ be the (a.s.\ unique) point of $A_{\sigma_n} \setminus B_n^{|\cdot|}(0)$. By~\eqref{eqn-ball-D-diam}, if $v\in \BB Z^d$ and we let $n_v = \lceil |v| \rceil$, then 
\eqbn
D(v_{n_v} , v) \leq \lambda n_v^{1-\alpha} + O_{n_v}(n_v^{-\alpha} ).
\eqen
By our assumption on $\alpha$, we have $\lambda < \frac{1}{\ol\rho \ol\kappa(\alpha-1)}$. Therefore, we can find $n_* \in \BB N$ and $R \geq 2$ (depending only on $\mu$, $\nu$, and $\alpha$) such that if $n_v \geq n_*$ and $\wt n_v := \lfloor R n_v \rfloor$, then $v\in V_{n_v , \wt n_v}$. By Lemma~\ref{prop-alpha-near-1-swallow}, we obtain~\eqref{eqn-nu-annulus-swallow}. The second assertion follows from the first assertion and the Borel-Cantelli lemma.
\end{proof}

\section{Proof of cone containment result}
\label{sec-cone-contain}

In this section we will prove Theorem~\ref{thm-alpha>1}.
Throughout this section, we always assume $\alpha>1$. 

In Section~\ref{sec-alpha>1-setup}, we will define the class of $\alpha$-weight functions $f$ which we will consider (which in particular includes the $\alpha$-th powers of a certain family of norms) and state a more quantitative version of Theorem~\ref{thm-alpha>1} (namely Theorem~\ref{thm-cone-contain}). 
We will give an outline of the content of the rest of this section just after the statement of Theorem~\ref{thm-cone-contain}. 

We remark that the main difficulty in the proof of Theorem~\ref{thm-alpha>1} is geometric, rather than probabilistic. In particular, we do not have good estimates for the deterministic metric $D$ of~\eqref{eqn-weighted-metric} unless $f$ takes a rather specific form. The primary reason for this problem is that very little is known about the Eden model limit shape $\BB A$ and the corresponding metric $\mu$.

\subsection{Cylindrical convex sets and admissible weight functions}
\label{sec-alpha>1-setup}

In this section we will define the class of $\alpha$-weight functions for which we will prove our cone containment result. 
We start by defining the set of norms whose $\alpha$-th powers are contained in this class. See Figure~\ref{fig-big-alpha-norm} for an illustration of the unit ball of such a norm when $d =2$. 

Recall the definition of the constant $\ol\rho$ and the set $\BB X$ of maximal $\mu$-unit vectors from Definition~\ref{def-angle-constant}.
For $\BB x \in \BB X$, let $P_{\BB x}$ be the $d-1$-hyperplane containing $\BB x$ which is perpendicular to the line through 0 and $\BB x$. 
Note that $P_{\BB x}$ intersects $B_{\ol\rho}^{|\cdot|}(0)$ only at $\BB x$ and $\BB A\subset B_{\ol\rho}^{|\cdot|}(0)$, so $P_{\BB x} \cap \BB A = \{\BB x\}$. 
Let $P_{\BB x}^0 := P_{\BB x} - \BB x$ be the hyperplane through 0 perpendicular to the line through 0 and $\BB x$.  

Fix $\BB x\in\BB X$ and let $Q$ be a compact convex subset of $P_{\BB x}^0$ which contains $B_{\ol\rho}^{|\cdot|}(0) \cap P_{\BB x}^0$ and is symmetric about the origin. For $s  > 1$, let
\eqb \label{eqn-cylinder-def}
\mcl Q_s := \left\{s z + t \BB x \,:\, z \in Q ,\, t \in [-1,1]\right\} 
\eqe 
be the cylinder of Euclidean height $2\ol\rho$ over $s Q$. The set $\mcl Q_s$ is compact, convex, and symmetric about the origin 
so 
\eqb \label{eqn-cylinder-norm}
\nu_s(z) := \inf\left\{r > 0 \,:\, z \in r \mcl Q_s \right\} \quad \forall z\in\BB R^d
\eqe 
defines a norm on $\BB R^d$ whose unit ball is $\mcl Q_s$. 

We note that the set $\bdy \mcl Q_s = \bdy B_1^{\nu_s}(0)$ possesses two distinguished flat faces, namely $\bdy \mcl Q_s \cap P_{\BB x}$ and $\bdy \mcl Q_s \cap P_{-\BB x}$, which are reflections of each other through the origin. Due to our choice of $\BB x$ and since $s > 1$, the set $\BB A\cap \bdy \mcl Q_s$ consists of two points, one of which belongs to each of these two distinguished flat faces.  

Let $f$ be an $\alpha$-weight function and for $s > 1$ let $f_s := f|_{\bdy \mcl Q_s}$. Then 
\eqb \label{eqn-f-convex-decomp}
f(z) = \nu_s(z)^{ \alpha} f_s\left(\frac{z}{\nu_s(z)}\right) , \quad \forall z\in\BB R^d
\eqe 
so we can represent $f$ by means of the parameters $f_s$ and $\alpha$, rather than $f_0$ and $\alpha$ from~\eqref{eqn-f-decomp}. Note that $f_s $ is Lipschitz continuous if and only if $f_0$ is Lipschitz continuous.  

\begin{defn} \label{def-max-cond}
For $f_s$ as above, write
\eqb \label{eqn-f-max-convex}
\ol\kappa_s := \sup_{z\in\bdy \mcl Q_s} f_s(z) \quad \op{and} \quad \ul\kappa_s := \inf_{z\in\bdy \mcl Q_s} f_s(z) .
\eqe
We say that $f_s : \bdy \mcl Q_s \rta (0,\infty)$ is \emph{admissible} if $f_s$ is Lipschitz continuous and $f_s \equiv \ol\kappa_s$ on $\bdy \mcl Q_s \cap \left(P_{\BB x} \cup P_{-\BB x}\right)$. 
\end{defn} 

Henceforth fix $s>1$ and an admissible function $f_s$ and let $f$ as in~\eqref{eqn-f-convex-decomp}. 
Note that constant functions are admissible in the sense of Definition~\ref{def-max-cond}, so we can take $f$ to be the $\alpha$-th power of the norm $\nu_s$. 
Let $\{A_t\}_{t\geq 0}$ be the $f$-weighted FPP clusters, as in Section~\ref{sec-fpp-setup}. 
The main goal of this section is to prove the following theorem, which immediately implies Theorem~\ref{thm-alpha>1}. 
  
\begin{thm} \label{thm-cone-contain}
Suppose $\alpha >1$, $s>1$, $Q$, $\mcl Q_s$, and $f_s$ are as above.
Let 
\eqbn
\mcl K:= \bigcup_{r > 0} r \left(\bdy\mcl Q_s \cap P_{\BB x} \right)  
\eqen
and note that $\mcl K$ is contained in a Euclidean cone of opening angle $<\pi$. If
\eqb \label{eqn-pos-prob-cond}
s > 2^{ \frac{\alpha}{\alpha-1}}  -1     ,
\eqe
then 
\eqb \label{eqn-pos-prob}
\BB P\left(\# \left(\mcl V(A_{\tau_\infty}) \setminus \mcl K\right) < \infty \right)  > 0 \quad \op{and} \quad
\BB P\left(\# \left(\mcl V(A_{\tau_\infty}) \setminus (-\mcl K) \right) < \infty \right)  > 0  .
\eqe 
If, in addition,  
\eqb \label{eqn-as-cond}
s > 1 + \frac{\ol\kappa_s \alpha^\alpha }{\ul\kappa_s  (\alpha-1)^{\alpha-1}} 
\eqe
then a.s.\ either
\eqb \label{eqn-as}
 \# \left(\mcl V(A_{\tau_\infty}) \setminus \mcl K\right) < \infty \quad \op{or} \quad  \# \left(\mcl V(A_{\tau_\infty}) \setminus (-\mcl K) \right) < \infty .
\eqe  
\end{thm}

In the rest of this section we will prove Theorem~\ref{thm-cone-contain}. We now give an outline of the proof.  

In Section~\ref{sec-alpha>1-D-estimates}, we will prove explicit bounds for various distances with respect to the metric $D$ of~\eqref{eqn-weighted-metric} defined with the above choice of $\alpha$-weight function $f$, via elementary geometric arguments. 
Due to the particular form of $f$, we will be able to obtain such estimates even without knowing the form of the standard FPP limit shape $\BB A$. 
In particular, we will obtain an explicit formula for the distance between the union of the distinguished flat faces $\bdy \mcl Q_s \cap (P_{\BB x} \cup P_{-\BB x})$ and its scaling $q(\bdy \mcl Q_s \cap (P_{\BB x} \cup P_{-\BB x}) )$ for $q>1$; and show that the minimum distance is attained along segments perpendicular to $\BB x$ (Lemma~\ref{prop-convex-dist}). We also prove upper and lower bounds for distances between arbitrary given points of $\mcl Q_s$ and $q\bdy\mcl Q_s$ (Lemma~\ref{prop-D-tube-dist}). 
 
In Section~\ref{sec-alpha>1-prob}, we will use the estimates of Sections~\ref{sec-growth-estimates} and~\ref{sec-alpha>1-D-estimates} to prove estimates for the $f$-weighted FPP clusters $A_t$. 
The most important estimate of Section~\ref{sec-alpha>1-prob} is Lemma~\ref{prop-alpha>1-hit}, which will tell us, roughly speaking, that the following holds. If for some large $n_0 \in\BB N$, the exit position of the clusters $\{A_t\}_{t\geq 0}$ from $B_{n_0}^{\nu_s}(0)$ lies in $n_0 (\bdy \mcl Q_s \cap (P_{\BB x} \cup P_{-\BB x}) )$ (up to rounding error), then it is likely that the following is true for each $n\in\BB N$ a little bit bigger than $n_0$. 
\begin{itemize}
\item The exit position of $\{A_t\}_{t\geq 0}$  from $B_{n }^{\nu_s}(0)$ lies in $n_0 (\bdy \mcl Q_s \cap (P_{\BB x} \cup P_{-\BB x}) )$ (up to rounding error).
\item $A_{\tau_\infty}$ does not contain any vertices of $\BB Z^d$ which lie at $D$-distance greater than a constant times $n^{1-\alpha}$ from $B_n^{\nu_s}(0)$.
\end{itemize}
We will also show that if~\eqref{eqn-as-cond} holds, then for large enough $n_0 \in\BB N$ it is likely that the clusters first exit $B_{n_0}^\nu(0)$ at a point near $n_0 (\bdy \mcl Q_s \cap (P_{\BB x} \cup P_{-\BB x}) )$ (see Lemma~\ref{prop-alpha>1-right-hit}), so the above two conditions are likely to hold for all large enough $n$. 
The proof of these estimates is inductive in nature, and relies crucially on the precise estimates for $D$ in Section~\ref{sec-alpha>1-D-estimates} to control the exit position of the clusters from $B_n^{\nu_s}(0)$. 

In Section~\ref{sec-cone-contain-proof}, we will use the estimates of Section~\ref{sec-alpha>1-prob} to conclude that if $s$ is chosen appropriately, then a.s.\ all but finitely many vertices of $A_{\tau_\infty}$ are contained in $\mcl K \cup (-\mcl K)$, in the notation of Theorem~\ref{thm-cone-contain}. We will then use Proposition~\ref{prop-one-end-weak} to rule out the possibility that there are infinitely many vertices of $A_{\tau_\infty}$ contained in each of $\mcl K$ and $-\mcl K$.

\subsection{Geometric estimates for cylindrical sets}
\label{sec-alpha>1-D-estimates}
 
In this subsection, we will prove some deterministic geometric properties of the metric $D$ associated with a general admissible function $f_s : \bdy \mcl Q_s \rta (0,\infty)$ (Definition~\ref{def-max-cond}). Throughout, we fix $\alpha>1$, $s>1$, and an admissible function $f_s$ and use the notation introduced in Section~\ref{sec-alpha>1-setup} and we let $D$ be as in~\eqref{eqn-weighted-metric} with $f = f_s$. Our main focus is on estimating distances in the metric $D$o that we can eventually apply the results of Section~\ref{sec-growth-estimates} to prove Theorem~\ref{thm-alpha>1}. See Figure~\ref{fig-convex-dist} for an illustration of the key idea of this subsection. 

\begin{figure}[ht!]
 \begin{center}
\includegraphics[scale=.8]{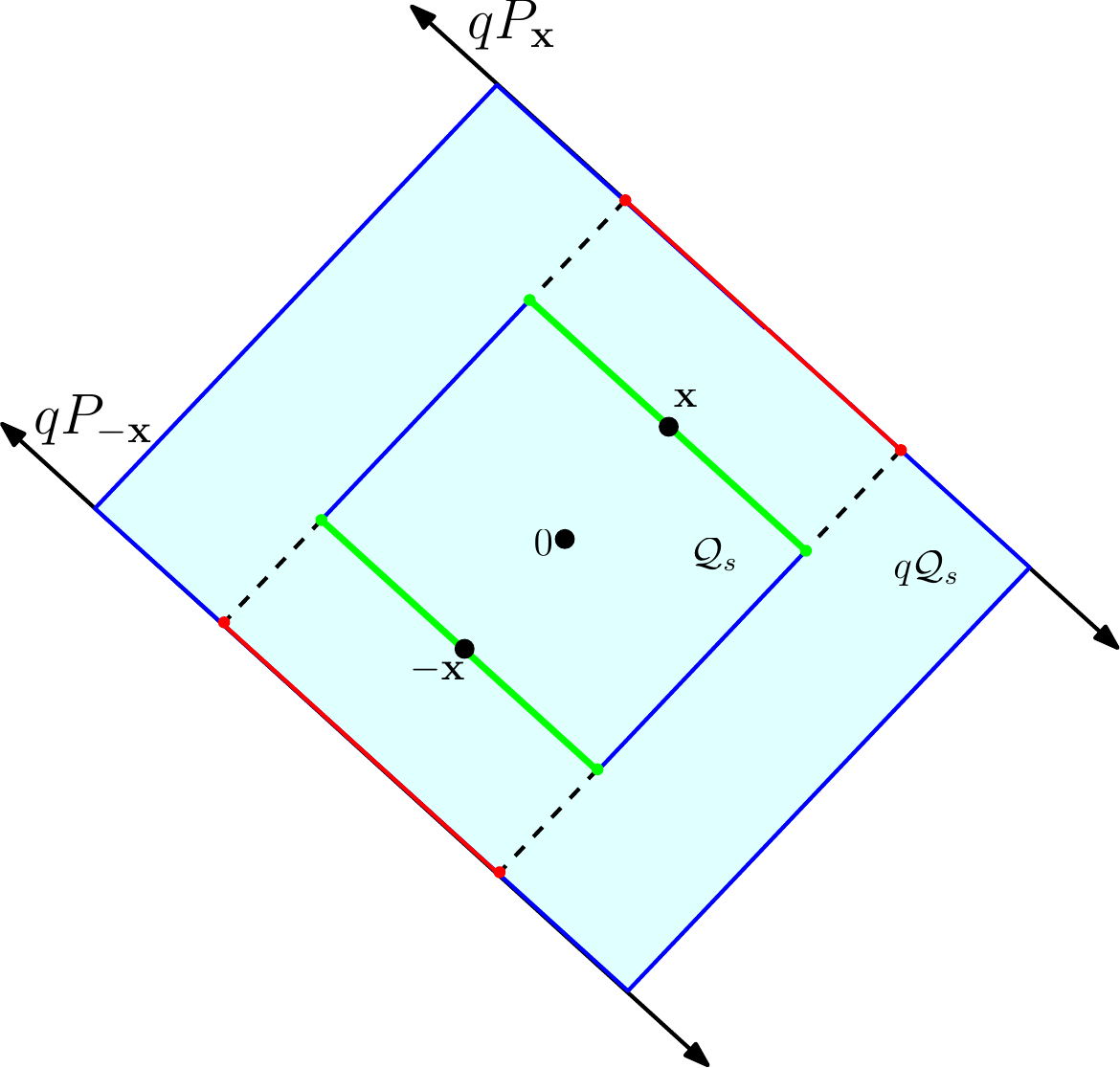} 
\caption{An illustration of the convex sets $\mcl Q_s$ and $q\mcl Q_s$ for $q > 1$. Lemma~\ref{prop-convex-dist} shows that every path of minimal $D$-length from $\bdy \mcl Q_s$ to $q\bdy \mcl Q_s$ lies between one of the two pairs of dotted lines shown in the figure. In particular, every point in $\bdy \mcl Q_s$ which lies at minimal $D$-distance from $q\bdy \mcl Q_s$ belongs to one of the green lines and every point in $q\bdy \mcl Q_s$ which lies at minimal $D$-distance from $\bdy \mcl Q_s$ belongs to one of the red lines. Hence, if we re-scale by $q^{-1}$ so that $q\bdy \mcl Q_s$ is mapped to $\bdy \mcl Q_s$, then the red lines are mapped to proper subsets of the green lines. This means that for each small $\delta>0$, each point of $q\bdy \mcl Q_s$ which lies within $D$-distance at most $D(\bdy \mcl Q_s , q\bdy \mcl Q_s)  +\delta$ from $\bdy \mcl Q_s$ lies at $D$-distance exactly $q^{1-\alpha} D(q \bdy \mcl Q_s , q^2 \bdy \mcl Q_s)$ from $q^2\bdy \mcl Q_s$. 
We also have the following facts, which come from Lemma~\ref{prop-D-tube-dist}. If $s$ is chosen so that~\eqref{eqn-pos-prob-cond} holds, then for large enough $q$ the $D$-distance from $ \mcl Q_s$ to any point in $q\left(\bdy \mcl Q_s \setminus (P_{\BB x} \cup P_{-\BB x}) \right)$ is greater than the $D$-distance from $ \mcl Q_s$ to $\infty$. Furthermore, if $s$ is chosen so that~\eqref{eqn-as-cond} holds, then for an appropriate choice of $q$ the $D$-distance from any point of $\bdy\mcl Q_s$ to $q\left(\bdy \mcl Q_s \setminus (P_{\BB x} \cup P_{-\BB x}) \right)$ is greater than its $D$-distance to $q\left(\bdy \mcl Q_s \cap (P_{\BB x} \cup P_{-\BB x}) \right)$. 
The above observations together with the estimates of Section~\ref{sec-growth-estimates} and an induction argument will be used to prove Theorem~\ref{thm-cone-contain}. 
}\label{fig-convex-dist}
\end{center}
\end{figure}

We start by collecting some basic properties of the set $\mcl Q_s$ and its associated norm $\nu_s$. For the statement, we recall the definition of the constant $\ol\rho$ from Definition~\ref{def-angle-constant}.

\begin{lem} \label{prop-cylinder-basic}
Suppose $\mcl Q_s$ is as in~\eqref{eqn-cylinder-def} and $\nu_s$ is as in~\eqref{eqn-cylinder-norm}. Then the following holds.
\begin{enumerate}
\item $B_{1}^{\mu}(0) \subset B_{ \ol\rho}^{|\cdot|}(0)\subset \mcl Q_s$, so $\mu(w-z) \geq \ol\rho^{-1} |w-z| \geq \nu_s(w-z)$ for each $z,w\in\BB R^d$. \label{item-cylinder-contain} 
\item For each $q>1$,  
\eqb \label{eqn-tube-dist}
\op{dist}^{|\cdot|}\left( \mcl Q_s ,\, q\left( \bdy \mcl Q_s \setminus (P_{\BB x} \cup P_{-\BB x}) \right)      \right) \geq   \ol\rho  s(q-1) .
\eqe
\label{item-tube-dist}
\item For each $q > 1$, we have $\op{dist}^{|\cdot|}(\bdy \mcl Q_s , q\bdy \mcl Q_s) = \ol\rho (q-1)$. \label{item-cylinder-dist} 
\item For each $q>1$, $z\in \bdy \mcl Q_s$, and $w \in q\bdy \mcl Q_s$ with $|w-z| = \ol\rho (q-1)$, we have $z \in P_{\BB x } \cup P_{-\BB x}$ and $w = z \pm (q-1) \BB x$.  \label{item-cylinder-min} 
\end{enumerate}
\end{lem}
\begin{proof}
We first check assertion~\ref{item-cylinder-contain}. 
Suppose $w \in B_{\ol\rho}^{|\cdot|}(0)$ and let $w^\perp $ be its projection onto the plane $P_{\BB x}^0$ through 0 perpendicular to $\BB x$. Then $|w^\perp| \leq |w|$ so $w^\perp \in B_{ \ol\rho}^{|\cdot|}(0)\cap P_{\BB x}^0 \subset Q$. Furthermore, $w-w^\perp = t \BB x$ where $t\in\BB R$ with $|t| = \ol\rho^{-1} |w-w^\perp| \leq \ol\rho$. Since $B_{\ol\rho}^{|\cdot|}(0) \cap P_{\BB x}^0 \subset Q$, it therefore follows from~\eqref{eqn-cylinder-def} that $w = w^\perp + t\BB x \in \mcl Q_s$. By definition of $\ol\rho$ we have $B_1^{\mu}(0)  = \BB A \subset  B_{ \ol\rho}^{|\cdot|}(0)$, and the statement about norms is immediate from~\eqref{eqn-cylinder-norm}. 
  
We next observe that for $q > 1$, each point of $q\bdy Q$ lies at $\nu_s$-distance $q-1$ from $\bdy Q$. Since $B_{ \ol\rho}^{|\cdot|}(0) \cap P_{\BB x}^0 \subset  Q$, each such point lies at Euclidean distance at least $   \ol\rho  (q-1)$ from $ Q$. From~\eqref{eqn-cylinder-def}, we now obtain assertion~\ref{item-tube-dist}.  
Since $|\BB x| = \rho$, we have $\op{dist}^{|\cdot|}\left(P_{\BB x} \cup P_{-\BB x} , q( P_{\BB x} \cup P_{-\BB x})\right) = \ol\rho(q-1)$
and $|q\BB x - \BB x| = \ol\rho(q-1)$. By combining this with assertion~\ref{item-tube-dist}, we obtain assertion~\ref{item-cylinder-dist}. 
 
Now suppose $q > 1$, $z\in \bdy \mcl Q_s$, and $w\in q\bdy Q$ with $|w-z| = \ol\rho (q-1)$. By assertions~\ref{item-tube-dist} and~\ref{item-cylinder-dist} we have $w \in q \left(  P_{\BB x} \cup P_{-\BB x} \right)$. By symmetry we can assume without loss of generality that $w \in q P_{\BB x}$. Any path from $w$ to $\mcl Q_s \setminus P_{\BB x}$ must pass through $P_{\BB x}$, so must have Euclidean length $>\ol\rho(q-1)$. Therefore $z\in P_{\BB x}$. The unique closest point to $z$ in $q P_{\BB x}$ is $z + (q-1) \BB x$, so we must in fact have $w = z + (q-1)\BB x$. 
\end{proof}

Our next lemma generalizes some of the statements of Lemma~\ref{prop-cylinder-basic} to the metric $D$. 

\begin{lem} \label{prop-convex-dist}
For each $q>1$ and each $z \in P_{\BB x} \cup P_{-\BB x}$,  
\eqb \label{eqn-convex-dist}
D\left(\bdy \mcl Q_s , q\bdy \mcl Q_s\right) = D(z, z + (q-1) \BB x) =  \frac{1-q^{1-\alpha}}{\ol\kappa_s(\alpha-1)} .
\eqe 
Furthermore, if $q>1$, $z\in\bdy \mcl Q_s$, and $w\in q\bdy \mcl Q_s$ with $D(z,w) = D(\bdy \mcl Q_s , q\bdy \mcl Q_s)$ then $z\in  P_{\BB x} \cup P_{-\BB x}  $ and $w   = z \pm (q-1)\BB x$. 
\end{lem}
\begin{proof}
First suppose $q>1$ and $z\in P_{\BB x} \cap \bdy \mcl Q_s$. For $t\in [0,q-1]$ let $\ell(t) := z + t \BB x$. Then $\ell$ is a linear path parametrized by $\mu$-length and 
$\nu_s(\ell(t)) = 1 + t $ for each $t\in [0,q-1]$.
Furthermore, for each such $t$ we have $(1+t)^{-1}\ell(t) \in P_{\BB x} \cap\bdy \mcl Q_s$. 

Therefore, $f(\ell(t)) =(1+t)^\alpha \ol\kappa_s$ for each $t\in [0,q-1]$ and
\eqb \label{eqn-convex-dist-upper}
D(\bdy \mcl Q_s , q\bdy \mcl Q_s) \leq D(z, z + (q-1) \BB x) \leq \op{len}^D(\ell) =\int_0^{q-1} (1+t)^{-\alpha} \ol\kappa_s^{-1} \, dt =  \frac{1-q^{1-\alpha}}{\ol\kappa_s(\alpha-1)} .
\eqe 

Now suppose $q>1$, $z\in\bdy \mcl Q_s$, and $w\in q\bdy \mcl Q_s$. Let $\gamma$ be a piecewise linear path from $z$ to $w$, parametrized by $\mu$-length. Let $T := \op{len}^\mu(\gamma)$. 
By assertion~\ref{item-cylinder-contain} of Lemma~\ref{prop-cylinder-basic}, for each $t\in [0,T]$ we have $\op{len}^{\nu_s}(\gamma([0,t])) \leq t$, so $\nu_s(\gamma(t)) \leq 1 + t$. Therefore $f(\gamma(t)) \leq (1+t)^\alpha \ol\kappa_s$ for each $t\in [0,q-1]$, so
\eqb \label{eqn-convex-dist-lower}
\op{len}^D(\gamma) \geq \int_0^T (1+t)^{-\alpha} \ol\kappa_s^{-1} \, dt \geq  \frac{1-(1+T)^{1-\alpha}}{\ol\kappa_s(\alpha-1)}  .
\eqe 
By assertion~\ref{item-cylinder-dist} of Lemma~\ref{prop-cylinder-basic}, $\op{len}^{|\cdot|}(\gamma)  \geq \ol\rho(q-1)$, so $T \geq q-1$. Furthermore, we have strict inequality unless $z\in  P_{\BB x} \cup P_{-\BB x} $ and $w   = z \pm (q-1)\BB x$.
By combining~\eqref{eqn-convex-dist-upper} and~\eqref{eqn-convex-dist-lower}, we obtain the statement of the lemma.
\end{proof}

To complement Lemma~\ref{prop-convex-dist}, we also have a lower bound for the distance from $\mcl Q_s$ to points of $q\bdy \mcl Q_s$ which are not translates of elements of $\bdy \mcl Q_s\cap (P_{\BB x} \cup P_{-\BB x})$ in a direction perpendicular to $\BB x$.

\begin{lem} \label{prop-convex-dist-pos}
For $q >1$, let
\eqbn
K_q := \left\{z + (q-1)\BB x\,:\,  z \in  P_{\BB x} \cap \bdy \mcl Q_s \right\} \cup  \left\{z - (q-1)\BB x\,:\,  z \in  P_{-\BB x} \cap \bdy \mcl Q_s \right\} .
\eqen
For each $q_2 > q_1 > 1$ and $\ep > 0$, there exists $\delta = \delta(f , q_1 , q_2 , \ep) >0$ such that for each $q\in [q_1, q_2]$ and each $w\in q \bdy Q$ with $\op{dist}^{|\cdot|}(w , K_{1,q}(\BB x) ) \geq \ep$,  
\eqbn
D\left(w  , \bdy \mcl Q_s\right) \geq   \frac{1-q^{1-\alpha}}{\ol\kappa_s (\alpha-1) }  + \delta .
\eqen 
\end{lem}
\begin{proof}
Let
\eqbn
K := \bigcup_{q = q_1}^{q_2} q \bdy \mcl Q_s \quad \op{and}\quad K' := \bigcup_{q = q_1}^{q_2} K_q  .
\eqen
Note that $K$ is compact and $K'\subset K$. 
Define $\psi : K \rta (0,\infty)$ by 
\[
\psi(w) =D\left(w  , \bdy \mcl Q_s\right)  -    \frac{1-\nu_s(w)^{1-\alpha}}{\ol\kappa_s (\alpha-1)}  .
\] 
By Lemma~\ref{prop-metric-compare} and since the norms $|\cdot|$ and $\nu_s$ are comparable, we infer that $\psi$ is Euclidean-continuous. 
By Lemma~\ref{prop-convex-dist}, $\psi(w) >0$ for each $w\in K\setminus K'$. By compactness, there is a $\delta>0$ such that $\psi(w) \geq\delta$ for each $w\in K$ with $\op{dist}^{|\cdot|}(w,K') \geq \ep$. If $\ep$ is chosen sufficiently small (depending only on $q_1 , q_2$, and $f$), then whenever $w\in q\bdy \mcl Q_s$ for $q\in [q_1,q_2]$ and $\op{dist}^{|\cdot|}(w,K_{q} ) \geq \ep$, we also have $\op{dist}^{|\cdot|}(w,K') \geq \ep$. The statement of the lemma follows. 
\end{proof}

Our next lemma is the source of the conditions~\eqref{eqn-pos-prob-cond} and~\eqref{eqn-as-cond} in Theorem~\ref{thm-cone-contain}. To state one of the estimates in the lemma, we will need the following notation.
For $r>0$, and $z,w \in B_r^{\nu_s}(0)$, we define a modified version of the metric $D$ by
\eqb \label{eqn-truncated-D}
D_r(z,w) := \inf\left\{\op{len}^D(\gamma) \,:\, \text{$\gamma$ is a piecewise linear path from $z$ to $w$ contained in $B_r^{\nu_s}(0)$}\right\} .
\eqe
We recall that similar modifications of $D$ appear in Lemma~\ref{prop-hull-time-metric} and Remark~\ref{remark-truncated-expand}.
 
\begin{lem} \label{prop-D-tube-dist}
Let $q>1$. 
\begin{enumerate}
\item For each $z\in\bdy \mcl Q_s$, there exists $c \in\BB R$ such that $z+c\BB x \in q (P_{\BB x} \cup P_{-\BB x})$ and
\eqb \label{eqn-D-tube-dist-upper}
D(z , z + c \BB x) \leq  \frac{1-q^{1-\alpha}}{\ol\kappa_s(\alpha-1)} + \frac{1}{\ul\kappa_s}  .
\eqe 
\label{item-D-tube-dist-upper}
\item For each $\zeta>0$, each $z\in  \mcl Q_s $, and each $w\in q\left( \bdy \mcl Q_s \setminus (P_{\BB x} \cup P_{-\BB x}) \right)$, we have (in the notation~\eqref{eqn-truncated-D})
\eqb \label{eqn-D-tube-dist-ball}
D_{q+\zeta}(z,w) \geq  \frac{1-(q+\zeta)^{1-\alpha}}{\ol\kappa_s(\alpha-1)}  +  \frac{(s-1)(q-1 ) - \zeta}{\ol\kappa_s (q+\zeta)^\alpha}  .
\eqe 
\label{item-D-tube-dist-ball}
\item For each $z\in  \mcl Q_s $ and each $w\in q\left( \bdy \mcl Q_s \setminus (P_{\BB x} \cup P_{-\BB x}) \right)$,  
\eqb \label{eqn-D-tube-dist-lower}
D(z,w) \geq \frac{1 + q^{1-\alpha} -   2^{\alpha } ( q+1+s(q-1)       )^{1-\alpha}     }{\ol\kappa_s(\alpha-1)} .
\eqe 
\label{item-D-tube-dist-lower}
\end{enumerate}
\end{lem}
\begin{proof}
First consider the setting of assertion~\ref{item-D-tube-dist-upper}. Assume without loss of generality that $z$ is closer to $ P_{\BB x}$ than $P_{-\BB x}$ in the Euclidean distance and let $\wt c \geq 0$ be chosen so that $z + \wt c \BB x \in P_{\BB x}$. By the definition~\eqref{eqn-cylinder-def} of $\mcl Q_s$, we have $\wt c \leq 1$. Furthermore, for each $t \in [0, \wt c]$ we have $z+t\BB x\in \bdy \mcl Q_s$, so $f(z+t\BB x) \geq \ul\kappa_s$. By integrating along the path $t\mapsto z + t\BB x$ (which is parametrized by $\mu$-length) we obtain $D(z , z+\wt c \BB x) \leq \ul\kappa_s^{-1}$. By Lemma~\ref{prop-convex-dist},
\eqbn
D(z + \wt c \BB x , z + c\BB x) \leq \frac{1-q^{1-\alpha}}{\ol\kappa_s(\alpha-1)} 
\eqen
for $c = \wt c + q-1$. The estimate~\eqref{eqn-D-tube-dist-upper} follows.

Now suppose we are in the setting of assertion~\ref{item-D-tube-dist-ball}. Let $\gamma$ be a piecewise linear path from $z$ to $w$ which is contained in $B_{q+\zeta}^{\nu_s}(0)$ and write $T = \op{len}^\mu(\gamma)$. By assertion~\ref{item-tube-dist} of Lemma~\ref{prop-cylinder-basic}, $T\geq   s(q-1)$. Furthermore, for each $t \in [0,T]$ we have $\nu_s(\gamma(t)) \leq (1+t) \wedge (q+\zeta)$, so 
\eqbn
f(\gamma(t)) \leq 
\begin{cases}
\ol\kappa_s (1+t)^\alpha,\quad &t \in [0,q-1+\zeta] \\
\ol\kappa_s (q+\zeta)^\alpha ,\quad &t \in [q-1+\zeta, T] .
\end{cases}
\eqen
Integrating, we get that $\op{len}^D(\gamma)$ is at least the right side of~\eqref{eqn-D-tube-dist-ball}. 

Finally, we consider the setting of assertion~\ref{item-D-tube-dist-lower}. Let $\gamma : [0,T]\rta\BB R^d$ be a piecewise linear path from $z$ to $w$ parametrized by $\mu$-length. As above, $T \geq  s(q-1)$ and $\nu_s(\gamma(t)) \leq 1+t$ for each $t\in [0,T]$. We can no longer say that $\nu_s(\gamma(t)) \leq q$, since $\gamma$ may not stay in $q\mcl Q_s$. However, $\gamma(T) = q$ so $\nu_s(\gamma(t)) \leq q + T-t$ for each $t\in [0,T]$. Set $r = (q-1+T)/2$. Then 
\alb
\op{len}^D(\gamma) &\geq  \frac{1}{\ol\kappa_s} \int_0^r (1+t)^{-\alpha}  \, dt + \frac{1}{\ol\kappa_s} \int_r^T (q+T-t)^{-\alpha} \, dt \\
&=  \frac{1 + q^{1-\alpha} -   2^{\alpha } (  q+T+1   )^{1-\alpha}     }{\ol\kappa_s(\alpha-1)} \\
&\geq \frac{1 + q^{1-\alpha} -   2^{\alpha } ( q+1+s(q-1)       )^{1-\alpha}     }{\ol\kappa_s(\alpha-1)} . \qedhere
\ale 
\end{proof}

\subsection{Probabilistic estimates}
\label{sec-alpha>1-prob}

Throughout this subsection, we fix $s>1$ and an admissible function $f_s : \bdy\mcl Q_s \rta (0,\infty)$. We continue to use the notation of Section~\ref{sec-alpha>1-setup}. Let $\{A_t\}_{t\geq 0}$ be the $f$-weighted FPP process with $f$ as in~\eqref{eqn-f-convex-decomp} and let $\{\mcl F_t\}_{t\geq 0}$ be the associated filtration as in Definition~\ref{def-fpp}. In what follows we will combine the estimates of Sections~\ref{sec-growth-estimates} and~\ref{sec-alpha>1-D-estimates} to prove some lemmas about the asymptotic behavior of the clusters $\{A_t\}_{t\geq 0}$. These lemmas will be used to prove Theorem~\ref{thm-cone-contain} in the next subsection.  
 
For the results in this subsection, we introduce the following additional notation.
For $r > 0$, let 
\eqbn
\tau_r := \inf\left\{t\geq 0 \,:\, A_t\not\subset B_r^{\nu_s}(0) \right\}  .
\eqen
Also let $u_r$ be the (a.s.\ unique) vertex in $\mcl V(A_{\tau_r}) \setminus r \mcl Q_s$ and let  
\begin{align} \label{eqn-good-hit-event}
G_r :=  \left\{\frac{u_r}{\nu_s(u_r)} \in  \bdy\mcl Q_s \cap  \left(P_{\BB x} \cup P_{-\BB x} \right) \right\}  
\end{align} 
be the event that the exit position of the clusters $A_{\tau_r}$ from $r\mcl Q_s = B_r^{\nu_s}(0)$ is on one of the flat faces of $r\bdy \mcl Q_s$ (i.e.\ the green faces in Figure~\ref{fig-convex-dist}), modulo rounding error. 

We start out with some basic exit time estimates which are similar to estimates from Section~\ref{sec-small-alpha}. 

\begin{lem} \label{prop-alpha>1-time}
Fix $R \geq 2$ and $ \beta \in \left(0 , 1/3\right)$. Suppose $n\in\BB N$ and $\wt n\in \left[(1+R^{-1})n , R n\right]_{\BB Z}$. Then
\eqb \label{eqn-alpha>1-time-lower}
\BB P\left( \tau_{\wt n} - \tau_n  \geq \frac{n^{1-\alpha} - \wt n^{1-\alpha} - n^{1-\alpha- \beta} }{\ol\kappa_s (\alpha-1) } \,|\, \mcl F_{\tau_n} \right) = 1-o_n^\infty(n)  ,
\eqe   
\eqb \label{eqn-alpha>1-time-upper}
\BB P\left( \tau_{\wt n} - \tau_n  \leq \frac{n^{1-\alpha} - \wt n^{1-\alpha} + n^{1-\alpha- \beta} }{\ol\kappa_s (\alpha-1)} \,|\, \mcl F_{\tau_n} \right) \BB 1_{G_n} = (1-o_n^\infty(n)) \BB 1_{G_n} ,
\eqe
and
\eqb \label{eqn-alpha>1-time-no-G}
\BB P\left( \tau_{\wt n} - \tau_n  \leq  \frac{n^{1-\alpha} - \wt n^{1-\alpha} + n^{1-\alpha- \beta} }{\ol\kappa_s (\alpha-1)} + \frac{n^{1-\alpha}}{\ul\kappa_s}    \,|\, \mcl F_{\tau_n} \right)   =  1-o_n^\infty(n)  ,
\eqe
all at a deterministic rate rate independent of the particular choice of $\wt n$.
\end{lem}
\begin{proof} 
The estimate~\eqref{eqn-alpha>1-time-lower} is deduced from Lemma~\ref{prop-hull-expand-metric} in a similar manner to the analogous estimates in Lemmas~\ref{prop-alpha<1-time} and~\ref{prop-alpha-near-1-time}.

To obtain~\eqref{eqn-alpha>1-time-upper}, suppose $G_n$ occurs. The proof of Lemma~\ref{prop-convex-dist} shows that we can find a vertex $\wt u_n \in \BB Z^d\setminus \wt n Q$ and a piecewise linear path $\gamma$ from $u_n$ to $\wt u_n$ contained in $B_{2\wt n}^{\nu_s}(0)\setminus B_{n/2}^{\nu_s}(0)$ which satisfies
\eqb \label{eqn-alpha>1-time-dist}
\op{len}^D(\gamma) \leq   \frac{n^{1-\alpha} - \wt n^{1-\alpha}   }{\kappa_s (\alpha-1)} +    O_n(n^{-\alpha}) .
\eqe 
Indeed, we can take $\gamma$ to be a small perturbation of a path which traces the line segment $[u_n , u_n + (n-\wt n) \BB x]$. 
Therefore, the estimate~\eqref{eqn-alpha>1-time-upper} follows from Lemma~\ref{prop-hull-time-metric}. 
The estimate~\eqref{eqn-alpha>1-time-no-G} is proven in the same manner as~\eqref{eqn-alpha>1-time-upper} but with assertion~\ref{item-D-tube-dist-upper} of Lemma~\ref{prop-D-tube-dist} used in place of Lemma~\ref{prop-convex-dist}. 
\end{proof}

Our next lemma tells us that it is very unlikely that vertices are absorbed by the FPP clusters sooner than we would expect after time $\tau_n$.

\begin{lem} \label{prop-alpha>1-expand}
Fix $R \geq 2$, $ \beta \in \left(0 , 1/3\right)$, and $\zeta>0$. 
Also let $n\in\BB N$, $\wt n \in [(1+R^{-1})n , R n]_{\BB Z}$, and let $v_* \in  \BB Z^d\cap \left(B_{\wt n }^{\nu_s}(0) \setminus B_{\zeta n}^{\nu_s}(0) \right)$ be chosen in a $\mcl F_{\tau_n}$-measurable manner. Then with $D_{(1+\zeta) n}$ as in~\eqref{eqn-truncated-D},  
\eqbn
\BB P\left(  T(0,v_*) - \tau_n \geq  \left(1 - n^{-\beta}\right) D_{(1+\zeta)\wt n}\left(v , \mcl V( A_{\tau_n} ) \cup B_{\zeta n}^{\nu_s}(0)   \right)  - n^{1- \beta-\alpha} \,|\, \mcl F_{\tau_n} \right) =1-o_n^\infty(n)  
\eqen
at a deterministic rate rate depending only on $R$, $\beta$, and $\zeta$. 
\end{lem}
\begin{proof}
Let $v_* \in  \BB Z^d\cap \left(B_{\wt n }^{\nu_s}(\cdot) \setminus B_{\zeta n}^{\nu_s}(0) \right)$ be chosen in a $\mcl F_{\tau_n}$-measurable manner. 
Let $\frk A$ be a possible realization of $A_{\tau_n}$ and let $\frk A'$ be the realization of $A_{\tau_n}$ obtained by adjoining to $\frk A$ each vertex of $\BB Z^d$ which is contained in $B_{\zeta n}^{\nu_s}(0)$ and each edge of $\BB Z^d$ which connects two such vertices.  
By Lemma~\ref{prop-hitting-time-law}, the conditional law of $T(0,v_*) -\tau_n$ given $\{A_{\tau_n} = \frk A\}$ stochastically dominates the conditional law of $T(0,v_*) - \tau_n$ given $\{A_{\tau_n} = \frk A'\}$. Hence to prove~\eqref{eqn-alpha-near-1-time}, it suffices to show that
\eqb \label{eqn-alpha>1-time}
\BB P\left( T(0,v_*) - \tau_n \geq  \left(1 - n^{-\beta}\right) D_{(1+\zeta)\wt n}\left(v , \mcl V( A_{\tau_n} ) \cup B_{\zeta n}^{\nu_s}(0)   \right)  - n^{1-\beta-\alpha}  \,|\, A_{\tau_n} = \frk A' \right) = 1- o_n^\infty(n) ,
\eqe 
uniformly over all possible realizations $\frk A'$.
To see this, suppose $A_{\tau_n} = \frk A'$. Then the FPP geodesic $\eta_{v_*}$ from 0 to $v_*$ satisfies $\eta_{v_*} \setminus \mcl E(\frk A') \subset     B_{2 \wt n}^{\nu_s}(\cdot) \setminus B_{\zeta  n}^{\nu_s}(0) $. The estimate~\eqref{eqn-alpha>1-time} therefore follows from Lemma~\ref{prop-hull-expand-metric} (c.f.\ Remark~\ref{remark-truncated-expand}).
\end{proof}

The following lemma is the key input in the proof of Theorem~\ref{thm-alpha>1}, and will eventually be used to show that the event~\eqref{eqn-as} of Theorem~\ref{thm-cone-contain} is very likely to occur provided $G_{n_0}$ occurs for some large $n_0 \in \BB N$. 

\begin{lem} \label{prop-alpha>1-hit}
Fix $R \geq 3$ and $  \beta \in \left(0 , 1/3\right)$. For $r\geq 0$, define the event $G_r$ as in~\eqref{eqn-good-hit-event}. For each $n_0 \in\BB N$, on the event $G_{n_0}$ it holds except on an event of conditional probability $o_{n_0}^\infty(n_0)$ (at a deterministic rate) given $\mcl F_{\tau_{n_0}}$ that the following is true.
\begin{enumerate}
\item The event $\bigcap_{  n=\lfloor (1+R^{-1} ) n_0 \rfloor}^{\infty} G_{ n} $ occurs. \label{item-alpha>1-hit-G}
\item For each $n \in \BB N$ with $ n \geq n_0$, 
\eqbn
\tau_\infty - \tau_{n} \leq \frac{n^{1-\alpha}}{\ol\kappa_s (\alpha -1)} + n^{1-\alpha- \beta} .
\eqen
 \label{item-alpha>1-time-to-infty}
\item For each $  n \in \BB N$ with $n \geq n_0$ and each $v\in B_{Rn}^{\nu_s}(0) \setminus B_{ n}^{\nu_s}(0) $ with  
\eqbn
D\left( v , A_{\tau_{  n}} \cup B_{ n}^{\nu_s}(0) \right) \geq \frac{ n^{1-\alpha}}{\ol\kappa_s (\alpha-1)} + 2 n^{1-\alpha- \beta}  ,
\eqen
we have $v\notin A_{\tau_\infty}$. \label{item-alpha>1-hit-no}
\end{enumerate}
\end{lem}
\begin{proof}
For $n_0 \in\BB N$ and $n \in \left[(1+R^{-1} ) n_0 , R n_0\right]_{\BB Z}$, write $V_{n_0,  n}$ for the set of $v\in B_{n+\ol\rho d}^{\nu_s}(0) \setminus B_{(1+R^{-1})n_0}^{\nu_s}(0)$ with $v/\nu_s(v) \notin  P_{\BB x} \cup P_{-\BB x} $.
By Lemma~\ref{prop-convex-dist-pos} and scale invariance, we can find $\delta  = \delta(f ,  R ) >0$ such that for each sufficiently large $n_0\in\BB N$, each $n \in \left[(1+R^{-1})n_0 , R n_0\right]_{\BB Z}$, and each $v\in V_{n_0 ,n}$,  
\eqbn
D\left(v , B_{n_0}^{\nu_s}(0)\right) \geq \frac{n_0^{1-\alpha} -n^{1-\alpha}}{\ol\kappa_s (\alpha-1) }  + \delta n_0^{1-\alpha} .
\eqen 
By Lemma~\ref{prop-alpha>1-expand}, it holds except on an event of conditional probability $o_{n_0}^\infty(n_0)$ given $\mcl F_{\tau_{n_0}}$ that
\eqb \label{eqn-bad-hit-pt-dist}
 T(0,v) - \tau_{n_0} \geq \frac{n_0^{1-\alpha} -  n^{1-\alpha}}{\ol\kappa_s (\alpha-1) }  + \delta n_0^{1-\alpha} -  n_0^{1-\alpha-\beta}  ,\quad \forall v\in V_{n_0, n} .
\eqe 
By~\eqref{eqn-alpha>1-time-upper} of Lemma~\ref{prop-alpha>1-time}, if $G_{n_0} $ occurs then except on an event of conditional probability $o_{n_0}^\infty(n_0)$ given $\mcl F_{\tau_{n_0}}$,
\eqbn
\tau_{ n} - \tau_{n_0} \leq  \frac{n_0^{1-\alpha} - n^{1-\alpha}}{\ol\kappa_s (\alpha-1) } + n_0^{1-\alpha- \beta}  ,
\eqen
which is smaller than the right side of~\eqref{eqn-bad-hit-pt-dist} for large enough $n_0$. Hence if $G_{n_0} $ occurs, then except on an event of conditional probability $o_{n_0}^\infty(n_0)$ given $\mcl F_{\tau_{n_0}}$, the unique element $u_{n} \in \mcl V(A_{\tau_{  n}}) \setminus  B_n^{\nu_s}(0)$ does not belong to $V_{n_0,  n}$, i.e.\ $G_{  n} $ occurs. 

By the union bound, we obtain
\eqb \label{eqn-alpha>1-hit-G-finite}
\BB P\left( \bigcap_{n =  \lceil (1+R^{-1}) n_0 \rceil }^{\lfloor R n_0\rfloor} G_{n_0}  \,|\, \mcl F_{\tau_{n_0}} \right) \BB 1_{G_{n_0} } = (1-o_{n_0}^\infty(n_0)) \BB 1_{G_{n_0} } .
\eqe 
Since $R \geq 3$, 
\eqbn
[(1+R^{-1}) n_0 , \infty) \subset \bigcup_{k=1}^\infty \left[(1+R^{-1})^k n_0 , R^k n_0 \right]_{\BB Z} .
\eqen
By~\eqref{eqn-alpha>1-hit-G-finite} and induction, we infer that condition~\ref{item-alpha>1-hit-G} in the statement of the lemma holds with conditional probability $1-o_{n_0}^\infty(n_0)$ given $\mcl F_{\tau_{n_0}}$ on $G_{n_0} $. 

The desired estimate for the probability of the event of condition~\ref{item-alpha>1-time-to-infty} follows from condition~\ref{item-alpha>1-hit-G} together with~\eqref{eqn-alpha>1-time-upper} of Lemma~\ref{prop-alpha>1-time} (the later is applied with $R^k  n $ for $k\in\BB N$ in place of $n$, and then summed over all $k\in\BB N$). We slightly shrink $\beta$ if necessary to allow us to drop an $R,\alpha,\beta$-dependent constant in front of $n^{1-\alpha-\beta}$. 

To estimate the probability of the event of condition~\ref{item-alpha>1-hit-no}, we first apply Lemma~\ref{prop-alpha>1-expand} and the union bound to find that except on an event of conditional probability $o_{n_0}^\infty(n_0)$ given $\mcl F_{\tau_{n_0}}$,  
\eqbn
 T(0,v) - \tau_n \geq    \frac{ n^{1-\alpha}}{\ol\kappa_s(\alpha-1)} + (2 - o_{  n}(1))  n^{1-\alpha- \beta}
\eqen
for each $ n \geq n_0$ and each vertex $v$ as in condition~\ref{item-alpha>1-hit-no} (here we use that $D\leq D_{(1+\zeta)\wt n}$). By combining this with the condition~\ref{item-alpha>1-time-to-infty}, we find that on $G_{n_0}$, the conditional probability given $\mcl F_{\tau_{n_0}}$ of the event in condition~\ref{item-alpha>1-hit-no} is at least $1-o_{n_0}^\infty(n_0)$. 
\end{proof}

In order to deduce Theorem~\ref{thm-cone-contain} from Lemma~\ref{prop-alpha>1-hit}, we need to know that $\BB P(G_{n_0})$ is large for large $n_0$ provided~\eqref{eqn-as-cond} holds. This is the purpose of the next lemma, which plays a role similar to that of Lemma~\ref{prop-alpha<1-small-ball} in the proof of Theorem~\ref{thm-alpha<1}.

\begin{lem} \label{prop-alpha>1-right-hit}
Suppose our parameters are such that~\eqref{eqn-as-cond} from Theorem~\ref{thm-cone-contain} holds.
Then for $n\in\BB N$,
\eqbn
\BB P\left(G_n \right) = 1-o_n^\infty(n) .
\eqen
\end{lem} 
\begin{proof}
Let $q_* := \frac{\alpha}{\alpha-1}$. By~\eqref{eqn-as-cond},
\eqb \label{eqn-optimal-q}
\frac{(s-1)(q_*-1)}{\ol\kappa_s q_*^\alpha} > \frac{1}{\ul\kappa_s}  .
\eqe 
Hence assertion~\ref{item-D-tube-dist-ball} of Lemma~\ref{prop-D-tube-dist} implies that we can find a $\zeta >0$ and a $\delta>0$ such that with $D_{q_*+\zeta}$ as in~\eqref{eqn-truncated-D}, 
\eqb \label{eqn-tube-dist-compare}
D_{q_*(1+\zeta)}\left(z , q_*\left( \mcl Q_s\setminus (P_{\BB x} \cup P_{-\BB x}) \right) \right) \geq \frac{1-q_*^{1-\alpha}}{\ol\kappa_s(\alpha-1)} + \frac{1}{\ul\kappa_s} + \delta ,\quad \forall z \in \bdy \mcl Q_s .
\eqe  

Now suppose given $n\in\BB N$ and set $n' := \lfloor q_*^{-1} n \rfloor$.  
By~\eqref{eqn-alpha>1-time-no-G} of Lemma~\ref{prop-alpha>1-time}, it holds except on an event of conditional probability $1-o_n^\infty(n)$ given $\mcl F_{\tau_{n'}}$ that
\eqb \label{eqn-right-hit-time}
\tau_{\wt n} - \tau_n  \leq  \frac{q_*^{\alpha-1} n^{1-\alpha} - n^{1-\alpha} + n^{1-\alpha- \beta} }{\ol\kappa_s (\alpha-1)} + \frac{q_*^{\alpha-1} n^{1-\alpha}}{\ul\kappa} . 
\eqe 
By~\eqref{eqn-tube-dist-compare} and the scaling property of $D$, 
\begin{align} \label{eqn-right-hit-dist}
&D_{(1+\zeta)n}\left(  n \left(\bdy\mcl Q_s \setminus \left(P_{\BB x} \cup P_{-\BB x} \right)\right) ,\,  \mcl V( A_{\tau_{n'}} ) \cup B_{\zeta n'}^{\nu_s}(0)     \right) \notag \\
&\qquad \geq    \frac{q_*^{\alpha-1} n^{1-\alpha} -n^{1-\alpha} }{\ol\kappa_s(\alpha-1)} + \frac{q_*^{\alpha-1} n^{1-\alpha} }{\ul\kappa_s} + q_*^{\alpha-1} n^{1-\alpha} \delta + O_n(n^{-\alpha}) .
\end{align}
The right side of~\eqref{eqn-right-hit-dist} minus the right side of~\eqref{eqn-right-hit-time} is $\succeq n^{1-\alpha}$ for large enough $n$. 
By Lemma~\ref{prop-alpha>1-expand}, we infer that except on an event of conditional probability $1-o_n^\infty(n)$ given $\mcl F_{\tau_{n'}}$, the vertex $u_n$ lies at $D$-distance $\succeq n^{1-\alpha}$ (and hence Euclidean distance $\succeq n$) from $ n \left(\bdy\mcl Q_s \setminus \left(P_{\BB x} \cup P_{-\BB x} \right)\right)$. Therefore $u_n/\nu_s(u_n) \in P_{\BB x} \cup P_{-\BB x}$, i.e.\ $G_n$ occurs.
\end{proof}

\subsection{Proof of Theorem~\ref{thm-cone-contain}}
\label{sec-cone-contain-proof}

For $n_0 \in \BB N$, let $E_{n_0}$ be the event that the three conditions of Lemma~\ref{prop-alpha>1-hit} are satisfied so that with $G_{n_0}$ as in~\eqref{eqn-good-hit-event}, 
\eqbn
\BB P\left(E_{n_0} \,|\,\mcl F_{\tau_{n_0}} \right) \BB 1_{G_{n_0}} = (1-o_{n_0}^\infty(n_0))\BB 1_{G_{n_0}} .
\eqen
We always have $\BB P(G_{n_0}) >0$ for each $n_0\in\BB N$, so for large enough $n_0 \in\BB N$ we have $\BB P\left(E_{n_0} \right) > 0$.
Furthermore, by Lemma~\ref{prop-alpha>1-right-hit} and the Borel-Cantelli lemma, whenever~\eqref{eqn-as-cond} holds a.s.\ $G_{n_0}$ occurs for large enough $n_0 \in \BB N$. It therefore suffices to prove that if~\eqref{eqn-pos-prob-cond} and $E_{n_0}$ occurs, then a.s.~\eqref{eqn-as} holds (note that we use symmetry between $\mcl K$ and $-\mcl K$ to obtain~\eqref{eqn-pos-prob} in the case that~\eqref{eqn-as-cond} does not hold). 

To this end, we first observe that~\eqref{eqn-pos-prob-cond} implies that there exists $q_0 > 1$ such that for each $q\geq q_0$,
\eqb \label{eqn-no-hit-parameters}
 \frac{1 + q^{1-\alpha} -   2^{\alpha } ( q+1+s(q-1)       )^{1-\alpha}     }{\ol\kappa_s(\alpha-1)} > \frac{1}{\ol\kappa_s (\alpha-1) }.
\eqe 
By Lemma~\ref{prop-metric-scale} and assertion~\ref{item-D-tube-dist-lower} of Lemma~\ref{prop-D-tube-dist}, there is an $n_* \in\BB N$ such that whenever $E_{n_0}$ occurs, $n \geq n_0 \vee n_*$, $q \in [q_0 ,2q_0]$, and $v\in \BB Z^d$ with $\nu_s(v) = q n$ and $v/\nu_s(v) \in \bdy \mcl Q_s \setminus (P_{\BB x} \cup P_{-\BB x})$, 
\eqbn
D\left( v , A_{\tau_{  n}} \cup B_{q^{-1} n}^{\nu_s}(0) \right) \geq D\left(v , B_n^{\nu_s}(0)\right) + O_n(n^{-\alpha}) \geq \frac{ n^{1-\alpha}}{\ol\kappa_s (\alpha-1)} + 2 n^{1-\alpha- \beta} .
\eqen 
By condition~\ref{item-alpha>1-hit-no} of Lemma~\ref{prop-alpha>1-hit}, it follows that if $E_{n_0}$ occurs, then no such $v$ belongs to $\mcl V(A_{\tau_\infty})$. Hence if $E_{n_0}$ occurs, then a.s.\ $\#\left(\mcl V(A_{\tau_\infty}) \setminus (\mcl K\cup (-\mcl K)) \right)  < \infty$
so since $\mcl K$ is convex, a.s.\
\eqb \label{eqn-double-cone}
\# \left\{ e \in \mcl E(A_{\tau_\infty}) \,:\, e \not\subset \mcl K \cup (-\mcl K) \right\} < \infty .
\eqe 

We will now apply Proposition~\ref{prop-one-end-weak} to show that on the event~\eqref{eqn-double-cone}, a.s.\ either all but finitely many edges of $A_{\tau_\infty}$ are contained in $\mcl K$ or all but finitely many edges of $A_{\tau_\infty}$ are contained in $-\mcl K$. 
Let $C > 0$ be chosen so that the graph distance from $\BB Z^d\cap (\mcl K\setminus B_C^{|\cdot|}(0))$ to $\BB Z^d\cap (-\mcl K \cap B_C^{|\cdot|}(0))$ is at least $6$. Let $\Gamma_1$ (resp. $\Gamma_2$) be the largest subgraph of $\BB Z^d$ which is contained in $\mcl K \setminus B_C^{|\cdot|}(0)$ (resp. $-\mcl K \setminus B_C^{|\cdot|}(0)$). Note that the graph distance in $\BB Z^d$ between $\Gamma_1$ and $\Gamma_2$ is at least 3 and that~\eqref{eqn-double-cone} implies
\eqb \label{eqn-double-cone-graph}
\# \left( \mcl E(A_{\tau_\infty}) \setminus   \mcl E(\Gamma_1\cup \Gamma_2)  \right)< \infty .
\eqe
For $k\in\BB N$, let $t^k$ be the smallest $t > 0$ for which $\#\mcl E(A_t) = k$. Almost surely, there is a $k_*\in\BB N$ for which no edge of $\mcl E(A_{\tau_\infty} ) \setminus \mcl E( A_{ t^{k_*}} )$ intersects $B_C^{|\cdot|}(0)$. If $\mcl E(A_{\tau_\infty} ) \setminus \mcl E( A_{ t^{k_*}} )$ is disjoint from either $\mcl E(\Gamma_1)$ or $\mcl E(\Gamma_2)$, then we are done. Otherwise, Proposition~\ref{prop-one-end-weak} applied with $\tau = t^{k }$ for generic $k\in\BB N$ implies that whenever~\eqref{eqn-double-cone-graph} holds, a.s.\ either $\mcl E(A_{\tau_\infty}) \setminus \mcl E(\Gamma_1)$ or $\mcl E(A_{\tau_\infty}) \setminus \mcl E(\Gamma_2)$ is finite, whence~\eqref{eqn-as} holds. 
\qed

\section{Open problems}
\label{sec-open-problems}

Here we list some open problems related to the model considered in this paper. We expect that the solutions to some of these problems may require additional knowledge of the Eden model limiting shape $\BB A$.
 
\begin{enumerate}
\item Under what conditions on $f$ is the limit shape $\BB B = B_1^D(0)$ in Theorem~\ref{thm-alpha<1} convex? Simulations suggest that this is not always the case when $f$ is the $\alpha$th power of a norm; see Figure~\ref{fig-sim2}.
\item It is easy to see from Lemma~\ref{prop-metric-scale} that for any choice of $f$, the limit shape $\BB B$ is compact, contains a neighborhood of 0, and that $\bdy \BB B$ intersects each ray emanating from 0 exactly once. 
If $K\subset \BB R^d$ satisfies these three conditions and has Lipschitz boundary, does there exist an $\alpha<1$ and an $\alpha$-weight function $f$ for which $K = \BB B$? If not, what conditions on $K$ do ensure the existence of such an $f$?
\item Does there exist an $\alpha=1$-weight function $f$ such that the sets $A_t$ a.s.\ converge to a limit shape in the sense of Theorem~\ref{thm-alpha<1}? What if we instead consider convergence of the re-scaled clusters $A_t$ in the Hausdorff distance (which is a weaker mode of convergence than the one in Theorem~\ref{thm-alpha<1})? What if we allow a random limit shape and relax a.s.\ convergence to convergence in law? We refer to the right panel of Figure~\ref{fig-sim4} for a simulation in the case $\alpha=1$. 
\item Give a more general characterization than the one provided in Theorem~\ref{thm-cone-contain} of the set of $\alpha$-weight functions $f$ for which a.s.\ all but finitely many vertices of $A_{\tau_\infty}$ are contained in a cone of opening angle $<\pi$. Simulations suggest that this statement is true in much greater generality than the setting of Theorem~\ref{thm-cone-contain}; see, e.g., Figure~\ref{fig-sim3}. 
\item If $f$ is such that a.s.\ all but finitely many vertices of $A_{\tau_\infty}$ are contained in a cone, what can be said about the law of the opening angle of this cone (as a function of $f$)? 
\item What can be said about the model of Definition~\ref{def-fpp} if instead of exponential edge passage times $X_e$ with parameter $\op{wt}(e)$, we consider a fixed random variable $X$ and take the random variables $X_e$ to be independent each with the law of $\op{wt}(e)^{-1} X$? Note that the results of~\cite{kesten-fpp-speed} do not require exponential passage times, but the proofs in the present paper use the Markov property (Lemma~\ref{prop-fpp-law}) which only works for exponential passage times.  
\end{enumerate}

\bibliography{cibiblong,cibib}

\def\cprime{$'$}
\begin{thebibliography}{FKOV14}

\bibitem[AD14]{auf-dam-fpp-exponent}
A.~Auffinger and M.~Damron.
\newblock A simplified proof of the relation between scaling exponents in
  first-passage percolation.
\newblock {\em Ann. Probab.}, 42(3):1197--1211, 2014, \arxiv{1109.0523}.
  \MR{3189069}

\bibitem[AHD15]{ahd-fpp-survey}
A.~{Auffinger}, J.~{Hanson}, and M.~{Damron}.
\newblock {50 years of first passage percolation}.
\newblock {\em ArXiv e-prints}, November 2015, \arxiv{1511.03262}.

\bibitem[Ale97]{alexander-fpp-speed}
K.~S. Alexander.
\newblock Approximation of subadditive functions and convergence rates in
  limiting-shape results.
\newblock {\em Ann. Probab.}, 25(1):30--55, 1997. \MR{1428498 (98f:60203)}

\bibitem[AS88]{aldous-shields-trees}
D.~Aldous and P.~Shields.
\newblock A diffusion limit for a class of randomly-growing binary trees.
\newblock {\em Probab. Theory Related Fields}, 79(4):509--542, 1988. \MR{966174
  (90k:60052)}

\bibitem[BH91]{bh-eden-sim}
M.~Batchelor and B.~Henry.
\newblock Limits to {E}den growth in two and three dimensions.
\newblock {\em Physics Letters A}, 157(4):229--236, 1991.

\bibitem[{Bla}10]{blair-stahn-fpp-survey}
N.~D. {Blair-Stahn}.
\newblock {First passage percolation and competition models}.
\newblock {\em ArXiv e-prints}, May 2010, \arxiv{1005.0649}.

\bibitem[BPP97]{barlow-dla-on-tree}
M.~T. Barlow, R.~Pemantle, and E.~A. Perkins.
\newblock Diffusion-limited aggregation on a tree.
\newblock {\em Probab. Theory Related Fields}, 107(1):1--60, 1997,
  \arxiv{math/0404089}. \MR{1427716 (97m:60146)}

\bibitem[Bub15]{bubeck-polya-aggregate}
S.~Bubeck.
\newblock {The P\'{o}lya Aggregate}.
\newblock In {\em 2015 Bellairs probability geometry and combinatorics
  workshop}. 2015.

\bibitem[CD81]{cox-durrett-fpp}
J.~T. Cox and R.~Durrett.
\newblock Some limit theorems for percolation processes with necessary and
  sufficient conditions.
\newblock {\em Ann. Probab.}, 9(4):583--603, 1981. \MR{624685 (82k:60208)}

\bibitem[CEG11]{ceg-short-path}
O.~Couronn{\'e}, N.~Enriquez, and L.~Gerin.
\newblock Construction of a short path in high-dimensional first passage
  percolation.
\newblock {\em Electron. Commun. Probab.}, 16:22--28, 2011, \arxiv{1008.5069}.
  \MR{2753301 (2012e:60254)}

\bibitem[Cha13]{chatterjee-fpp-exponent}
S.~Chatterjee.
\newblock The universal relation between scaling exponents in first-passage
  percolation.
\newblock {\em Ann. of Math. (2)}, 177(2):663--697, 2013, \arxiv{1105.4566}.
  \MR{3010809}

\bibitem[ED14]{alm-deijfen-fpp-sim}
S.~{Erick Alm} and M.~{Deijfen}.
\newblock {First passage percolation on $\mathbb{Z}^2$ -- a simulation study}.
\newblock {\em ArXiv e-prints}, December 2014, 1412.5924.

\bibitem[Ede61]{eden}
M.~Eden.
\newblock A two-dimensional growth process.
\newblock In {\em Proc. 4th {B}erkeley {S}ympos. {M}ath. {S}tatist. and
  {P}rob., {V}ol. {IV}}, pages 223--239, Berkeley, Calif., 1961. Univ.
  California Press. \MR{0136460 (24 \#B2493)}

\bibitem[FKOV14]{fko-rumor}
G.~{Fanti}, P.~{Kairouz}, S.~{Oh}, and P.~{Viswanath}.
\newblock {Spy vs. Spy: Rumor Source Obfuscation}.
\newblock {\em ArXiv e-prints}, December 2014, \arxiv{1412.8439}.

\bibitem[FSS85]{fss-eden-sim}
P.~Freche, D.~Stauffer, and H.~Stanley.
\newblock Surface structure and anisotropy of {E}den clusters.
\newblock {\em Journal of {P}hysics {A}: {M}athematical and {G}eneral},
  18(18):761--781, 1985.

\bibitem[GK12]{grimmett-kesten-fpp-survey}
G.~R. Grimmett and H.~Kesten.
\newblock Percolation since {S}aint-{F}lour.
\newblock pages ix--xxvii, 2012. \MR{3014795}

\bibitem[How04]{douglas-fpp-models}
C.~D. Howard.
\newblock Models of first-passage percolation.
\newblock In {\em Probability on discrete structures}, volume 110 of {\em
  Encyclopaedia Math. Sci.}, pages 125--173. Springer, Berlin, 2004.
  \MR{2023652 (2005b:60258)}

\bibitem[HW65]{hammersley-welsh-fpp}
J.~M. Hammersley and D.~J.~A. Welsh.
\newblock First-passage percolation, subadditive processes, stochastic
  networks, and generalized renewal theory.
\newblock In {\em Proc. {I}nternat. {R}es. {S}emin., {S}tatist. {L}ab., {U}niv.
  {C}alifornia, {B}erkeley, {C}alif}, pages 61--110. Springer-Verlag, New York,
  1965. \MR{0198576 (33 \#6731)}

\bibitem[Jan14]{janson-tail}
S.~Janson.
\newblock {Tail bounds for geometric and exponential random variables}.
\newblock {http://www2.math.uu.se/~svante/papers/sjN14.pdf}, 2014.

\bibitem[Kes86]{kesten-fpp-aspects}
H.~Kesten.
\newblock Aspects of first passage percolation.
\newblock In {\em \'{E}cole d'\'et\'e de probabilit\'es de {S}aint-{F}lour,
  {XIV}---1984}, volume 1180 of {\em Lecture Notes in Math.}, pages 125--264.
  Springer, Berlin, 1986. \MR{876084 (88h:60201)}

\bibitem[Kes87]{kesten-fpp-survey}
H.~Kesten.
\newblock Percolation theory and first-passage percolation.
\newblock {\em Ann. Probab.}, 15(4):1231--1271, 1987. \MR{905330 (88g:60246)}

\bibitem[Kes93]{kesten-fpp-speed}
H.~Kesten.
\newblock On the speed of convergence in first-passage percolation.
\newblock {\em Ann. Appl. Probab.}, 3(2):296--338, 1993. \MR{1221154
  (94m:60205)}

\bibitem[KPZ86]{kpz-fluctuation}
M.~Kardar, G.~Parisi, and Y.-C. Zhang.
\newblock {D}ynamic scaling of growing interfaces.
\newblock {\em {Phys. Rev. Lett.}}, 56:889--892, March 1986.

\bibitem[Ric73]{richardson-fpp}
D.~Richardson.
\newblock Random growth in a tessellation.
\newblock {\em Proc. Cambridge Philos. Soc.}, 74:515--528, 1973. \MR{0329079
  (48 \#7421)}

\end{thebibliography}
\bibliographystyle{hmralphaabbrv}

\end{document}